\documentclass[10pt,a4paper]{amsart}
\usepackage{amsmath, amssymb,amsthm}
\usepackage[all]{xy}
\usepackage{enumerate,enumitem}
\usepackage{mathrsfs}
\usepackage{graphicx}
\usepackage{tabularx,multicol}
\usepackage{tikz-cd} 
\usepackage{MnSymbol} 
\usepackage[T1]{fontenc}
\usepackage[hidelinks]{hyperref}

\newtheorem{thm}{Theorem}[section]
\newtheorem{lemma}[thm]{Lemma}
\newtheorem{prop}[thm]{Proposition}
\newtheorem{cor}[thm]{Corollary}

\theoremstyle{definition}
\newtheorem{defn}[thm]{Definition}
\newtheorem{rem}[thm]{Remark}
\newtheorem{theoremletter}{Theorem}

\newcommand{\A}{\mathbf{A}}

\newcommand{\Q}{\mathbf{Q}}
\newcommand{\R}{\mathbf{R}}
\newcommand{\Z}{\mathbf{Z}}
\newcommand{\rG}{\mathrm{G}}
\newcommand{\rH}{\mathrm{H}}
\newcommand{\rI}{\mathrm{I}}
\newcommand{\rM}{\mathrm{M}}
\newcommand{\rN}{\mathrm{N}}
\newcommand{\ga}{\mathfrak{a}}
\newcommand{\gb}{\mathfrak{b}}
\newcommand{\gc}{\mathfrak{c}}
\newcommand{\gd}{\mathfrak{d}}
\newcommand{\gl}{\mathfrak{l}}
\newcommand{\gn}{\mathfrak{n}}
\newcommand{\gm}{\mathfrak{m}}
\newcommand{\go}{\mathfrak{o}}
\newcommand{\gp}{\mathfrak{p}}
\newcommand{\gq}{\mathfrak{q}}
\newcommand{\cA}{\mathcal{A}}
\newcommand{\cB}{\mathcal{B}}
\newcommand{\cC}{\mathcal{C}}
\newcommand{\cD}{\mathcal{D}}
\newcommand{\cE}{\mathcal{E}}
\newcommand{\cF}{\mathcal{F}}
\newcommand{\cG}{\mathcal{G}}
\newcommand{\cH}{\mathcal{H}}
\newcommand{\cI}{\mathcal{I}}
\newcommand{\cJ}{\mathcal{J}}

\newcommand{\cO}{\mathcal{O}}

\newcommand{\cP}{\mathcal{P}}
\newcommand{\cR}{\mathcal{R}}
\newcommand{\cT}{\mathcal{T}}
\newcommand{\cU}{\mathcal{U}}
\newcommand{\cV}{\mathcal{V}}
\newcommand{\cW}{\mathcal{W}}
\newcommand{\cX}{\mathcal{X}}
\newcommand{\sL}{\mathscr{L}}
\newcommand{\sP}{\mathscr{P}}
\DeclareMathOperator{\eis}{eis}
\DeclareMathOperator{\Hom}{Hom}
\DeclareMathOperator{\End}{End}
\DeclareMathOperator{\ad}{ad}

\DeclareMathOperator{\Fil}{Fil}
\DeclareMathOperator{\Gal}{Gal}
\DeclareMathOperator{\ord}{ord}
\DeclareMathOperator{\cusp}{cusp}
\DeclareMathOperator{\full}{full}

\DeclareMathOperator{\GL}{GL}

\DeclareMathOperator{\Frob}{Frob}
\DeclareMathOperator{\cyc}{cyc}
\DeclareMathOperator{\loc}{loc}
\DeclareMathOperator{\rec}{rec}

\DeclareMathOperator{\tr}{tr}

\DeclareMathOperator{\res}{res}
\DeclareMathOperator{\val}{val}
\DeclareMathOperator{\Spec}{Spec}
\DeclareMathOperator{\Spf}{Spf}
\DeclareMathOperator{\nord}{no}
\DeclareMathOperator{\sw}{\mathsf{w}}
\DeclareMathOperator{\univ}{\mathrm{u}}

\newcommand{\Cl}{\mathscr{C}\!\ell}
\newcommand{\dedicationpage}{
  \begin{center}
  \usefont{\encodingdefault}{pzc}{m}{n}
Dedicated to the memory of Jo\"{e}l Bella\"{i}che.    
  \end{center}
}
\begin{document}
\title{Eisenstein points on the Hilbert cuspidal eigenvariety}
\author{Adel Betina, Mladen Dimitrov and Sheng-Chi Shih}
\address{Université Marie et Louis-Pasteur, LMB, 16 route de Gray, 25030 Besançon, France}
\email{adel.betina@univ-fcomte.fr}
\address{University of Lille, CNRS, UMR 8524 -- Laboratoire Paul Painlev\'e, 59000 Lille, France}
\email{mladen.dimitrov@univ-lille.fr}
\address{Faculty of Mathematics, University of Vienna, Oskar-Morgenstern-Platz 1, A-1090 Wien, Austria.}
\email{shengchishih@gmail.com}

\begin{abstract} 
We present a comprehensive study of the geometry of  Hilbert $p$-adic eigenvarieties at classical parallel weight one intersection points of their cuspidal and Eisenstein loci. For instance, we determine all such points at which the weight map is \'etale. 
The Galois theoretic approach presents  genuine difficulties due to the lack of good deformation theory for pseudo-characters irregular at $p$  and  reflects  the richness of the  local geometry. We believe that our geometric results lead to deeper insight into the arithmetic of Hilbert automorphic  forms and we produce in support  several applications in Iwasawa theory. 
\end{abstract}
\maketitle    
\addtocontents{toc}{\setcounter{tocdepth}{0}}
\dedicationpage
\section*{Introduction}
  
Families of $p$-adic cusp forms for $\GL_2$ over $\Q$ were first introduced by Hida in his seminal work \cite{Hida86} later leading to the construction of the eigencurve by Coleman and Mazur \cite{coleman-mazur}. Generalizations to reductive groups of higher rank, called eigenvarieties, were constructed using  overconvergent cohomology of locally symmetric spaces (e.g.~\cite{bellaiche}, \cite{bel-che09}, \cite{urban}) or coherent cohomology of Shimura varieties  (e.g.~\cite{AIPsiegel}, \cite{AIP}). These are  rigid analytic spaces providing  the correct setup for the study of  $p$-adic deformations of automorphic forms. In order to obtain  arithmetic applications, such  as constructing $p$-adic $L$-functions or proving explicit reciprocity laws  for Euler systems, one needs to  perform a meaningful limit process requiring to understand the geometry of the eigenvarieties at the points corresponding to the $p$-stabilizations of the  automorphic forms we are interested in. 
 
While the geometry of an eigenvariety at  points of cohomological weight is well understood  thanks to  classicality results for $p$-adic overconvergent modular forms of small slope \cite{bijakowski17,BPS16, Col96}, the study at  classical  points which are limit of discrete series (such as weight $1$ Hilbert modular forms or weight $(2,2)$ Siegel modular forms) is  more involved (e.g.~\cite{bellaiche-dimitrov}, \cite{bet19},  \cite{BD21}, \cite{BDP})  with richer consequences. For example, the smoothness at such points is a crucial input in the proof of many cases of the Bloch--Kato Conjecture, the Iwasawa Main Conjecture and Perrin-Riou's Conjecture (e.g.~\cite{bertolini-darmon-venerucci}, \cite{BSV22}, \cite{castella-wan}, \cite{LZ22}).
Far more fascinating is the study of the geometry at singular points, especially  at the intersection  between irreducible components 
of the eigenvariety, as such classical points are related to trivial zeros of adjoint $p$-adic $L$-functions (e.g.~\cite{bellaiche},  \cite{BD21}, \cite{BH22}, \cite{DDP}, \cite{DKV}). The aim of this paper is to study the geometry of the Hilbert  eigenvariety at classical intersection points of the cuspidal and the  Eisenstein loci. Such intersections occur at weight $1$ Eisenstein series which are cuspidal as $p$-adic modular forms, very similar to those considered in Ribet's famous proof of the converse to Herbrand's Theorem \cite{Ri76}.

 Let $F$ be a totally real number field of degree $d$ with ring of integers $\go$. Let $\Sigma$ be the set of embeddings of $F$ in  $\overline{\Q}$ and 
 $\Sigma_p$ denote the set of places of $F$ above an odd prime number $p$. 

 Let $\phi_1$ and $\phi_2$ be two finite order Hecke characters of  $F$ such that  $\phi=\phi_1^{-1}\phi_2$ is totally odd
and let $\gn$ denote the prime-to-$p$ part of the level  of  the corresponding weight $1$ Hilbert--Eisenstein series $E_1(\phi_1,\phi_2)$  
whose adelic $q$-expansion is given in  \eqref{eq:fourier_coeff_wt_1_eisen_series}. 
The  eigencurve $\cC$ parameterizing finite slope overconvergent Hilbert eigenforms of parallel weight and tame level $\gn$ is endowed with a locally finite flat map $\sw$ to the $1$-dimensional space $\cW^{\parallel}$ of parallel weights.  It contains a closed analytic subspace $\cC_{\cusp}$ parameterizing  cuspidal eigenforms (see \S\ref{sec:full_eigencurve}). In addition, the Hilbert cuspidal eigenvariety  $\cE$  parameterizing finite slope overconvergent Hilbert cuspidal eigenforms of tame level $\gn$ is endowed with a locally finite and torsion-free map 
$\kappa$  to the $(d+1)$-dimensional weight space $\cW\times \cW^{\parallel}$.  The functoriality properties of eigenvarieties yield the following commutative diagram 
\[\begin{tikzcd}
\cC \arrow{rd}[swap]{\sw} & \cC_{\cusp} \arrow[l, hook']  \arrow[r, hook]  \arrow[d, "\sw"]
& \cE \arrow[d, "\kappa" ] \\ & \cW^{\parallel}  \arrow[r, hook]
& \cW \times \cW^{\parallel}.
\end{tikzcd}
\] 
 
 As we are  interested in  points lying on an Eisenstein component, we have to additionally assume that one of the characters, say $\phi_1$, has conductor relatively prime to $p$. Then  
 the $p$-stabilization $f_{\phi_1}$ of $E_1(\phi_1,\phi_2)$, having $U_{v}$-eigenvalue $\phi_1(v)$ for all $v\in \Sigma_p$, belongs to an Eisenstein family $\cE_{\phi_1,\phi_2}$ whose geometric constant terms are  all multiples  of the imprimitive Deligne--Ribet  $p$-adic $L$-function $\zeta_{\phi_1,\phi_2} \in \overline{\Z}_p \lsem X \rsem$ (see \eqref{eq:imprimitive-zeta} and Proposition~\ref{const_Eisen_family}), and the interpolation formula \eqref{eq:interpolation_p-adic_L-fcn} implies that 
\[\zeta_{\phi_1,\phi_2}(0)= L(0,\phi) \cdot  \prod_{v\in \Sigma_p} (1-\phi(v))
\prod_{\gq|\gn, \gq\nmid \mathrm{cond}(\phi)}\left(1-\phi(\gq)^{-1}\rN_{F/\Q}(\gq)^{-1}\right)
\] 
vanishes if, and only if, the  set $\Sigma_p^{\mathrm{irr}}=\{  \gp \in \Sigma_p \mid \phi(\gp)=1 \}$  of  {\it irregular} primes for  $\phi$  is non-empty. The Gross--Kuzmin Conjecture \cite{kuzmin72} further predicts that  $\mathrm{ord}_{X=0}\zeta_{\phi_1,\phi_2}(X)=\#  \Sigma_p^{\mathrm{irr}}$ and  Gross   \cite{gross81}   proposed  a formula for the expected leading term at $X=0$. The formula was 
proven by Darmon, Dasgupta and Pollack \cite{DDP} when $\#  \Sigma_p^{\mathrm{irr}}=1$, and in general by  Dasgupta, Kakde and Ventullo \cite{DKV}. Henceforth, we will assume that  $\#  \Sigma_p^{\mathrm{irr}} \geqslant 1$, so that $f_{\phi_1}$ defines a point on $\cC_{\cusp}$.  We  note in passing that when $\#  \Sigma_p^{\mathrm{irr}} \geqslant 2$  one can deduce from  \cite[\S3.2]{DKV} that $\cC_{\cusp}$ is singular while  $\kappa:\cE \to \cW \times \cW^{\parallel}$  is ramified at $f_{\phi_1}$ (see Remark~\ref{rem:higher_rank}). 
It is important to observe that if  the restriction $\phi_p$ of the Hecke character $\phi$ to $(\go\otimes \Z_p)^\times$ 
does not  factor through the norm, then $f_{\phi_1}$  does not  define a point on the Andreatta--Iovita--Pilloni eigenvariety  \cite{AIP}, nor on the Kisin--Lai eigencurve \cite{KL}. To palliate this problem, we revisit  the former  construction in  \S\ref{sec:11} and  perform a twisted descent by an arbitrary finite order character of $(\go\otimes \Z_p)^\times$. 
Our main result describes the geometry of the cuspidal eigenvarieties  at $f_{\phi_1}$.  
 
 \begin{theoremletter}\label{main-thm}
Assume that $\Sigma_p^{\mathrm{irr}}=\{ \gp \}$ and  that the Leopoldt defect 
$\delta_{F,p}$ vanishes. Set $d_{\gp}=[F_{\gp}:\Q_p]$.
\begin{enumerate}
\item  If  $d_{\gp}= 1$, then $\cE$ is smooth at $f_{\phi_1}$ and  $\cC_{\cusp}$ is of complete intersections at $f_{\phi_1}$.
\item  If $d_{\gp} = d-1$, then the weight maps $\kappa:\cE \to \cW \times \cW^{\parallel}$ and  $\sw:\cC_{\cusp}\to \cW^{\parallel}$ are 
 \'etale at $f_{\phi_1}$.
 \item If $d_{\gp} = d$, then $\cE$ and $\cC_{\cusp}$  are smooth at  $f_{\phi_1}$. The weight maps $\kappa:\cE \to \cW \times \cW^{\parallel}$ and  $\sw:\cC_{\cusp}\to \cW^{\parallel}$ are  \'etale at $f_{\phi_1}$ if, and only if,  $\sL(\phi) +\sL(\phi^{-1})  \ne 0$ (see \eqref{eq:L-inv} for  definition of the $\mathscr{L}$-invariants).

\item  If $2 \leqslant  d_{\gp} \leqslant  d-2$, then the local ring of any irreducible component of $\cE$ at $f_{\phi_1}$ is not factorial (hence not regular). In particular, $\cE$ and $\cC_{\cusp}$ are not smooth at $f_{\phi_1}$.
\end{enumerate}
\end{theoremletter} 
The  non-vanishing of $\sL(\phi) +\sL(\phi^{-1})$ is established  in Proposition~\ref{non-vanishing} in  the Galois case. 
Let us now briefly describe the techniques involved in the proof. 
Mazur's deformation theory provides a powerful tool to analyze the completed local rings $\cT$ of  eigenvarieties at classical points whose 
Galois representation are irreducible and $p$-regular, as one can then define a trianguline deformation functor which is representable by a ring  surjecting to $\cT$ and whose tangent space computes in terms of Galois cohomology. In the ordinary case, this further leads to modularity lifting 
$\cR=\cT$ theorems allowing precise control of $\cT$ even when the classicality theorems alluded to above are not available (see \cite{Wake2018TheEI} and \cite[\S7.6.3]{bel-che09}).  Some of this can be salvaged for representations which are  reducible but still regular at $p$ (see \cite{bellaiche-chenevier-eis}).
However, for reducible Galois representation which are  irregular at $p$, these approaches are no more viable. 
A tentative approach to palliate to this difficulty  is to instead consider a well tailored deformation functor of an indecomposable representation  having the same semi-simplification. This  was successfully used in \cite{bet19,BDP} for $\GL_2$  over $\Q$ and it generalizes well to the Hilbert modular setting  in case (i) of the above Theorem. The reason is that when $d_{\gp}=1$ the classes unramified away from $\gp$ produced by  Ribet's method lie on a specific line  in the $d$-dimensional cohomology group  $\rH^1(F,\phi^{-1})$. Beyond this case, there is no such canonical line and we have to resort to other methods. 
 
Pseudo-characters and generalized matrix algebras have efficiently been used to study deformations of reducible representations (see for example the book of Bella\"iche--Chenevier \cite{bel-che09}) and our completed local rings $\cT$ carry such objects. 
However, in  the absence of an $\cR=\cT$ theorem, there is no direct description of the tangent space of $\cT$ in Galois theoretic and eventually linear algebra terms. A genuine difficulty is the translation of the ordinarity condition at places in $\Sigma_p^{\mathrm{irr}}$ which is essential to shape the geometry. Ultimately our proof of cases (ii-iv) of Theorem~\ref{main-thm} uses in a delicate mix of Galois theoretic and modular properties of (nearly)-ordinary cuspidal Hida families. 
Our analysis of the geometry turns out to be closely related to the basic question in deformation theory  whether a pseudo-character is the trace of a representation. In cases (ii-iii), we can show that the $\cT$-valued pseudo-character does come from an actual representation, reflecting the fact that the extension group  $\rH^1(F,\phi)$ admits a `modular' parametrization (see \S\ref{sec:43}). In contrast, in case (iv),  $\rH^1(F,\phi)$ lacks canonical lines reflecting the non-existence of $\cT$-valued Galois representation ultimately forcing $\cT$ to be singular.
 
Describing   the geometry of $\cC$ at $f_{\phi_1}$ requires the additional  study of the congruence ideal 
between the Eisenstein and cuspidal families specializing to $f_{\phi_1}$.
For example, when $\sw:\cC_{\cusp}\to \cW^{\parallel}$ is  \'etale at $f_{\phi_1}$, 
we show in Corollary~\ref{TC} that  the complete local ring of $\cC$ at $f_{\phi_1}$ is isomorphic to $\varLambda \times_{\overline{\Q}_p} \varLambda \times_{\overline{\Q}_p} \varLambda$, if  $d_{\gp}=d$, and to $\varLambda \times_{\overline{\Q}_p} \varLambda$, 
if  $d_{\gp}=d-1$,  where $\varLambda\simeq \overline{\Q}_p\lsem X \rsem$ is the completion of $\cO_{\cW^{\parallel},\sw(f_{\phi_1})}$. 
Our next  result determines in full generality, i.e., assuming neither $\#\Sigma_p^{\mathrm{irr}}=1$ nor  $\delta_{F,p}=0$,  the congruence ideal $\cI_{\phi_1,\phi_2} \subset  \overline{\Z}_p \lsem X \rsem$ between the Eisenstein family $\cE_{\phi_1,\phi_2}$  and the cuspidal families (see \S\ref{sec:123}). 
 
\begin{theoremletter}\label{congruencemodule}(Theorem~\ref{congruencemodule-general})
If $\#\Sigma_p^{\mathrm{irr}}\geqslant 1$, then  one has $\displaystyle
\cI_{\phi_1,\phi_2} \cdot \overline{\Z}_p \lsem X \rsem[1/p]=(\zeta_{\phi_1,\phi_2} )$.
\end{theoremletter}

The points of intersection between the Eisenstein components  and the  Hilbert cuspidal eigenvariety correspond to zeros of the Deligne--Ribet $p$-adic $L$-function. In his proof of the  Iwasawa Main Conjecture over totally real fields Wiles
had to treat these  trivial zeros separately  (see \cite[\S11]{wiles90}), as the cohomological classes constructed by Ribet's method do not have the required local property at $p$. 
Such difficulties seem to be inherent to the non-Gorenstein nature of the singularity, as Hida's congruence ideal $\mathbf{C}^0$ and  Mazur's module $\mathbf{C}^1$, forming the bridge between the analytic and arithmetic  sides of Iwasawa's Main Conjecture, differ. 
Geometrically,  Theorem~\ref{congruencemodule} asserts that the  intersection multiplicity  between $\cE_{\phi_1,\phi_2}$ and $\cC_{\cusp}$ at  $f_{\phi_1}$  is given by  the order of vanishing of $\zeta_{\phi_1,\phi_2}$. Our proof crucially uses  the existence of a fundamental exact sequence for the space of Coleman families of parallel weight  in full generality (Theorem~\ref{fund_exact_seq}). To this end, we prove that parallel weight $p$-adic automorphic sheaves on the Hilbert modular varieties introduced in \cite{AIS} extend to  invertible  sheaves on the minimal compactification, allowing  to define geometric $q$-expansions of Coleman families. Our method is very general and has the potential to extend to unitary and symplectic Shimura varieties to study Eisenstein congruences in parallel weight.

The study of the geometry of the Hilbert eigenvariety presented here informed  the fundamental work \cite{DPV} of Darmon--Pozzi--Vonk   computing the Fourier coefficients of the infinitesimal cuspidal deformation of $f_{\phi_1}$ in the anti-parallel direction, and   further relating those  to values of the Dedekind--Rademacher cocycle at real multiplication points. Those values are shown to be Gross--Stark $p$-units in the narrow Hilbert class field of 
$F$ as predicted by the explicit class field theory for real quadratic fields proposed by Dasgupta, Darmon and Vonk \cite{DD06, DV21,DV22}
which can be seen as  an analogue of the singular moduli theory of elliptic units for  imaginary quadratic fields. Their approach 
relies on $p$-adic methods in the theory of automorphic forms and fits in  the emerging $p$-adic Kudla program where real analytic families of Eisenstein series are replaced by $p$-adic families parametrized by the weight. Using a vast generalization of Ribet’s method,  Dasgupta and Kakde \cite{DK21}  significantly refined the approach of  \cite{DKV} to prove the Brumer--Stark conjecture and thus generalized \cite[Thm.~B]{DPV}  to arbitrary totally real fields.  
Theorem~\ref{main-thm}(i) would also have applications to the arithmetic of elliptic curves over $F$ once the 
diagonal cycles method of \cite{BDR-II} becomes available for Shimura curves.

 We close the introduction by  describing two new applications in Iwasawa theory.  
\subsubsection*{CM congruences and Katz $p$-adic $L$-functions at $s=0$}
Suppose that $\phi$ is quadratic, denote by $H$ the corresponding CM quadratic extension of $F$ and by $\Cl_H(p^{\infty}) $  its ray class group of level $p^\infty$.  
Assuming that there is a unique prime $\gp$ of $F$ above $p$, the Eisenstein series $E_1(\mathbf{1},\phi)$,  which in this case is also a  theta series relatively to $H$, has a unique $p$-stabilization  $f_{\mathbf{1}}$.   As $\phi(\gp)=1$, $\gp$ splits  in $H$, hence there exists a nearly-ordinary theta family $\Theta_{\mathbf{1}}$ specializing to $f_{\mathbf{1}}$ in weight $1$. 
By Theorem~\ref{main-thm}(iii),  the congruence module $\mathbf{C}^0(\mathbf{1})$ attached to $\Theta_{\mathbf{1}}$ (see \cite[(6.9)]{HT})  is not supported at the  prime of   $\Z_p \lsem \Cl_H(p^{\infty}) \rsem$ corresponding to the trivial character $\mathbf{1}$, 
 as predicted by a conjecture of  Hida--Tilouine \cite[p.192]{HT}.  This implies that  the Katz \cite{katz}   $p$-adic $L$-function $\zeta_H \in \Z_p \lsem \Cl_H(p^{\infty}) \rsem$ does not vanish at $\mathbf{1}$. 
Recently, Hsieh and the first named author established  a $p$-adic analogue of Kronecker's second limit formula for the value $\zeta_H(\mathbf{1})$ involving the $p$-adic regulator of $F$ (see \cite[Prop.~4.8]{BH22}).

\subsubsection*{Iwasawa's growth for the compositum of all $\Z_p$-extensions}
 Suppose that $\#  \Sigma_p = 2$ and that $\phi$ is not quadratic. Let $H_{\infty}$ denote the compositum of all $\Z_p$-extensions of the  CM field   $H$ cut out by $\phi$, and let   $M_{\infty}$ be the maximal abelian unramified $p$-extension of $H_{\infty}$. Then $Y_{\infty}=\Gal(M_{\infty}/H_{\infty})$ is a  finitely generated torsion module over $\Z_p\lsem \Gal(H_{\infty}/H) \rsem= \Z_p\lsem X_1,\ldots,X_r \rsem$, whose study occupies a central  place in the classical Iwasawa theory. 
  Under the  assumptions of Theorem~\ref{main-thm}(iv)  we   show in  Corollary~\ref{prop_L_infty} that 
$Y_{\infty}/(X_1,\ldots,X_r)  Y_{\infty}$ is infinite.  This is in contrast with the case where $p$ is inert in $H$ and $H_\infty^{\cyc}$  the cyclotomic $\Z_p$-extension (so that $r=1$), as then $Y_{\infty}^{\cyc}/X\cdot Y_{\infty}^{\cyc}$ is finite \cite[Chap.13]{Wa97}. 

\subsubsection*{An algebro-geometric proof of the rank-one Gross--Stark conjecture}
In addition to the above applications, our work provides a new proof of the rank-one Gross--Stark conjecture over totally real fields (Corollary~\ref{Gross--Stark-rk-1}).  
A distinctive advantage of our approach is its natural potential to extend to higher-rank settings thereby offering a framework for studying trivial zeros of more general $p$-adic $L$-functions.

\vspace{-8mm}
\tableofcontents

{ \noindent{\it Acknolwedgements:} { \small
We would like to thank Fabrizio Andreatta, Jo\"{e} Bella\"{i}che, Denis Benois, Henri Darmon, Harald Grobner, Ming-Lun Hsieh, Chi-Yun Hsu, Adrian Iovita, Mahesh Kakde, David Loeffler, Alexandre Maksoud, Vincent Pilloni, Alice Pozzi,  Eric Urban, Jan Vonk and Sarah Zerbes for helpful comments and for many stimulating conversations. Finally, the authors thank the anonymous referees for their careful reading of the manuscript. The first and third named authors acknowledge  support from the START-Prize Y966 of the Austrian Science Fund (FWF).  The second named author was partially supported by  the Agence Nationale de la Recherche grants  ANR-18-CE40-0029 and ANR-16-IDEX-0004. }

\subsection*{General notation}
Given a local or a global  field $L$, we let $\cO_L$ denote its  ring of integers,  $\Sigma_L$ the set of infinite places   and  $\rG_L=\Gal(\overline{L}/L)$ the absolute Galois group. Given a finite field extension $K$ of $L$, we let  $\rN_{K/L}$ denote the relative norm. 
We fix an embedding $\iota_p:\overline{\Q} \to \overline{\Q}_p$  and consider the resulting partition 
$\Sigma=\coprod_{ \gp\in\Sigma_p} \Sigma_{\gp}$, where $\Sigma_{\gp}$ is canonically in bijection with 
$\Hom_{\Q_p-\text{alg}}(F_\gp,\overline{\Q}_p)$.  We let $\log_p:\overline{\Q}_p^\times\to \overline{\Q}_p$ be the $p$-adic logarithm normalized by $\log_p(p)=0$. For simplicity, we will sometimes write $\log_p$ for $\log_p\circ \iota_p$. 
  
Rigid analytic spaces will usually be denoted with calligraphic letters: $\cC$, $\cE$, etc. 
Given a reduced $\Q_p$-affinoid space $\cX$, we let  $\cO_{\cX}^\circ$ is the subring of  power-bounded elements in $\cO_{\cX}$. We endow the $\Q_p$-algebra $\cO(\cX)$ with the coarsest locally convex topology \cite{peter}. For any $x \in \cX$, the local ring $\cO_{\cX,x}$ is a Henselian local topological ring  endowed with the finest locally convex topology such that for any admissible affinoid $\cU$ of $X$ containing $x$, the localization morphism $\cO(\cU) \to \cO_{\cX,x}$ is continuous \cite[Chap.~I]{peter}. 
For any finite type $ \cO_{\cX,x} $-module $M$, we endow $M$ with the quotient locally convex topology associated to a surjection $  \cO_{\cX,x}^n \twoheadrightarrow M$ \cite[p.7]{bellaiche-chenevier-eis}.

\addtocontents{toc}{\setcounter{tocdepth}{2}}

\section{Twisted Coleman families and Hilbert eigenvarieties} \label{sec:11}
Weight $1$ Hilbert modular forms are limits of discrete series and as such  contribute to the coherent cohomology of the Hilbert modular variety but not to its Betti cohomology. In the ordinary case, a geometric definition of $p$-adic Hilbert modular forms is given in 
Hida's book \cite{hida-PAF} and a control theorem is  established under the assumption that $p\geqslant 5$  is unramified in $F$. Rather than extending his results and methods based on the study of the Igusa tower, we follow Andreatta--Iovita--Pilloni \cite{AIP} who  construct families of overconvergent modular sheaves at a fixed finite level. The classicity results are then unconditional (using the work of Bijakowski) and are well tailored for our applications, as inverting $p$ seems crucial to establish the fundamental exact sequence \eqref{eq:fund_seq}.  Furthermore, we  show that the overconvergent modular sheaves in parallel weight extend to invertible sheaves over the minimal compactification.

\subsection{Formal models for Hilbert modular schemes} \label{sec:formal} 
Following \cite{rap78}, \cite{cha90} and \cite{dim04}, 
we introduce the Hilbert modular schemes with $\mu_\gn$-level structure and their  various compactifications.
Let $T=\mathrm{Res}_{\go/\Z}\mathbf{G}_m$ and denote by $\mathscr{T}$  (resp.~$\widehat{\mathbf{G}}_m$)  the formal group obtained by  completing $T$ (resp.~$\mathbf{G}_m$)  along the identity section. Let $\cO$ be the ring of integers of a  finite extension $K$ of $\Q_p$  containing the values of $\phi_p$ and  all the embedding of $F$ into $\overline{\Q}_p$. 
We let $\gd$  denote the different ideal  of  $F$.

Let $\gn$ be an integral ideal of $\go$ relatively prime to $p$ and divisible by a rational integer $\geqslant 4$, and  
  $\mu_\gn$ be the Cartier dual of the constant group scheme $(\go/\gn)$ over  $\Z_p$. We set 
$\Delta_{\gn}=\go_{+}^\times/E_{\gn}^2$, where  $\go_{+}^\times$ denotes the group of totally positive units and $E_{\gn}=\{u\in \go^\times\mid u\equiv 1 \pmod{\gn}\}$.  Given a fractional ideal  $\gc$ of $F$, we consider the functor attaching to any $\Z_p$-scheme $S$, the set of  quadruplets $ (A, \iota, \alpha, \lambda)$ with
 \begin{itemize}
\item $A \rightarrow S$  an abelian scheme of relative dimension $d$ with real multiplication  $\iota:\go \hookrightarrow \End_S (A)$, 
\item $\alpha:\mu_\gn \otimes \gd^{-1} \hookrightarrow A[\gn]$  an $\go$-linear embedding of group schemes \'etale locally over $S$, 
\item $\lambda$  a $\gc$-polarization, i.e.,  an $\go$-linear isomorphism 
$\lambda:A^\vee\xrightarrow{\sim} A\otimes \gc$ satisfying the Deligne--Pappas condition.
\end{itemize}
Under the above assumptions on $\gn$, this functor is representable by a  scheme  $X_\gc$ 
of finite type over $\Z_p$, endowed  with a universal abelian scheme $\pi: A_{\gc} \rightarrow X_{\gc}$. 
Let $\omega $ be the conormal  $\go \otimes  \cO_{X_{\gc}} $-module 
$e^{*}(\Omega_{A_{\gc}/X_{\gc}})$, where $e:X_\gc\rightarrow A_{\gc}$ is the identity section.  Let $\overline{X}_\gc$ be a  toroidal compactification of $X_\gc$ and let $\pi: G_{\gc} \rightarrow \overline{X}_\gc$ be the semi-abelian scheme extending $A_{\gc}$ to $\overline{X}_\gc$. The conormal sheaf $\omega$ extends naturally to an $\go \otimes \cO_{\overline{X}_{\gc}}$-module which is also denoted by $\omega$ and we denote by $\displaystyle X^*_\gc=\mathrm{Proj}\left(\bigoplus_{k\geqslant 0} \rH^0(\overline{X}_{\gc}, \det(\omega)^k )\right)$ the  minimal compactification of $X_\gc$. 
Let $\overline{\mathscr{X}_{\gc}} $ (resp.~$\mathscr{X}^*_{\gc}$) denote the formal completion of $\overline{X}_\gc$ (resp.~$X_\gc^*$) along its special fiber. The natural group action of $\epsilon\in  \go_{+}^\times$ on $X_\gc$ sending $(A, \iota, \alpha, \lambda)$ to 
$(A, \iota, \alpha, \epsilon \cdot \lambda)$  factors through the finite group $\Delta_{\gn}$, because the 
multiplication by $\epsilon\in E_{\gn}$ induces an isomorphism between $(A, \iota, \alpha, \lambda)$ and $(A, \iota, \alpha, \epsilon^2 \cdot \lambda)$. This action extends to  both  $X_{\gc}^*$ and $\overline{X}_{\gc}$. 
All these constructions remain valid over $\Z[1/\rN_{F/\Q}(\gn)]$ ensuring the compatibility with the transcendental 
theory. Given $n\in\Z_{\geqslant 1}$ and  $v \in [0,\frac{1}{4 p^{n-1}}]\cap \Q$,   we let $\overline{\cX}_{\gc}(v)$ (resp.~$\cX^{*}_{\gc}(v)$)  denote the  $v$-overconvergent strict neighborhood of the  ordinary locus $\overline{\cX}_{\gc}(0) \subset \overline{X}_{\gc}^{\mathrm{rig}}$ (resp.~$\cX^{*}_{\gc}(0) \subset X^{*\mathrm{rig}}_{\gc}$),  i.e., the locus where  the troncated $p$-adic valuation of the total Hasse invariant is at most $v$. Under the assumption   $p\geqslant 3$, Fargues \cite[Thm.~7]{fargues} constructed a canonical $\go$-stable subgroup of rank $p^{nd}$ 
\[\mathscr{H}_n \hookrightarrow G_{\gc}[p^n]_{\mid \overline{\cX}_{\gc}(v)}.\] 

 Consider the  \'etale  covering 
\[
\overline{\cX}_{\gc}(\mu_{p^n},v)=\mathrm{Isom}^{\go}_{\overline{\cX}_{\gc}(v)}(\mathscr{H}_n,\mu_{p^n} \otimes \gd^{-1}) \to \overline{\cX}_{\gc}(v). 
\] 
with Galois group $T(\Z/p^{n}\Z)$.  After enlarging $K$ so that  $p^{v}\in K$, following \cite[\S 3]{AIP}, we consider the normal formal model 
$\overline{\mathscr{X}_{\gc}}(v)$ (resp.~$\mathscr{X}^*_{\gc}(v)$)  of $\overline{\cX}_{\gc}(v)$ (resp.~$\cX^{*}_{\gc}(v)$) obtained via normalization of a certain admissible blow-up of $\overline{\mathscr{X}_{\gc}}$ (resp.~$\mathscr{X}^*_{\gc}$)  \cite[\S5.2]{AIPsiegel},  then we define the formal  models $\overline{\mathscr{X}_{\gc}}(\mu_{p^n},v)$ and $\mathscr{X}^*_{\gc}(\mu_{p^n},v)$ by taking normalizations in the generic fiber, and finally we define  $\mathscr{X}_{\gc}(\mu_{p^n},v)$ as the largest open over which $G_{\gc}$ is an abelian scheme. To sum up, we have the following commutative diagram of  formal schemes
\begin{align}\label{eq:diag_formal_schs}
 \xymatrix{
 	\mathscr{X}_{\gc}(\mu_{p^n},v) \ar@{^{(}->}[r] \ar@{^{(}->}[dr]
 	& \overline{\mathscr{X}_{\gc}}(\mu_{p^n},v) \ar^{\overline{h}_n}[r]\ar[d]^{\overline{\pi}_n}& \overline{\mathscr{X}_{\gc}}(v) \ar[d]^{\bar{\pi}} \ar[r]  &  \overline{\mathscr{X}_{\gc}}  \ar[d] \\
 	& \mathscr{X}^*_{\gc}(\mu_{p^n},v) \ar^{h_n}[r] & \mathscr{X}^*_{\gc}(v)  \ar[r] &  \mathscr{X}^*_{\gc},
}
\end{align}
\subsection{The Hodge--Tate map}\label{HT-map}
As $\overline{\mathscr{X}_{\gc}}(v)$ is an admissible normal formal $\cO$-scheme, $\mathscr{H}_n$ extends to a finite flat $\go$-stable subgroup (see \cite[Prop.~4.1.3]{AIPsiegel}) 
\[
\mathscr{H}_n\hookrightarrow G_{\gc}[p^n]_{|\overline{\mathscr{X}_{\gc}}(v)}
\] 
whose generic fiber is  isomorphic to $\go/(p^n)$ for the \'etale topology (see  \cite[Lem.~3.3]{AIP}).
We consider  the conormal sheaf  $\omega_{\mathscr{H}_n}=e^*\left(\Omega_{\mathscr{H}_n/\overline{\mathscr{X}_{\gc}}(v)}\right)$ of $\mathscr{H}_n$ and we recall that the Hodge--Tate map  is a morphism of abelian sheaves over $\overline{\mathscr{X}_{\gc}}(v)$ for the fppf-topology
\[\mathrm{HT}:\mathscr{H}_n^D \to \omega_{\mathscr{H}_n},\]
sending an $S$-point $x \in \mathscr{H}_n^D(S)=\Hom_{\go}(\mathscr{H}_n, \mathbf{G}_m \otimes \gd^{-1})$ to $x^*( \tfrac{dt}{t} \otimes 1) \in \omega_{\mathscr{H}_n}(S)$, in view of the canonical 
isomorphism $\omega_{\mathbf{G}_m \otimes \gd^{-1} \mid S}=\go\otimes \cO_S\cdot  \tfrac{dt}{t}$. 
As  $n-\frac{1}{2} \leqslant n-v \frac{p^n-1}{p-1}$, arguing exactly as  in \cite[Prop.~4.2.1]{AIPsiegel} yields a canonical isomorphism 
\[
\omega\otimes_{\cO_{\overline{\mathscr{X}_{\gc}}(v)}} (\cO_{\overline{\mathscr{X}_{\gc}}(v)}/(p^{n-\frac{1}{2}}))=
\omega_{G_{\gc}[p^n]}\otimes_{\cO_{\overline{\mathscr{X}_{\gc}}(v)}} (\cO_{\overline{\mathscr{X}_{\gc}}(v)}/(p^{n-\frac{1}{2}}))\xrightarrow{\sim} 
\omega_{\mathscr{H}_n}\otimes_{\cO_{\overline{\mathscr{X}_{\gc}}(v)}} (\cO_{\overline{\mathscr{X}_{\gc}}(v)}/(p^{n-\frac{1}{2}}) )
\]
Moreover by \cite[Prop.~6.1]{AIS}, $p^{\frac{v}{p-1}}$ annihilates the cokernel of the linearized  Hodge-Tate map 
\begin{align}\label{HTcokernel}
\overline{\mathrm{HT}}: \mathscr{H}_n^D \otimes \cO_{\overline{\mathscr{X}_{\gc}}(v)}/(p^{n-\frac{1}{2}}) \to \omega\otimes_{\cO_{\overline{\mathscr{X}_{\gc}}(v)}}\cO_{\overline{\mathscr{X}_{\gc}}(v)}/(p^{n-\frac{1}{2}}).
\end{align}
We let $\mathscr{F}_{n}$ be the subsheaf of $\omega$ over  $\overline{\mathscr{X}_{\gc}}(\mu_{p^n},v)$ constructed in \cite[Prop.~3.4]{AIP}. It is locally free of rank $1$ over $\go \otimes \cO_{ \overline{\mathscr{X}_{\gc}}(\mu_{p^n},v)}$ and  uniquely characterized  by the facts that  it contains $p^{\frac{v}{p-1}} \cdot \omega$ and that  the induced Hodge--Tate map yields a  $\go$-linear  isomorphism
\begin{align}\label{HTF}
	\overline{\mathrm{HT}}: \mathscr{H}_n^D \otimes \cO_{\overline{\mathscr{X}_{\gc}}(\mu_{p^n},v)}/(p^{n-\frac{1}{2}}) \xrightarrow{\sim} \mathscr{F}_{n} /p^{n-\frac{1}{2}} \mathscr{F}_{n}. 
\end{align}
\subsection{Andreatta--Iovita--Pilloni--Stevens' overconvergent modular sheaves}\label{AIP-OCMS}
  The  weight space 
\[\cW=\Hom_{\mathrm{cont}}(T(\Z_p),\mathbf{G}_m)\] 
is  the rigid analytic generic fiber of the formal scheme $\Spf(\cO \lsem T(\Z_p) \rsem)$. The group of connected components of the weight space is isomorphic to the group of tamely ramified characters of the corresponding torus, while each connected component  is isomorphic to a rigid open polydisc.  Given an affinoid $\cV$ in $\cW$, the  universal character 
\begin{align}\label{universalcharacter}
\kappa_{\cV}: T(\Z_p)  \longrightarrow \cO_{\cV}^\times
\end{align} 
is $w$-analytic for some $w \in \Q_{>0}$ (see \cite[\S2.2]{AIPsiegel}), which after shrinking $\cV$  can be assumed to be   $n-\frac{1}{2}$.
Let $\mathscr{T}_{n}^\circ \subset \mathscr{T}$ denote the formal subgroup of units congruent to $1$ modulo $p^{n-\frac{1}{2}}$, and let  $\mathscr{T}_{n} \subset \mathscr{T}$ denote the formal subgroup generated by $T(\Z_p)$ and $\mathscr{T}_{n}^\circ$. 
By \cite[p.9]{AIP},   $\kappa_{\cV}$ extends uniquely  to a character 
\begin{align}
 \kappa_{\cV}:\mathscr{T}_{n} \times \mathscr{V} \to \widehat{\mathbf{G}}_m  \times \mathscr{V},
\end{align}
where $\mathscr{V}=\Spf(\cO_{\cV}^\circ)$ is a formal model of $\cV$. 
Henceforth we fix $n \in \Z_{\geqslant 1}$  such that $\phi_p$  factors through $T(\Z/p^{n-1}\Z)$, while keeping 
 the notation from previous subsections.  
Let $\gamma_{n}^{\circ}: \mathscr{I}_{n} \to \overline{\mathscr{X}_{\gc}}(\mu_{p^n},v)$ be the formal $\mathscr{T}_{n}^\circ $-torsor constructed in \cite[\S 3.4]{AIP}. A point $x\in  \mathscr{I}_{n}(R)$ over a normal admissible $\cO$-algebra $R$  corresponds to isomorphism classes of $(A,\iota,\alpha, \lambda, P, \Omega)$, where
\begin{itemize}
\item $P$ is an $\go$-generator of $\mathscr{H}^D_{n}(R)=\mathscr{H}^D_{n}(R\otimes_{\cO} K) \simeq \go/(p^n)$ such that  $(A,\iota,\alpha,\lambda, P)$ corresponds to the point $\gamma_{n}^{\circ}(x) \in \overline{\mathscr{X}_{\gc}}(\mu_{p^n},v)(R)$, and
\item $\Omega$ is an $\go \otimes R$-basis of $\mathscr{F}_{n}(R)$ such that $\overline{\mathrm{HT}}(P)=\Omega \mod p^{n-\frac{1}{2}}$ via \eqref{HTF}.
\end{itemize}
  We set  $\kappa_{\cV}^{\circ}= \kappa_{\cV  \mid  \mathscr{T}_{n}^\circ \times \mathscr{V} }$. 
  Andreatta--Iovita--Pilloni--Stevens \cite{AIP,AIS} defined the $p$-adic automorphic sheaf 
\begin{align}\label{AIPS-sheaf0}
\omega^{\kappa_{\cV}^{\circ}}=\left(\gamma_{n}^{\circ}\right)_*\left(\cO_{ \mathscr{I}_{n}} 
\widehat{\otimes} \cO_{\cV}^\circ\right)[-\kappa_{\cV}^{\circ}]
\end{align}  
which is invertible on $\overline{\mathscr{X}_{\gc}}(\mu_{p^n},v) \times  \mathscr{V}$. In addition, as  $\gamma_{n}^{\circ}$ is equivariant for the action of $t\in T(\Z_p)$ sending 
$(A,\iota,\alpha, \lambda, P, \Omega)$ to $(A,\iota,\alpha, \lambda, t\cdot P,  t\cdot \Omega)$
in a compatible way with the $\mathscr{T}_{n}^\circ$-torsor structure, we obtain a $\mathscr{T}_{n}$-morphism 
\[
\gamma_{n}: \mathscr{I}_{n} \xrightarrow{\gamma_{n}^{\circ}} \overline{\mathscr{X}_{\gc}}(\mu_{p^n},v) \to \overline{\mathscr{X}_{\gc}}(v), 
\]
allowing us to define a  coherent sheaf 
\begin{align}\label{AIPS-sheaf}
\omega^{\kappa_{\cV}}=\left(\gamma_{n}\right)_*
(\cO_{\mathscr{I}_{n}} \widehat{\otimes} \cO_{\cV}^\circ)[-\kappa_{\cV}]
\end{align}  
on $\overline{\mathscr{X}_{\gc}}(v) \times \mathscr{V}$ which is invertible on the generic fibre. There is an inclusion of coherent sheaves $\omega^{\kappa_{\cV}} \hookrightarrow \overline{h}_{n\ast}( \omega^{\kappa_{\cV}^{\circ}})$ over $\overline{\mathscr{X}_{\gc}}(v) \times \mathscr{V}$, where $\omega^{\kappa_{\cV}}$ is exactly the sheaf of invariant sections of $\overline{h}_{n\ast}( \omega^{\kappa_{\cV}^{\circ}})$ under the (twisted) action of $T(\Z/p^{n}\Z)$. This clarifies the relationship between these sheaves.

 The sheaf $\omega^{\kappa_{\cV}}$ can be used to give geometric definition of  Coleman families of overconvergent Hilbert modular forms. Namely, an element $f_{\cV} \in \rH^0(\overline{\cX}_{\gc}(v) \times \cV, \omega^{\kappa_{\cV}})$ 
can be seen as a rule assigning to each normal admissible $\cO_{\cV}^{\circ}$-algebra $R$
and  a tuple $(A,\iota,\alpha, \lambda, P, \Omega) \in \mathscr{I}_{n}(R)$ an  element $f_R(A,\iota,\alpha, \lambda,  P,  \Omega) \in R[1/p]$ such that  for all $ t \in T(\Z_p)$ 
\begin{align}\label{Katzcond}
f_R(A,\iota,\alpha, \lambda, t^{-1} P, t^{-1} \Omega) =\kappa_{\cV}(t) f_R(A,\iota,\alpha, \lambda,  P,  \Omega). 
\end{align}  
\subsection{Classicality for overconvergent modular forms}\label{classicalitycriterion}
For any $k\in \Z[\Sigma]$,  we let  \begin{align}\label{classicalautomorphicsheaf}
\omega^{k}=\bigotimes_{\sigma \in \Sigma} \omega_\sigma^{k_\sigma },\end{align} where 
$\omega=\bigoplus_{\sigma \in \Sigma} \omega_{\sigma}$  is the direct sum decomposition of the classical automorphic sheaf on any Hilbert modular variety over $K$.  We  let $X_{\gc,1}(p^n)$, resp. $X_{\gc,0}(p^n)$,  be  the  Hilbert modular variety over $K$ of level $\Gamma_1^1(\gc,\gn p^n)$, resp.~ $\Gamma_1^1(\gc,\gn ) \cap \Gamma_0(\gc,p^n)$ (see \cite{dim04} for a precise definition). 
As the covering  $\mathrm{pr}: X_{\gc,1}(p^n) \to X_{\gc,0}(p^n)$  is \'etale  with  Galois group $T(\Z/p^{n}\Z)$, there is a natural decomposition 
\begin{align}\label{decompositiontorsor}  
\mathrm{pr}_{*}( \mathrm{pr}^{*}(\omega^{k}))=\bigoplus_{\phi_p \,: \,T(\Z/p^{n}\Z) \to K^\times} \omega^{k}(\phi_p). 
\end{align}
It is unclear whether  $\omega^{k}(\phi_p)$ extends to an invertible sheaf on a toroidal compactification  $\overline{X}_{\gc,0}(p^n)$, as the morphism between the toroidal compactifications $\overline{X}_{\gc,1}(p^n) \to \overline{X}_{\gc,0}(p^n)$ may be ramified at the non-multiplicative cusps.

By  \eqref{AIPS-sheaf0},  the specialization of $\omega^{\kappa_{\cV}^{\circ}}$  at a point of $\cV(K)$ corresponding to an algebraic character $k \in \Z[\Sigma]$ is the invertible sheaf $\left(\gamma_{n}^{\circ}\right)_*(\cO_{\mathscr{I}_{n}})[-k^{\circ}]$ on $\overline{\cX}_{\gc}(\mu_{p^n},v)$ which  is canonically isomorphic to the classical automorphic sheaf  $\omega^{k}$ from \eqref{classicalautomorphicsheaf}. 
The \'etale  covering of rigid spaces $\overline{\cX}_{\gc}(\mu_{p^n},v)\to \overline{\cX}_{\gc}(v)$, defined using  
the canonical subgroup $\mathscr{H}_n$,  allows one  to similarly   decompose the push-forward of 
$\omega^{k}$ into a direct sum of invertible sheaves $\omega^{k}(\phi_p)$ over  $\overline{\cX}_{\gc}(v)$. 
By construction, for any $\phi_p:T(\Z_p) \to T(\Z/p^{n}\Z) \to K^\times$, $\omega^{k}(\phi_p)$ is the 
specialization of  $\omega^{\kappa_{\cV}}$ at the point of $\cV(K)$ corresponding to the locally algebraic character
\[T(\Z_p) \to \cO_K^\times, \quad x  \mapsto \phi_p(x) \cdot \prod_{\sigma \in \Sigma} \sigma(x)^{k_\sigma}. \]
The classicality results of \cite{pilloni-stroh}, \cite{bijakowski} and \cite{bijakowski17} imply that the natural inclusion 
\begin{align}\label{clvsov}
\rH^0(\overline{X}_{\gc,1}(p^n)^{\mathrm{an}}, \omega^{k})   \hookrightarrow \rH^0(\overline{\cX}_{\gc}(\mu_{p^n},v), \omega^{k})  
\end{align}
obtained from  the  restriction along the open immersion $\overline{\cX}_{\gc}(\mu_{p^n},v) \hookrightarrow  \overline{X}_{\gc,1}(p^n)^{\mathrm{an}}$  becomes an isomorphism when restricting to the part where for all $\gp\in \Sigma_p$, the Hecke operator $U_\gp$ has slope  strictly less than $\underset{\sigma \in \Sigma_\gp}{\min}(k_{\sigma})  - d_\gp$. One should be careful to use an $m$-th power of $U_\gp$ so that $\gc \gp^m$ and $\gc$ have the same image in the  narrow class group $\Cl_F^+$ of $F$. 
 Koecher's Principle and the rigid analytic GAGA yield  isomorphisms
\[ 
\rH^0(X_{\gc,0}(p^n),\omega^{k}(\phi_p)) \xrightarrow{\sim}  \rH^0(\overline{X}_{\gc,1}(p^n),\omega^{k})[\phi_p] \xrightarrow{\sim}  
\rH^0(\overline{X}_{\gc,1}(p^n)^{\mathrm{an}}, \omega^{k}) [\phi_p].
\]
Hence  the inclusion $ \rH^0(X_{\gc,0}(p^n),\omega^{k}(\phi_p)) \hookrightarrow \rH^0(\overline{\cX}_{\gc}(v), \omega^{k}(\phi_p)) $ obtained by taking the $\phi_p$-isotypic components  in \eqref{clvsov} restricted to the same small slope part for $U_\gp$ ($\gp\in \Sigma_p$) is an 
 isomorphism as well. 
  
\subsection{Twisted $p$-adic families of arithmetic Hilbert modular forms} \label{sec:twisted}
In order to define Coleman families interpolating eigenforms for $G=\mathrm{Res}_{F/\Q}\mathrm{GL}(2)$
one needs to perform a descent from the fine moduli space to the Shimura variety for $G$ using the action of $\Delta_\gn$, while moving to a $(d+1)$-dimensional weight space. 
Indeed, the algebraic automorphic forms for $G$ have weights $(k, \sw)\in \Z[\Sigma]\times \Z$, 
where for all $\sigma\in \Sigma$ we have $k_\sigma\geqslant 1$ and $k_\sigma\equiv \sw\pmod{2}$. 
Our normalization is such that $\sw$ is the purity weight, i.e., the automorphic form has central character
of the form $\phi \vert \cdot \vert_F^{\sw}$ for some finite order Hecke character $\phi$. 
 As a consequence we  will only consider $p$-adic weights $\kappa$ which are in the image of 
\begin{align} \label{d+1-weights}
 \cW\times \cW^{\parallel} \longrightarrow \cW \quad (\nu,\sw)\mapsto \phi_p\cdot \nu^{2} \cdot \sw \circ \rN_{F/\Q},
\end{align}
where  $\cW^{\parallel}=\Hom_{\mathrm{cont}}(\Z_p^\times,\mathbf{G}_m)$ is the rigid analytic generic fiber of the formal scheme $\Spf(\cO \lsem \Z_p^\times \rsem)$. As parallel  weight satisfy  $k_\sigma=2-\sw$ for all $\sigma\in \Sigma$ (see for example \cite[I.1]{BDJ}), we consider the following embedding of the subspace of parallel weights
\begin{align}\label{eq:rel_wt_sp}
\cW^{\parallel} \to \cW\times \cW^{\parallel}, \enspace \sw\mapsto (\rN_{F/\Q} \cdot\sw\circ \rN^{-1}_{F/\Q}, \sw) . 
\end{align}
We let  $\mathscr{G}_{n}^\circ \subset \widehat{\mathbf{G}}_m$ denote the formal subgroup of units congruent to $1$ modulo $p^{n-\frac{1}{2}}$ and  $\mathscr{G}_{n} \subset \widehat{\mathbf{G}}_m$ the formal subgroup generated by $\Z_p^\times$ and $\mathscr{G}_{n}^\circ$. We recall that $n\in\Z_{\geqslant 1}$ is chosen so that  $\phi_p$ factors through $T(\Z/p^{n-1}\Z)$, and  we fix an affinoid $\cU\subset  \cW\times \cW^{\parallel}$ whose image by 
\eqref{d+1-weights} is contained in an affinoid  
$\cV\subset \cW$  such that $\kappa_\cV$ is $\left(n-\frac{1}{2}\right)$-analytic. 
 By definition, the  push-forward $\kappa_{\cU}:\mathscr{T}_{n} \times \mathscr{U} \to \widehat{\mathbf{G}}_m  \times \mathscr{U}$ of  $\kappa_{\cV}$ via the natural map  $\cO_{\cV}^\circ \to \cO_{\cU}^\circ$ coming from \eqref{d+1-weights} factors through 
  \begin{align}\label{wtd+1character}
 (\nu_{\cU}, \sw_{\cU}):\mathscr{T}_{n} \times \mathscr{G}_{n}  \times\mathscr{U} \to \widehat{\mathbf{G}}_m  \times \mathscr{U},
\end{align}
i.e., $\kappa_{\cU}=\phi_p\cdot \nu_{\cU}^2\cdot \sw_{\cU}\circ\rN_{F/\Q}$. 
Henceforth we assume that $\phi_{p}(\epsilon)=1$ for all $\epsilon\in E_{\gn}$. It is worth noting  that this condition is  satisfied by the restriction to $T(\Z_p)$ of any totally odd finite order Hecke character $\phi:\A_F^\times/F^\times \to \overline{\Q}^\times$ having  tame conductor $\gn$. 
 As $\rN_{F/\Q}(\epsilon)=1$ for all $\epsilon\in E_{\gn}$, applying \eqref{Katzcond} to a global section $f_{\cU} \in \rH^0(\overline{\cX}_{\gc}(v) \times 
 \cU, \omega^{\kappa_{\cU}})$, where  $v \leqslant \frac{1}{4p^{n-1}}$,  for all $\epsilon\in E_{\gn}$ the automorphism $A \xrightarrow{\epsilon\cdot } A$ yields 
\[
\begin{split}
f_{\cU}(A,\iota,\alpha, \lambda, P, \Omega)
&= f_{\cU}(A,\iota,\alpha, \epsilon^2 \lambda, \epsilon P, \epsilon \Omega)
=\kappa_{\cU}(\epsilon^{-1}) f_{\cU}(A,\iota,\alpha, \epsilon^2 \lambda, P, \Omega)\\ 
&= \phi_p(\epsilon^{-1})\sw_{\cU}(\rN_{F/\Q}(\epsilon^{-1})) \cdot \nu_{\cU} (\epsilon^{-2})  f_{\cU}(A,\iota,\alpha, \epsilon^2 \lambda, P, \Omega)\\
&= \nu_{\cU}  (\epsilon^{-2})  f_{\cU}(A,\iota,\alpha, \epsilon^2 \lambda, P, \Omega).
\end{split}
\]
Hence, the following action of $\epsilon\in \go_{+}^\times$ on $\rH^0(\overline{\cX}_{\gc}(v) \times \cU, \omega^{\kappa_{\cU}})$  factors through the finite group $\Delta_{\gn}$  
\begin{align}\label{actionunit}
(\epsilon \cdot f_{\cU})(A,\iota,\alpha, \lambda, P, \Omega)=\nu_{\cU}  (\epsilon^{-1}) \cdot f_{\cU}(A,\iota,\alpha, \epsilon \lambda, P, \Omega). 
\end{align}
Let  $\mathbf{M}^{\dagger}_{\cU}(v)=\bigoplus_{\gc \in \Cl_F^+} \rH^0(\overline{\cX}_{\gc}(v) \times \cU, \omega^{\kappa_{\cU}})^{\Delta_{\gn}}$ denote the Banach $\cO(\cU)$-module of $\Delta_{\gn}$-invariants and let
\begin{align}\label{eq:coleman-families}
\mathbf{M}^{\dagger}_{\cU}= \varinjlim_{v>0} \mathbf{M}^{\dagger}_{\cU}(v)
\end{align}
be the Fr\'echet $\cO(\cU)$-module of Coleman families. The corresponding submodule of cuspidal Coleman families $\mathbf{S}^{\dagger}_{\cU}$ is defined by replacing the sheaf $\omega^{\kappa_{\cU}}$  with $ \omega^{\kappa_{\cU}}(-D)$, where $D$ is the divisor with normal crossings  at infinity of $\overline{\mathscr{X}_{\gc}}(v)$.
 For an affinoid  $\cU$ of  $\cW^\parallel$, we let   $\omega^{\sw_{\cU}}$ denote the sheaf of parallel weight 
$(\rN_{F/\Q} \cdot\sw_{\cU}\circ \rN^{-1}_{F/\Q}, \sw_{\cU})$ and we analogously define  $\mathbf{M}^{\dagger}_{\cU}$ and $\mathbf{S}^{\dagger}_{\cU}$. 

By \cite[\S 3.7]{AIP}  the Hecke correspondences  $T_v$,  for $v \nmid \gn p$, and $U_{v}$, for $v \in \Sigma_p$,  act  continuously on the Banach $\cO(\cU)$-module  $\bigoplus_{\gc \in \Cl_F^+} \rH^0(\overline{\cX}_{\gc}(v) \times \cU, \omega^{\kappa_{\cU}})$.  Moreover they commute  with the action of $\Delta_\gn$, hence also act continuously on the Fr\'echet $\cO(\cU)$-modules  $\mathbf{M}^{\dagger}_{\cU}$ and $\mathbf{S}^{\dagger}_{\cU}$,  generating a commutative algebra of endomorphisms. It is  shown in \cite[Lem.~3.27]{AIP} that $U_p=\prod_{v \in \Sigma_p} U_{v}^{e_{v}} $ is a compact operator.
 
\subsection{Adic version of the formal function theorem}
Recall that we have a proper morphism $\overline{\pi}_n: \overline{\mathscr{X}_{\gc}}(\mu_{p^n},v) \to  \mathscr{X}^*_{\gc}(\mu_{p^n},v)$  of admissible formal $\cO$-schemes.
Let $C$ be a cusp and $\overline{C}=\overline{\pi}_n^{-1}(C)$. Let
\begin{enumerate}
\item[$\bullet$]  $\widehat{\overline{\pi}}_n: \widehat{\overline{\mathscr{X}_{\gc}}}(\mu_{p^n},v)_{\mid \mathscr{V}} \longrightarrow \widehat{\mathscr{X}^*_{\gc}}(\mu_{p^n},v)_{\mid \mathscr{V}}$ be the formal completion of $\overline{\pi}_n$
 along $ \overline{C} \times \mathscr{V} \to C \times \mathscr{V}$.
\item[$\bullet$] $\widehat{\mathrm{R}^n \overline{\pi}_{n \ast} }(\omega^{\kappa_{\cV}^{\circ}})$ be the formal completion of the coherent $\cO_{ \mathscr{X}^*_{\gc}(\mu_{p^n},v)_{\mid \mathscr{V}}}$-module $\mathrm{R}^n \overline{\pi}_{n \ast} (\omega^{\kappa_{\cV}^{\circ}})$ along $C \times \mathscr{V} $ and 
 $\widehat{\omega^{\kappa_{\cV}^{\circ}}}$ be the completion of $\omega^{\kappa_{\cV}^{\circ}}$ along $ \overline{C} \times \mathscr{V}$.
\end{enumerate}
\begin{prop}\label{formalfunctionadmissible}
The canonical  homomorphism
\begin{align}\label{formalcompletion}
 \widehat{\mathrm{R}^n \overline{\pi}_{n \ast} }(\omega^{\kappa_{\cV}^{\circ}}) \longrightarrow \mathrm{R}^n \widehat{\overline{\pi}}_{n \ast} (\widehat{\omega^{\kappa_{\cV}^{\circ}}})\end{align}
is an isomorphism.
\end{prop}
\begin{proof} The statement being  local it  suffices to consider  the case $\pi: \mathscr{X} \longrightarrow \Spf(\widehat{A})$, where
\begin{enumerate}
\item $\Spf(\widehat{A}) \subset \mathscr{X}^*_{\gc}(\mu_{p^n},v)_{\mid \mathscr{V}}$ is an  affine open containing the closed formal subscheme $C \times \mathscr{V} $, defined by an open ideal $\gm_C$, 
\item $\pi: \mathscr{X} \longrightarrow \Spf(\widehat{A})$ is algebraizable by a proper morphism $\pi^{\mathrm{alg}}: X \to \Spec(A)$ of schemes of finite type over $\cO$, and
\item  $\overline{C} \times \mathscr{V} =\pi^{-1}(C \times \mathscr{V} )$ a divisor of $\mathscr{X} $.
\end{enumerate}
 
One has the diagram of our schemes and their formal completions
\begin{align}\label{eq:diag_formal_schsGene}
 \xymatrix@-1.75pc{
 	 & &  \widehat{  \mathscr{X}}^{\overline{C}} \ar[rr]\ar[dd]^{\widehat{\pi}}  	& &  \mathscr{X}  \ar[rr]\ar[dd]^{\pi} \ar[rr]\ar[dd]^{\pi}& & X_{\mid \widehat{A}} \ar[dd]^{\pi^{\mathrm{alg}}_{\mid \widehat{A}}} \ar[rr]  &  &  X  
 	\ar[dd]^{\pi^{\mathrm{alg}}} \\ 
 	& & & \circlearrowleft &  & \circlearrowleft & \\ & & \Spf(\widehat{A}_{\gm_C}) \ar[rr] 
 	& & \Spf(\widehat{A}) \ar[rr] & & \Spec(\widehat{A}) \ar[rr] & & \Spec(A),
}
\end{align}
where $\widehat{  \mathscr{X}}^{\overline{C}}$ is the formal completion of $\mathscr{X} $ along $\overline{C}\times \mathscr{V}$ and  
$\widehat{A}_{\gm_C}$ is the $\gm_C$-adic completion of $\widehat{A}$. 
Note that  $\pi$ is also the $p$-adic completion of the proper morphism of noetherian schemes $\pi^{\mathrm{alg}}_{\mid \widehat{A}}$ which is obtained from $\pi^{\mathrm{alg}}$ by base change. 
By Grothendieck's formal existence theorem  \cite[Thm.~3]{GAGF}, the sheaf $\omega^{\kappa_{\cV}^{\circ}}$ is the $p$-adic completion of a coherent sheaf $\omega^{\kappa_{\cV}^{\circ}}_{\mathrm{alg}}$ on the scheme $X_{\mid \widehat{A} }$.
Nevertheless, since the completion is transitive with a finer topology, $\widehat{\omega^{\kappa_{\cV}^{\circ}}}$ is the completion $\omega^{\kappa_{\cV}^{\circ}}_{\mathrm{alg}}$ with respect to $\overline{C} \times \Spec(\cO( \mathscr{V})/(p))$.
Applying the theorem on formal functions  to $\pi^{\mathrm{alg}}_{\mid \widehat{A}}:X_{\mid  \widehat{A} }  \to \Spec(\widehat{A} )$ with respect to  the ideal $\gm_C \cdot( \widehat{A}  )$ yields 
\[ \widehat{\mathrm{R}^n (\pi^{\mathrm{alg}}}_{\mid \widehat{A} })_{\ast}(\omega^{\kappa_{\cV}^{\circ}}_{\mathrm{alg}}) \simeq \mathrm{R}^n (\widehat{\pi} )_{\ast} (\widehat{\omega^{\kappa_{\cV}^{\circ}}}). \]
The desired assertion then follows from the isomorphism 
\[ \mathrm{R}^n (\pi^{\mathrm{alg}}_{\mid \widehat{A} })_{\ast}(\omega^{\kappa_{\cV}^{\circ}}_{\mathrm{alg}}) \simeq \mathrm{R}^n (\pi)_{\ast}(\omega^{\kappa_{\cV}^{\circ}})\]
resulting again from the theorem on formal functions, this time applied to $\pi^{\mathrm{alg}}_{\mid \widehat{A}}: X_{\mid \widehat{A}  } \to \Spec(\widehat{A}  )$ and the ideal $(p)$.
\end{proof}
   
\subsection{Toroidal compactifications and descent to the minimal compactification}\label{torcom}
Throughout this subsection, we only consider  parallel weights. We will  recall the geometric $q$-expansions of Coleman families at various cusps.
 For any  $U\subset F\otimes_{\Q} \R$, we let $U_+$ denote the subset consisting of totally positive elements.  
By \cite[Def.~2.2]{dim04},  any cusp $C$ of $\mathscr{X}^*_{\gc}(\mu_{p^n},v)$ gives rise to:
\begin{enumerate}
\item  Two fractional ideals $\ga$ and $\gb$ of $F$ and a positive polarization $\beta:\gb^{-1}\ga \simeq \gc$.
\item  A decomposition $\gn=\gn' \gn''$ with  isomorphisms $\gn'' \gb/\gb \simeq \mu_{\gn''} \otimes \gd^{-1}$ and 
 $\ga/p^n\gn' \ga \simeq \go/p^n \gn'  $ (the cusp is unramified at $\gn'$).
\end{enumerate}
We denote by $\sP(\gc,p^n)$ the set of  cusps on $\mathscr{X}^*_{\gc}(\mu_{p^n},v)$ and denote by $\sP(\gc)=\sP(\gc,1)$ the set of cusps on $\mathscr{X}_{\gc}^*(v)$, which can be described in the same manner as above by replacing $n\in \Z_{ \geqslant 1}$ by $0$. Set $M=\ga\gb\gn''{}^{-1}$ and $M^* = \Hom(\ga\gb\gn''{}^{-1},\R)$.  
The construction of the toroidal compactification $\overline{\mathscr{X}_{\gc}}(\mu_{p^n},v)$ of 
$\mathscr{X}_{\gc}(\mu_{p^n},v)$ at the cusp $C \in \sP(\gc,p^n)$ involves the following data (see \cite{rap78})  
\begin{enumerate}[label=\textbf{(C\arabic*)}]
\item  A smooth projective rational polyhedral cone decomposition $\Sigma^{C}$ of $M_{+}^*$ invariant under  $E_{\gn p^n}$ with finitely many orbits. We recall that  $\epsilon \in E_{\gn p^n}$ acts  by multiplication by $\epsilon^2$. 
\item  A toric $\cO$-scheme $S_{\Sigma^C}$  covered by 
$S_{\sigma}=\Spec(\cO[q^\xi; \xi \in M, \sigma(\xi) \geqslant 0 ])$ for $\sigma \in \Sigma^C$, and  
endowed with a natural toric embedding $S^0_{\Sigma^C}=\Spec(\cO[q^\xi]_{\xi \in M}) \hookrightarrow  S_{\sigma}$. 
\item \label{C3} The formal completion of $\overline{\mathscr{X}_{\gc}}(\mu_{p^n},v)$ along the  divisor $\overline{\pi}_n^{-1}(C)$ is isomorphic to  $\widehat{S}_{\Sigma^{C}}/E_{\gn p^n}$, where  $\widehat{S}_{\Sigma^{C}}$  is  the formal completion of $S_{\Sigma^{C}}$ along the special fiber of the divisor  $S_{\Sigma^C} \setminus S^0_{\Sigma^C}$.
\item The completed local ring $\widehat{\cO}_{\mathscr{X}^*_{\gc}(\mu_{p^n},v),C}$ at the  point $C:\Spf(\cO) \to \mathscr{X}^*_{\gc}(\mu_{p^n},v)$ is isomorphic to the $E_{\gn p^n}$-invariant global sections 
$\rH^0(\widehat{S}_{\Sigma^{C}},\cO_{\widehat{S}_{\Sigma^{C}}})^{E_{\gn p^n}}=\cO\lsem q^\xi \rsem_{\xi \in M_{+} \cup \{0\}}^{E_{\gn p^n}}$. 
\item   A Tate object $\left(G_{\gc}\right)_{\mid \widehat{S}_{\Sigma^{C}}}$  uniformized  by the $1$-motive $[\gb \xrightarrow{q} \mathbb{G}_m \otimes \ga^{-1}\gd^{-1}]$. In particular, one has $\omega_{\mid \widehat{S}_{\Sigma^{C}}} \simeq \omega_{\mathbb{G}_m \otimes \ga^{-1}\gd^{-1} }\simeq \ga \otimes \cO_{\widehat{S}_{\Sigma^{C}}} \cdot \frac{dt}{t}$, where $\frac{dt}{t}$ is the canonical invariant differential of $\mathbb{G}_m$. 
\item \label{(C6)} The action of $\epsilon \in E_{\gn p^n}$ induces at the level of the conormal sheaves
\[  
\omega_{\mid \widehat{S}_{\epsilon^2 \cdot \sigma}}  \longrightarrow  \omega_{\mid \widehat{S}_{ \sigma}},  \quad
 \tfrac{dt}{t}  \longmapsto \epsilon \cdot \tfrac{dt}{t}.\]
\end{enumerate}

Recall  the canonical proper  adic morphism $\overline{\pi}_n: \overline{\mathscr{X}_\gc}(\mu_{p^n},v) \to \mathscr{X}^*_\gc(\mu_{p^n},v)$ 
from \eqref{eq:diag_formal_schs}.  Choose an   affinoid $\cU$ in $\cW^{\parallel}$ such that the universal character $\sw_{\mathcal{U}}$ is   
$\left(n-\frac{1}{2}\right)$-analytic and let  $\mathscr{U}=\Spf(\cO(\cU)^\circ)$. 

\begin{prop}\label{trivialtorsor} 
Assume that $\phi_p$ factors through $T(\Z/p^{n-1}\Z)$. The restriction $\mathscr{I}_{n \mid \widehat{S}_{\Sigma^{C}}}$ of the $\mathscr{T}_{n}^\circ$-torsor $\mathscr{I}_{n} \to \overline{\mathscr{X}_{\gc}}(\mu_{p^n},v)$ along $ \widehat{S}_{\Sigma^{C}} \to \overline{\mathscr{X}_{\gc}}(\mu_{p^n},v)$  is the trivial $\mathscr{T}_{n}^\circ$-torsor, i.e.,  $\mathscr{I}_{n \mid \widehat{S}_{\Sigma^{C}}} \simeq \widehat{S}_{\Sigma^{C}} \times \mathscr{T}_{n}^\circ$.  In particular, the direct image $(\overline{\pi}_n)_* (\omega^{\sw_{\cU}^{\circ}})$ is an invertible sheaf on $\mathscr{X}^*_{\gc}(\mu_{p^n},v) \times \mathscr{U}$.
\end{prop}

\begin{proof}
To prove the assertion, it is sufficient to produce a section  of $\mathscr{I}_{n \mid \widehat{S}_{\Sigma^{C}}} \to  \widehat{S}_{\Sigma^{C}}$. The isomorphism $\ga/p^n\gn' \ga \simeq \go/p^n\gn' $ given by the level structure of the cusp $C$ yields  $\omega_{\mid \widehat{S}_{\Sigma^{C}}}= \ga \otimes \cO_{\widehat{S}_{\Sigma^{C}}} \cdot \frac{dt}{t} = \go \otimes \cO_{\widehat{S}_{\Sigma^{C}}} \cdot \frac{dt^{\xi}}{t^{\xi}}$, where $\xi \in \ga$ is such that its image under the identification $\ga/p^n \ga \simeq \go/p^n \go $ equals $1$  and $\frac{dt^{\xi}}{t^\xi}= \xi \otimes 1 \cdot \frac{dt}{t}$. 
 
The canonical subgroup $\mathscr{H}_{n \mid \widehat{S}_{\Sigma^{C}}} \hookrightarrow \left(G_{\gc}\right)_{\mid \widehat{S}_{\Sigma^{C}}}$ is the multiplicative group given by an injection 
\[
\mu_{p^n} \otimes \gd^{-1} \ga^{-1} \hookrightarrow \mathbf{G}_m \otimes \gd^{-1}\ga^{-1}.
\] 
It follows from \cite[Lem.~2.9, \S4.1]{AIS} that we have $\go$-linear maps: 
\begin{equation}\label{HTcusps}
\begin{tikzcd}
 &  
 \mathscr{F}_{n}(\widehat{S}_{\Sigma^{C}})=\omega(\widehat{S}_{\Sigma^{C}})=  \go \otimes \cO(\widehat{S}_{\Sigma^{C}}) \cdot \frac{dt^{\xi}}{t^{\xi}} \arrow{d}{\pmod{p^n}}
 \\
  \mathscr{H}_{n}^D(\widehat{S}_{\Sigma^{C}})=  \ga/p^n \ga \arrow{r} & \omega_{\mathscr{H}_n}( \widehat{S}_{\Sigma^{C}})=\ga  \otimes  \frac{\cO_{\widehat{S}_{\Sigma^{C}}}}{p^n \cO_{\widehat{S}_{\Sigma^{C}}}}(\widehat{S}_{\Sigma^{C}})  \cdot \frac{dt}{t}=\go\otimes  \frac{\cO_{\widehat{S}_{\Sigma^{C}}}}{p^n \cO_{\widehat{S}_{\Sigma^{C}}}}(\widehat{S}_{\Sigma^{C}}) \cdot \frac{dt^{\xi}}{t^{\xi}},
 \end{tikzcd}
\end{equation}
 where the horizontal  Hodge--Tate map sends the  generator  $(\xi\mod{p^n})$ of  $\ga/p^n \ga$ to  $\frac{dt^\xi}{t^\xi} \pmod{p^n}$. According to the universal property of $\mathscr{I}_{n}$, the data $(\left(G_{\gc}\right)_{\mid \widehat{S}_{\Sigma^{C}}},\xi  \in \mathscr{H}_{n}^D(\widehat{S}_{\Sigma^{C}}), \frac{dt^{\xi}}{t^{\xi}})$ gives a section $\widehat{S}_{\Sigma^{C}} \to \mathscr{I}_{n \mid \widehat{S}_{\Sigma^{C}}} $. This implies that $\mathscr{I}_{n \mid \widehat{S}_{\Sigma^{C}}} \simeq \widehat{S}_{\Sigma^{C}} \times \mathscr{T}_{n}^\circ$ and that  $\omega^{\sw_{\cU}^{\circ}}$ can be trivialized on $\widehat{S}_{\Sigma^{C}} \times \mathscr{U}$, namely, 
\begin{align}\label{firsttrivialization}
\omega^{\sw_{\cU}^{\circ}}_{\mid \widehat{S}_{\Sigma^{C}} \times \mathscr{U}} = \cO_{\widehat{S}_{\Sigma^{C}}\times \mathscr{U}}\cdot \sw_{\cU}^{\circ}, 
\end{align}  
where $\sw_{\cU}^{\circ}: \mathscr{T}_{n}^\circ \times \mathscr{U} \to \widehat{\mathbf{G}}_m \times \mathscr{U}$ is defined as the pull back of the $\left(n-\tfrac{1}{2}\right)$-analytic character $\kappa_{\cV}^{\circ}$ from \S\ref{AIP-OCMS} by the composed map of \eqref{d+1-weights} and \eqref{eq:rel_wt_sp}.  As $\sw_{\cU}^{\circ}(\epsilon)=1$ for all $\epsilon \in E_{\gn p^n}$ (because $\sw_{\cU}^{\circ}$ factors through  $\rN_{F/\Q}$), \ref{(C6)}
 yields that the trivialization \eqref{firsttrivialization}  descends to a trivialization on $\widehat{S}_{\Sigma^{C}}/E_{\gn p^n}  \times \mathscr{U}$, namely, 
\begin{align}\label{trivializationmain}
\omega^{\sw_{\cU}^{\circ}}_{\mid \widehat{S}_{\Sigma^{C}/E_{\gn p^n}}\times \mathscr{U}} = 
\cO_{\widehat{S}_{\Sigma^{C}/E_{\gn p^n}}\times \mathscr{U}}\cdot \sw_{\cU}^{\circ}.  
\end{align}
Thus, \ref{C3} and Proposition~\ref{formalfunctionadmissible} yield that the completed stalk of $(\overline{\pi}_n)_* (\omega^{\sw_{\cU}^{\circ}})$ is  free of rank $1$ at any cusp $C$, and hence, $(\overline{\pi}_n)_* (\omega^{\sw_{\cU}^{\circ}})$ is invertible on $\mathscr{X}^*_{\gc}(\mu_{p^n},v)\times \mathscr{U}$.
\end{proof}
\begin{rem} In the non-parallel setting given an affinoid $\cU$ in $\cW\times \cW^{\parallel}$,  the sheaf
  $(\overline{\pi}_n)_* (\omega^{\kappa_{\cU}^{\circ}})$ on $\mathscr{X}^*_{\gc}(\mu_{p^n},v) \times \cU$ 
  is not necessarily  invertible   as the monodromy action of $E_{\gn p^n}$ on $\widehat{S}_{\Sigma^{C}}$ forces the 
constant term of any sections of $\omega^{\kappa_{\cU}^{\circ}}$ to vanish.  
\end{rem}
\subsection{The specialization map for overconvergent $p$-adic modular forms}
The geometric results of the previous subsection have the following important consequence. 
\begin{cor}\label{special-surj} For any  $\sw \in \mathscr{U}(\cO)$,  the natural  specialization map 
$\mathbf{M}^{\dagger}_{\cU}\to  \mathbf{M}^{\dagger}_{\sw}$   is surjective.
\end{cor}
\begin{proof} Consider the adjunction homomorphism for   $n$ such that $\phi_p$ factors through $T(\Z/p^{n-1}\Z)$:
\begin{align}
\nu_n: (\overline{\pi}_n)_* (\omega^{\sw_{\cU}^{\circ}})\otimes_{\cO_{\mathscr{U}},\sw^{\circ}}\cO \longrightarrow 
(\overline{\pi}_n)_* \left(\omega^{\sw_{\cU}^{\circ}}\otimes_{\cO_{ \mathscr{U}},\sw^{\circ}}\cO\right)=(\overline{\pi}_n)_*(\omega^{\sw^{\circ}}).  
\end{align}
Proposition~\ref{trivialtorsor}  asserts that $(\overline{\pi}_n)_* (\omega^{\sw_{\cU}^{\circ}})$ is an invertible sheaf on $\mathscr{X}^*_{\gc}(\mu_{p^n},v) \times \mathscr{U}$, hence its specialization $(\overline{\pi}_n)_* (\omega^{\sw_{\cU}^{\circ}})\otimes_{\cO_{ \mathscr{U}},\sw^{\circ}} \cO$ is an invertible sheaf on $\mathscr{X}^*_{\gc}(\mu_{p^n},v)$. Similarly, $(\overline{\pi}_n)_*(\omega^{\sw^{\circ}})$ is an invertible sheaf on $\mathscr{X}^*_{\gc}(\mu_{p^n},v)$.
 As $\mathscr{X}^*_{\gc}(\mu_{p^n},v)$ is normal and  $\nu_n$ is an isomorphism over the open $\mathscr{X}_{\gc}(\mu_{p^n},v)$ whose complement has codimension at least $2$, it follows that $\nu_n$ is an isomorphism of invertible sheaves. 
 Passing to the generic fiber $\cX^*_{\gc}(\mu_{p^n},v)$ and using the finiteness of $h_n$ from \eqref{eq:diag_formal_schs}, we obtain a canonical isomorphism $\nu: \overline{\pi}_{*} (\omega^{\sw_{\cU}})\otimes_{\cO(\cU),\sw} K \xrightarrow{\sim}
\overline{\pi}_{*}(\omega^{\sw})$ of coherent sheaves on $\cX^*_{\gc} (v)$. As $\cX^*_\gc(v)$ is an affinoid, the 
resulting specialization map  is surjective:  
\[
 \rH^0\left(\cX^{*}_\gc (v)\times \cU, \overline{\pi}_{*}(\omega^{\sw_{\cU}}) \right) \twoheadrightarrow \rH^0(\cX^{*}_\gc (v), \overline{\pi}_{*}( \omega^{\sw}) ). \qedhere
\]
\end{proof}
\subsection{Geometric $q$-expansion, the  residue map and the fundamental exact sequence}\label{sec:19}
Keeping the notation from the previous subsections, we obtain  from \eqref{trivializationmain} and \ref{C3}  a geometric $q$-expansion at a cusp $C \in \sP(\gc,p^n)$
\begin{equation}\label{q-expanv}
\mathrm{Ev}_{C,v}: \rH^0(\overline{\mathscr{X}_\gc}(\mu_{p^n},v)\times \mathscr{U},\omega^{\sw_{\cU}^{\circ}}) \longrightarrow \cO(\mathscr{U}) \lsem q^\xi \rsem_{\xi \in M_{+} \cup \{0\}}^{E_{\gn p^n}}, 
\end{equation}
and we denote its geometric constant term by $a_{C}(0,\cdot)$.  Since all the cusps of $\overline{\mathscr{X}_\gc}(\mu_{p^n},v)$ belong to  the ordinary locus $\overline{\mathscr{X}_\gc}(\mu_{p^n},0) \subset \overline{\mathscr{X}_\gc}(\mu_{p^n},v)$,  one has a commutative diagram
\begin{equation}\label{injectq-exp}
\xymatrix{
\rH^0(\overline{\mathscr{X}_\gc}(\mu_{p^n},v)\times  \mathscr{U},\omega^{\sw_{\cU}^{\circ}})   \otimes_{\Z_p} \Q_p    \ar@{->}[dr]_{\mathrm{Ev}_{C,v}} \ar@{^{(}->}[r]&  \rH^0(\overline{\mathscr{X}_\gc}(\mu_{p^n},0)\times \mathscr{U},\omega^{\sw_{\cU}^{\circ}}) \otimes_{\Z_p} \Q_p   
\ar@{->}[d]^{\mathrm{Ev}_{C,0}}   \\
& \cO(\mathscr{U}) \lsem q^\xi \rsem_{\xi \in M_{+} \cup \{0\}}^{E_{\gn p^n}} \otimes_{\Z_p} \Q_p,  }  
\end{equation}
where the  horizontal map corresponds  to the restriction along the open immersion $\overline{\cX_\gc}(\mu_{p^n},0) \hookrightarrow \overline{\cX_\gc}(\mu_{p^n},v)$ at the level of the generic fiber. By the irreducibility of the Igusa tower \cite{Hida09,ribet1975p}, the special fiber of $\overline{\mathscr{X}_\gc}(\mu_{p^n},0) $ is geometrically irreducible and the map \eqref{q-expanv} is injective for  $v=0$; it then follows from the diagram \eqref{injectq-exp} that injectivity also holds for $v>0$ after inverting $p$. Combining these observations with \eqref{AIPS-sheaf}, we obtain an injective map
\[
\mathrm{Ev}_{C}: \rH^0(\overline{\cX}_{\gc}(v) \times \cU ,\omega^{\sw_{\cU}}) \hookrightarrow  \rH^0(\overline{\mathscr{X}_\gc}(\mu_{p^n},v)\times \mathscr{U},\omega^{\sw_{\cU}^{\circ}}) \otimes_{\Z_p} \Q_p \hookrightarrow \cO(\mathscr{U}) \lsem q^\xi \rsem_{\xi \in M_{+} \cup \{0\}}^{E_{\gn p^n}}\otimes_{\Z_p} \Q_p.  
\]
Note that these Fourier coefficients are unnormalized, while we will use the normalized Fourier coefficient in next subsection. We refer the reader to \cite[\S2]{DDP} or \cite[\S2]{Shih} for their relationship.  

Recall that for $n\in \Z_{ \geqslant 1}$, we denote by $\sP(\gc,p^n)$ the set of cusps of the formal scheme $\mathscr{X}^*_{\gc}(\mu_{p^n},v)$. 
 Set $\sP_1(\gc,p^n)=\sP(\gc,p^n)/\Delta_{\gn}$ and $\cO(\cU)[\sP_1(p^n)]=\bigoplus_{\gc\in \Cl_F^+} \cO(\cU)[\sP_1(\gc,p^n)]$. We have $\cO(\cU)[\sP_1(\gc,p^n)]=\cO(\cU)[\sP(\gc,p^n)]^{\Delta_{\gn}}$. By the definition of the space $\mathbf{M}^{\dagger}_{\cU}$ and the above discussion, we obtain a morphism 
\begin{align}\label{geom_res}
\res: \mathbf{M}^{\dagger}_{\cU} \to \cO(\cU)[\sP_1(p^n)]^{(\phi_p)}, \qquad \cF \mapsto \sum_{C\in \sP_1(p^n)} a_{C}(0,\cF) \cdot [C], 
\end{align}
where $\cO(\cU)[\sP_1(p^n)]^{(\phi_p)}$ is   the subgroup of $\cO(\cU)[\sP_1(p^n)]$  on which $T(\Z/p^{n}\Z)$  acts via $\phi_p$. 

\begin{thm}\label{fund_exact_seq} One has a Hecke equivariant exact sequence of $\cO(\cU)$-modules 
\begin{align}\label{eq:fund_seq}
0\to \mathbf{S}^{\dagger}_{\cU} \to \mathbf{M}^{\dagger}_{\cU} \xrightarrow{\res} \cO(\cU)[\sP_1(p^n)]^{(\phi_p)} \to 0.
\end{align}
\end{thm}
\begin{proof} 
Using the notation of diagram \eqref{eq:diag_formal_schs}, it has been shown in \cite[Cor.~3.20]{AIP}  that 
\[\mathrm{R}^1 (h_n\circ \overline{\pi}_n)_{*}(\omega^{\sw_{\cU}^{\circ}}(-D)) =0.\]
As the rigid analytic space $\cX^*_\gc(v)$ is an affinoid, it follows from \cite[\S6.3]{Bosch14}, that  
\begin{align}\label{vanishingh1} 
\rH^1(\overline{\cX}_\gc(\mu_{p^n},v)\times \cU,\omega^{\sw_{\cU}^{\circ}}(-D))=
\mathrm{R}^1 (h_n\circ \overline{\pi}_n)_{*}(\omega^{\sw_{\cU}^{\circ}}(-D))(\cX^*_{\gc}(v)\times \cU)=
0.
\end{align}
We have an exact sequence 
\begin{align}\label{exactsequbound} 
0 \to \cO_{\overline{\mathscr{X}_\gc}(\mu_{p^n},v)\times \mathscr{U}} (-D) \to \cO_{\overline{\mathscr{X}_\gc}(\mu_{p^n},v)\times \mathscr{U}} \to (\mathbf{j} \times \mathbf{1})_{*}(\cO_{D \times \mathscr{U}}) \to 0, 
\end{align}
where $\mathbf{j}: D \hookrightarrow \overline{\mathscr{X}_\gc}(\mu_{p^n},v) $ is the closed immersion of the 
divisor with normal crossings at infinity. 
The following exact sequence is obtained by taking the tensor product of \eqref{exactsequbound} with the invertible sheaf $\omega^{\sw_{\cU}^{\circ}}$ and the projection formula:
 \begin{align}\label{exactsequbound2} 0 \to \omega^{\sw_{\cU}^{\circ}} (-D) \to \omega^{\sw_{\cU}^{\circ}} \to (\mathbf{j} \times \mathbf{1})_{*}\left( (\mathbf{j} \times \mathbf{1})^{*}(\omega^{\sw_{\cU}^{\circ}} ) \right) \to 0. \end{align}
On the other hand, it follows  from \eqref{trivializationmain} that  \begin{align}\label{trivialatbundaryformal}
(\mathbf{j} \times \mathbf{1})^{*}(\omega^{\sw_{\cU}^{\circ}} ) \simeq \cO_{D \times \mathscr{U}}. \end{align}
Finally, \eqref{vanishingh1},  \eqref{exactsequbound2} and \eqref{trivialatbundaryformal}  yield the following short exact sequence (we take the generic fiber)
\begin{align}\label{fondexact1}
\begin{split}
0 \to \rH^0(\overline{\cX}_\gc(\mu_{p^n},v)  \times \cU,\omega^{\sw_{\cU}^{\circ}}(-D)) &\to \rH^0(\overline{\cX}_\gc(\mu_{p^n},v)  \times \cU,\omega^{\sw_{\cU}^{\circ}})\\ &\to \rH^0(\overline{\cX}_\gc(\mu_{p^n},v)  \times \cU, (\mathbf{j} \times \mathbf{1})_{*}(\cO_{D \times \cU  }) ) \to 0. 
\end{split}
\end{align}
Since $D$ is proper over $\Spf(\cO)$ and $\mathbf{j}$ is affine, it follows that 
\begin{align}\label{bord}
\begin{split}
\rH^0(\overline{\cX}_\gc(\mu_{p^n},v) \times \cU, (\mathbf{j} \times \mathbf{1})_{*}(\cO_{D \times \cU  }) )
&=\rH^0(D \times \cU, \cO_{D \times \cU})\\
&= \rH^0(D, \cO_{D}) \otimes  \cO(\cU)   = \cO(\cU)[\sP(\gc,p^n)].
\end{split}
\end{align} 
 The desired short exact sequence \eqref{eq:fund_seq} is obtained from \eqref{fondexact1} and \eqref{bord} by taking a direct sum over 
$\Cl_F^+$, then a  direct limit on $v$, and finally  $\Delta_{\gn}$-invariants.
\end{proof}

\section{The Eisenstein congruence ideal}\label{Congruenceidealsection}
The main goal of this section is to prove Theorem~\ref{congruencemodule}. The proof presented in  \S\ref{sec:123} mixes three ingredients. The first is the surjectivity of the geometric residue map established in \S\ref{sec:19}. 
The second is the duality \eqref{eq:cuspidal-duality} between the cuspidal Hecke algebra and the space of cuspidal Hida  families based on the fact, 
established in \S\ref{sec:full_eigencurve}, that  the eigencurves $\cC$ and $\cC^{\full}$ are locally isomorphic at $f$. The third is a formula for the  constant terms of  the Eisenstein families presented in \S\ref{sec:122}.

 Along the way, we extend the existing  construction \cite{AIP} of  an eigenvariety for $\GL_{2/F}$ by introducing a twist by  a finite order character $\phi_p$ of  $(\go\otimes \Z_p)^\times$. The novelty is best illustrated when restricting to parallel weights, where the construction in {\it loc. cit.} boils down exactly to the cuspidal locus of the Kisin--Lai eigencurve \cite{KL}, whereas our eigencurve has no dense subset of crystalline points and also interpolates  $p$-ordinary Eisenstein series. Although rather formal, this extension is crucial for our global applications unless the level of the weight $1$ Eisenstein series is relatively prime to  $p$.

\subsection{Eisenstein families}\label{sec:eisenstein_series_def}

In this section, we  recall the definitions of weight $1$ Eisenstein series and  Eisenstein families
with coefficients in the  Iwasawa algebra $\varLambda_{\cO}=\cO \lsem X \rsem$.  
We fix a  topological generator $\gamma$  of the Galois group $\Gal(F_{\infty}/F)$, where  $F_{\infty}$ is the cyclotomic $\Z_p$-extension of $F$, and we let $u\in 1+p\Z_p$ be such that $\gamma\cdot \zeta=\zeta^u$ for all $\zeta\in \mu_{p^{\infty}}$. For any integer $k\in \Z_{\geqslant 1}$, the weight $k$ specialization $\varLambda_{\cO}\to \cO$ map sends $X$ to $(u^{k-1}-1)$. 

Let $\phi_1$ and $\phi_2$ be two primitive finite order Hecke characters of $F$ having conductors $\gn_1$ and $\gn_2$, respectively, such that  $\phi=\phi_1^{-1}\phi_2$ is totally odd and $\phi_1$ has conductor relatively prime to $p$. By enlarging the field $K$, we may assume that its  ring of integers $\cO$ contains all the values of $\phi_1$ and $\phi_2$.
We assume in addition that $\phi\not\equiv\omega_p^{-1}\pmod{\varpi}$,  with $\varpi$  a uniformizer of $\cO$, which is automatically satisfied if 
$\Sigma_p^{\mathrm{irr}}\ne\varnothing$.  The weight $1$ Eisenstein series  $E_1(\phi_1,\phi_2)$ has 
 level $\gn_1\gn_2$ and central character $\phi_1\phi_2$ which we will see as a (possibly imprimitive) 
 character of modulus $\gn_1\gn_2$. Let  $\delta(\phi_i)=1$, if $\gn_i=\go$, and  $0$ otherwise ($i=1,2$).
 Following \cite[Prop.~2.1]{DDP}, the adelic Fourier coefficients of   $E_1(\phi_1,\phi_2)$   are given by 
\begin{align}\label{eq:fourier_coeff_wt_1_eisen_series}
C(\gb,E_1(\phi_1,\phi_2)) & =\sum_{\ga|\gb} \phi_1(\tfrac{\gb}{\ga})\phi_2(\ga), 
\text{ for } (0)\ne \gb\subset \go, \text{ and} \\
C((0),E_1(\phi_1,\phi_2)) & =2^{-d}\sum_{[\gc]\in\Cl_F^+}\left(\delta(\phi_1)\phi_1^{-1}(\gc\gd)L(0,\phi)+\delta(\phi_2)\phi_2^{-1}(\gc\gd)L(0,\phi^{-1})\right)[\gc]. 
\end{align}

As $\gn_1$ is  relatively prime to $p$, the $p$-stabilization $f_{\phi_1}$ of $E_1(\phi_1,\phi_2)$, defined recursively for all $v\in \Sigma_p$ relatively prime to $\gn_2$ as
\[ E_1(\phi_1,\phi_2)-\phi_2(v) \left(\begin{smallmatrix} \varpi_v & 0 \\ 0 & 1 \end{smallmatrix}\right)\cdot E_1(\phi_1,\phi_2),\]
is a $p$-ordinary modular form.  Following \cite[\S1.3]{wiles} there exists a  $\varLambda_{\cO}$-adic Eisenstein series $\cE_{\phi_1,\phi_2}$  having adelic Fourier coefficients  given by
\begin{align}  \label{Eisensteinnonconstanterms}
C(\gb,\cE_{\phi_1,\phi_2}) & =\sum_{\ga|\gb, (p)+\ga=\go} \phi_1\left(\tfrac{\gb}{\ga}\right) \phi_2(\ga)(1+X)^{\frac{\log_p(\rN_{F/\Q}(\ga))}{\log_p(u)}}, 
\text{ for } (0)\ne \gb\subset \go, \text{ and} \\ \label{Eisensteinconstanterms}
C((0),\cE_{\phi_1,\phi_2}) & =2^{-d} \zeta_{\phi_1,\phi_2}(X)\sum_{[\gc]\in\Cl_F^+} \delta(\phi_1)  \phi_1^{-1}(\gc\gd)  [\gc], 
\end{align}

 The  constant term of $\cE_{\phi_1,\phi_2}$  is related to the {\it imprimitive} $p$-adic  zeta function 
\begin{align}\label{eq:imprimitive-zeta}
\zeta_{\phi_1,\phi_2}(X)=\zeta_{\phi}(X) \prod_{\gq|\gn, \gq\nmid \mathrm{cond}(\phi)}(1-\phi(\gq)^{-1}(1+X)^{-{\frac{\log_p(\rN_{F/\Q}(\gq))}{\log_p(u)}}}\rN_{F/\Q}(\gq)^{-1})\in \varLambda_{\cO}, \text{ where}  
\end{align}
we recall that   $\zeta_{\phi}(X)$ is the $p$-adic inverse Mellin transform of the Deligne--Ribet $p$-adic $L$-function $L_p(s,\phi\omega_p)$ satisfying the interpolation property that, for all positive integers $k$, one has
\begin{align}\label{eq:interpolation_p-adic_L-fcn}
	\zeta_{\phi}(u^{k-1}-1)=L_p(1-k,\phi\omega_p)=L(1-k,\phi\omega_p^{1-k})\cdot
	\prod_{ \gq\in\Sigma_p} (1-\phi\omega_p^{1-k}(\gq) \rN_{F/\Q}(\gq)^{k-1}).
\end{align}

  One sees that  the  weight $1$ specialization of $\cE_{\phi_1,\phi_2}$  is  $f_{\phi_1}$.  As the Eisenstein family $\cE_{\phi_1,\phi_2}$ is $p$-ordinary, the  argument of \cite[\S6]{Pi13} shows that for any open admissible  affinoid $\cU$ of $\cW^{\parallel}$, the push-forward  of the Eisenstein family $\cE_{\phi_1,\phi_2}$ along the  natural inclusion $\Lambda_\cO \hookrightarrow \cO(\cU)$ belongs to the module $\mathbf{M}^{\dagger}_{\cU}$ introduced in \eqref{eq:coleman-families}, for   $\gn$   the prime-to-$p$ part of $\gn_1\gn_2$ and  $n$ such that  $\gn_1\gn_2$ divides  $\gn p^n$.

\subsection{Eigenvarieties}\label{sec:full_eigencurve}
We recall that a Banach $\cO(\cU)$-module $P$ satisfies the property $(\mathbf{Pr})$ if there is a Banach $\cO(\cU)$-module $Q$ such that $P \oplus Q$ is potentially orthonormalizable \cite[\S2]{Bu07}.
\begin{lemma} The Banach $\cO(\cU)$-modules  $\mathbf{M}^{\dagger}_{\cU}(v)$ and $\mathbf{S}^{\dagger}_{\cU}(v)$ satisfy $(\mathbf{Pr})$.
\end{lemma}
\begin{proof}We recall that $(\overline{\pi}_n)_* (\omega^{\sw_{\cU}^{\circ}})$ is an invertible sheaf on $\mathscr{X}^*_{\gc}(\mu_{p^n},v) \times \mathscr{U}$ by Proposition~\ref{trivialtorsor}. We choose a finite affinoid  cover $\{ \cB_i \}_{1\leqslant i \leqslant r}$  of  $\cX^*_{\gc}(\mu_{p^n},v)$ such that $(\overline{\pi}_n)_* (\omega^{\sw_{\cU}^{\circ}})_{\mid \cB_i\times\cU }$ is a free sheaf of rank $1$ for every $i$, and we consider 
the augmented \v{C}ech complex associated to the line bundle $(\overline{\pi}_n)_* (\omega^{\sw_{\cU}^{\circ}})$:
\[ 0 \to M_0=\rH^0(\cX^*_{\gc}(\mu_{p^n},v) \times \cU, (\overline{\pi}_n)_* (\omega^{\sw_{\cU}^{\circ}})) \to M_1 \to \cdots \to M_r \to 0. \]
This complex is exact because $h_n:\cX^*_{\gc}(\mu_{p^n},v) \to \cX^*_\gc(v)$ is finite and  $\cX^*_\gc(v)$ is an affinoid (so it has no cohomology in degree $\geqslant 1$ by \cite[\S6]{Bosch14}), hence  $M_0$ satisfies $(\mathbf{Pr})$ by  \cite[Lem.5.2]{Pi13}.  We deduce that both  $\mathbf{M}^{\dagger}_{\cU}(v)$ and $\mathbf{S}^{\dagger}_{\cU}(v)$ satisfy $(\mathbf{Pr})$ as direct $\cO(\cU)$-factors in 
$\bigoplus_{\gc \in \Cl_F^+} \rH^0(\cX^*_{\gc}(\mu_{p^n},v) \times \cU, (\overline{\pi}_n)_* (\omega^{\sw_{\cU}^{\circ}}))$. 
\end{proof}
 Buzzard's eigenvariety machine \cite[\S5]{Bu07} allows then to associate an eigencurve $\cC$ to the data  
\[
(\cO_{\cU}, \mathbf{M}^{\dagger}_{\cU}, \Z[T_{\gl}, U_{\gp}: \gl \nmid \gn p, \gp\in \Sigma_p]), \] 
where the affinoids  $\cU$ form an admissible cover of $\cW^{\parallel}$ such that for $v$ sufficiently small the $\cO(\cU)$-module $\mathbf{M}^{\dagger}_{\cU}(v)$ admits  a slope decomposition with respect to the compact operator $U_p$. 
 Let $\mathbf{M}^{  \leqslant s}_{\cU}(v)$ denote the projective finite type $\cO(\cU)$-module of elements of  slope at most $s\in \Q_{ \geqslant 0}$
 and consider   the  finite and flat  affinoid $\cO(\cU)$-algebra  $\cT^{ \leqslant s}_{\cU}$    defined as the image of $\cO(\cU)[T_{\gl}, U_{\gp}: \gl \nmid \gn p, \gp \in \Sigma_p]$ in $\End_{\cO(\cU)}(\mathbf{M}^{  \leqslant s}_{\cU}(v))$. The Hilbert eigencurve $\cC$, obtained by glueing together
 $\mathrm{Spm}(\cT^{ \leqslant s}_{\cU})$ above the spectral curve for $U_p$,  is reduced and  endowed with a  locally finite flat morphism $\sw: \cC \rightarrow \cW^{\parallel}$ called the weight map. 
 
 Similarly, we obtain a cuspidal Hilbert eigencurve $\sw: \cC_{\cusp} \rightarrow \cW^{\parallel}$ which is locally finite and flat and a Hilbert cuspidal eigenvariety $\kappa:\cE\to \cW\times\cW^{\parallel}$ of dimension $d+1$ which is locally finite and torsion-free. 
The ordinary locus (i.e., the locus where the $U_p$-slope is zero) of  $\cC$ (resp.~$\cC_{\cusp}$ and $\cE$) has a formal model given by the ordinary Hida Hecke algebra $\mathfrak{h}$ \cite{hida90} (resp.~cuspidal ordinary Hecke algebra $\mathfrak{h}^{\cusp}$ and the nearly-ordinary Hecke algebra $\mathfrak{h}^{\nord}$) of level $\gn p^\infty$ constructed using the Hecke operators $\{T_{\gl}, U_{\gp}: \gl \nmid \gn p, \gp\in \Sigma_p\}$.
By construction of $\cC$, there exist bounded  global sections $\{ T_{\gl},U_{\gp}\}_{\gl \nmid \gn p, \gp\in \Sigma_p}$. A fundamental  tool in the study of the geometry of $\cC$ is the  universal $2$-dimensional pseudo-character  
\begin{align}\label{pseudo-char}
	\tau_{\cC}:\rG_{F} \rightarrow  \cO_{\cC}(\cC),
\end{align}
which is unramified at all primes $\gl\nmid p\gn$ and maps an arithmetic Frobenius  $\Frob_\gl$ to $T_{\gl}$ (see 
\cite[Prop.~7.1.1]{Chenevier}). While $\tau_{\cC}$ interpolates $p$-adically the traces of semi-simple $p$-adic Galois representations attached to the classical points of $\cC$ which are De Rham at $p$, the semi-simple $p$-adic Galois representation attached to an arbitrary specialization of $\tau_{\cC}$  is only trianguline at $v\in\Sigma_p$ and $U_{v}$ interpolates the crystalline period (see, for example \cite[\S5.3]{BDJ}). The above discussion  also holds if one replaces $\cC$ by $\cE$.
Let $\cC^{\full}$ (resp.~$\cE^{\full}$) be the $p$-adic eigencurve (resp.~eigenvariety) over $\cW^{\parallel}$ (resp.~$\cW\times\cW^{\parallel}$) of tame level $\gn$ constructed by the 
 data 
\[
(\cO_{\cU}, \mathbf{M}^{\dagger}_{\cU}, \Z[T_{\gl}, U_{\gp}, U_{\gq} : \gl \nmid \gn p, \gp\in \Sigma_p, \gq \mid \gn ]), 
 \] 
It is endowed with a locally finite surjective morphism $\cC^{\full}\to \cC$ (resp.~$\cE^{\full}\to \cE$). 
We recall that $\gn_1$ is relatively prime to $p$. The following proposition asserts that these surjective morphisms are indeed  isomorphisms locally at (not necessarily classical)  intersection points of $\cE_{\phi_1,\phi_2}$ with  $\cC_{\cusp}$. 
	
\begin{prop}\label{bad-Hecke-ops}
Let $\cP$ be a maximal ideal of $\mathfrak{h}[1/p]$ corresponding to an  intersection point of $\cE_{\phi_1,\phi_2}$ and $\cC_{\cusp}$. Then $\cC^{\full}\to \cC$ is \'etale at $\cP$. An analogous  assertion also holds for $\cE$.
\end{prop}
\begin{proof} We will only deal with the case of $\cC$,  the proof for $\cE$ being very similar. 
By functoriality of the eigenvariety machine, for each $\gn' \mid \gn$, one has a closed immersion $\iota_{\gn'} : \cC_{\gn'} \hookrightarrow \cC$ from the eigencurve $\cC_{\gn'}$ of tame level $\gn'$. As $\cP$ has  tame level  exactly $\gn$, it does  not belong to the image of $\iota_{\gn'}$ for any $\gn' \supsetneq \gn$. After shrinking $\cU$ to avoid the image of $\iota_{\gn'}$ for any  $\gn' \supsetneq \gn$,  the classical points in $\cU$ are all $\gn$-new. Using the Zariski density of classical points and  the  Strong Multiplicity One Theorem, one concludes   that $\cC^{\full}$ is reduced at $\cP$, and $U_\gq\in \cO_{\cC}(\cU)$ for all $\gq\mid  \gn$ as in  \cite[Prop.~4.4]{BDP}.
\end{proof}

For $\cP$ as in the above proposition, we  denote respectively by $\varLambda_{\sw(\cP)}$ (resp.  $\cT_{\cP}$ and $\cT_{\cP}^{\cusp}$) the complete localization of $\varLambda_{\cO}[1/p]$  (resp. $\cC^{\full}$ and $\cC^{\cusp,\full}$) at $\sw(\cP)$ (resp. at  $\cP$). Set $\mathbf{M}^{\ord}_{\cU}=\mathbf{M}_{\cU}^{\dagger, \leqslant 0}$ and 
$\mathbf{S}^{\ord}_{\cU}=\mathbf{S}_{\cU}^{\dagger, \leqslant 0}$, and  let  $\mathbf{S}^{\ord}_{\cP}=\mathbf{S}^{\ord}_{\cU} \widehat{\otimes}_{\cT^{\cusp,\leqslant 0}_{\cU}} \cT_{\cP}^{\cusp} $ and $\mathbf{M}^{\ord}_{\cP}=\mathbf{M}^{\ord}_{\cU} \widehat{\otimes}_{\cT^{\leqslant 0}_{\cU}} \cT_{\cP}$.
By  \cite[Thm.~5]{hida93} (or \cite[Thm.~4.30]{hida-PAF}), we have the following perfect pairing
\begin{align}\label{eq:cuspidal-duality}
\cT_{\cP}^{\cusp} \times \mathbf{S}^{\ord}_{\cP}\to \varLambda_{\sw(\cP)} \quad
(h,\cF)\mapsto C(\go,h\cdot \cF).
\end{align}

\subsection{Constant terms of Eisenstein families}\label{sec:122}
Applying Hida's ordinary idempotent $e=\displaystyle\lim_{n\to\infty} U_p^{n!}$ to \eqref{eq:fund_seq} yields a short exact sequence
\begin{align}\label{eq:fund_seq_ord}
0\to \mathbf{S}^{\ord}_{\cU} \to \mathbf{M}^{\ord}_{\cU} \xrightarrow{\res} e\cdot \cO(\cU)[\sP_1(p^n)]^{(\phi_p)} \to 0.
\end{align}
We let $(e_1,e_2)$ be the canonical basis of $F^2$.  The group $\GL_2(F)$ acts on the left on $F^2$ by 
\[ 
g \ast (x,y)= (x,y) \cdot g^{-1} \text{ for all }g \in \GL_2(F).
\] 
This action extends to an action of $\GL_2(\A_{F,f})$ on $\A_{F,f}^2$.
 As we have seen in \S\ref{torcom}, a cusp of level $\Gamma_1(\gn p^n)$ corresponds to two fractional ideals $\ga$ and $\gb$ with extra data (see \cite[Def.~2.1]{dim04} for more details). In fact, we can associate to a cusp an $\go$-lattice $L$ of $F^2$ \cite[p.37]{dim04}  and  $\GL_2(\A_{F,f})$ acts on such lattices (hence acts on the cusps) as follows
\[ 
g \ast L =  F^2 \cap (L \cdot g^{-1}) \text{ for all } g \in \GL_2(\A_{F,f}).
\]

In particular, $\gb= g \ast L \cap (F\cdot e_1)$ and $ \ga^{-1}\gd^{-1}= g \ast L \cap (F\cdot e_2)$. 
For $[\gc]\in \Cl_F^+$, we let  $\infty(\gc)=\left( \go \oplus \gd^{-1}\gc^{-1},  \alpha_{\gn p^n}: \gc /\gn p^n \gc  \simeq \go/\gn p^n  \right)$ be the {\it unramified} cusp  in the sense of \cite[Def.~3.2]{dim04} associated to the pair of ideals $(\ga,\gb) = (\gc, \go)$.
As $\GL_2(\A_{F,f}) $ acts transitively, it follows  that the set of cusps of level $\Gamma_1(\gn p^n)
$ can be identified with  the orbit of  $\infty=\infty(\gd^{-1})$ given, up to isomorphisms, by
\begin{align}\label{eq:cusp_adelic_form}
B(F)^+  \backslash \GL_2(\A_{F,f})/\Gamma_1(\gn p^n),
\end{align}
where  $B(F)^+$ is the group of upper triangular matrices having a  totally positive determinant. 
\begin{lemma}\label{cusp_repr}
Representatives of \eqref{eq:cusp_adelic_form} can be taken of the form $\left(\begin{smallmatrix}
	t_1 & 0 \\ 0 & t_2\end{smallmatrix}\right)\cdot g$, where 
\begin{enumerate}
\item $g\in \GL_2(\widehat{\go})$ with  $g_v=\left(\begin{smallmatrix} 1 & 0 \\ 0 & 1\end{smallmatrix}\right)$ for all $v\nmid \gn p$,  and
\item $t_i\in \A_{F,f}^\times$ with $\val_v(t_i)=0$ for all  $v\mid \gn p$   ($i=1,2$).
\end{enumerate}
\end{lemma}
 
\begin{defn} A cusp of level $\Gamma_1(\gn p^n)$  or level $\Gamma_1(\gn) \cap \Gamma_0(p^n)$  is $p$-ordinary (multiplicative) if it is represented by  $t\cdot g$ for some $t\in T(\A_{F,f})$ and for some $g\in \GL_2(\widehat{\go})$ with $g_p\in \Gamma_{0}(p^n)$. 
  \end{defn} 
 
 \begin{lemma}\
 \begin{enumerate}
  \item The cusps of $\coprod_{\gc \in \Cl_F^+} \overline{\mathscr{X}_{\gc}} (v)$ are exactly the cups of  level $\Gamma_1(\gn)$.
 \item The cusps of $\coprod_{\gc \in \Cl_F^+} \overline{\mathscr{X}_{\gc}}(\mu_{p^n},v)$ are exactly the $p$-ordinary cusps of  level $\Gamma_1(\gn p^n)$. 
  \end{enumerate}
 \end{lemma}
 \begin{proof}
  (i)   Mumford's  semi-abelian scheme associated to any cusp of  level $\Gamma_1(\gn)$ is $p$-ordinary.

(ii)  Since the lattice $\go \oplus \go$ is  endowed  with the natural unramified  $\Gamma_1(\gn p^n)$-level structure $\alpha_{\gn p^n}$ (see \cite[Def.~2.2]{dim04}), the  level structure at $p$ of  Mumford's  semi-abelian scheme associated to a $p$-ordinary cusp  of  level $\Gamma_1(\gn p^n)$ is necessary of multiplicative type (see \cite[p.44]{dim04}). 
 \end{proof}
 
Recall that  $\gn_i=\mathrm{cond}(\phi_i)$ for $i=1,2$. As before, we denote by $\gn$ the prime-to-$p$ part of $\gn_1\gn_2$ and let $n$ be the smallest positive integer such that $\gn_1\gn_2$ divides $\gn p^n$. Denote by $\gn'_1$ the prime-to-$\gn_2$ part of $\gn_1$ and denote by $\gn'_2$ the prime-to-$\gn_1$ part of $\gn_2$. For any $\alpha\in \A_{F,f}$ and for $i=1,2$, we denote by $\alpha_{\gn_i}$ the $\gn_i$-part of $\alpha$. 
We next recall the constant terms of $E_k(\phi_1,\phi_2)$ for all $k\geqslant 2$ computed in \cite[\S 3]{Shih} (see also \cite{dasgupta-kakde}). In \textit{loc.~cit.}, the Eisenstein series $E_k(\phi_1,\phi_2)$ is a member of a family of adelic Eisenstein series  
evaluated at $s=\frac{1-k}{2}$ (see also  \cite[\S3.7]{Bump97}). This family is defined via explicit choices of local sections $f_{\phi_1,\phi_2,s,v}$ in the normalized induced representations $\mathrm{Ind}_{B(F_v)}^{\GL_2(F_v)} (\phi_2|\cdot|^s_v,\phi_1|\cdot|_v^{-s})$ for $v$ a  place of $F$ (see \cite[\S3.4]{Shih}). We consider a specific 
value of the local intertwining operator (see \cite[(7.22)]{Bump97})
\[
M_{\phi_1,\phi_2,v,s}=\int_{F_v} f_{\phi_1,\phi_2,s,v}\left(
\left(\begin{matrix}0 & -1 \\ 1 & x \end{matrix}\right)
\left(\begin{matrix}1 & 0 \\ \varpi_v^{\val_v(\gn_2)} & 1 \end{matrix}\right)\right)dx.
\] 
\begin{prop}\cite[Prop.~3.10]{Shih}\label{constant_term_finite_wt}
The geometric constant term of $E_k(\phi_1,\phi_2)$ at the $p$-ordinary $\Gamma_1(\gn p^n)$-cusp $\left(\begin{smallmatrix} t_1 & 0 \\ 0 & t_2 \end{smallmatrix}\right)\cdot \left(\begin{smallmatrix} a & b \\ c & d \end{smallmatrix}\right)$ is zero, unless  the following three conditions hold
\begin{enumerate}
\item $\val_v(c)\geqslant \val_v(\gn_2)$ for all $v\mid \gn'_2$.
\item $\val_v(c)=0$ for all $v\mid \gn'_1$.
\item $\val_v(c)=\val_v(\gn_2)$ for all $v$ dividing both $\gn_1$ and $\gn_2$, 
\end{enumerate}
in which case, letting $\varepsilon_v(s,\phi_{1,v},\psi_v)$ denote the local epsilon factor,  it is given by 
\begin{align}\label{eq:const}
\begin{split}
&\phi_1(t_1)\phi_2(t_2) \rN_{F/\Q}(t_1/t_2)^{-k/2}
\phi_2^{-1}(\gn_1')\rN_{F/\Q}(\gn_1')^{-1}\phi_{2}(d_{\gn_2}\hat{\go})\phi_1^{-1}(-c_{\gn_1}/\gn_2') \\
& \times L^{(\gn)}(1-k,\phi)\cdot
\prod_{v|\gn_1,v|\gn_2} M_{\phi_1,\phi_2,v,(1-k)/2}\prod_{v|\gn_1}\varepsilon_v(2-k,\phi_{1,v},\psi_v)^{-1}.
\end{split}
\end{align}
\end{prop}

\begin{prop}\label{const_Eisen_family}
The constant term $a_C(0,\cE_{\phi_1,\phi_2})$ of $\cE_{\phi_1,\phi_2}$ at the 
$p$-ordinary $\Gamma_1(\gn p^n)$-cusp $C$ represented by $\left(\begin{smallmatrix} t_1 & 0 \\ 0 & t_2 \end{smallmatrix}\right)\cdot \left(\begin{smallmatrix} a & b \\ c & d \end{smallmatrix}\right)$
belongs to $\zeta_{\phi_1,\phi_2}\cdot \cO(\cU)^\times$ if  the  three conditions of Proposition~\ref{constant_term_finite_wt} hold, else  
it is zero.
\end{prop}
\begin{proof} The specialization of $\cE_{\phi_1,\phi_2}$ at  $k \geqslant 2$ such that  $\omega_p^{k-1}=\mathbf{1}$
  is the $p$-stabilization of the weight $k$ Eisenstein series $E_k(\phi_1,\phi_2)$ of level $\Gamma_1(\gn_1\gn_2)$.  
   As in \cite[Prop.~4.7]{BDP}, one  obtains the assertions by computing the constant terms of $E_k(\phi_1,\phi_2)$ at all $p$-ordinary   $\Gamma_1(\gn p^n)$-cusps  using the formula \eqref{eq:const}  from Proposition~\ref{constant_term_finite_wt}, all the terms of which interpolate $p$-adically  (see \cite[Prop.~4.5]{Shih}). 
  \end{proof}

\subsection{Hecke duality and Proof of Theorem~\ref{congruencemodule}}\label{sec:123}
The {\it Eisenstein ideal} $\cJ_{\phi_1,\phi_2}\subset \mathfrak{h}^{\cusp}$ associated to $\cE_{\phi_1,\phi_2}$ is the image of the kernel of the $\cE_{\phi_1,\phi_2}$-projection $\mathfrak{h} \to \varLambda_{\cO}$ via the natural surjection $\mathfrak{h} \twoheadrightarrow \mathfrak{h}^{\cusp}$ of the Hida Hecke algebra introduced in  \S\ref{sec:full_eigencurve}. The {\it congruence ideal} $\cI_{\phi_1,\phi_2}\subset \varLambda_{\cO}$ of the Eisenstein family $\cE_{\phi_1,\phi_2}$ is uniquely  characterized  by the  following commutative diagram of $ \varLambda_{\cO}$-algebras
\begin{equation}
\begin{tikzcd}\label{eq:congruencesEisenstein}
\mathfrak{h} \arrow[r, two heads] \arrow[d, two heads] 
& \varLambda_{\cO} \arrow[d, two heads] \\ 
\mathfrak{h}^{\cusp}/\cJ_{\phi_1,\phi_2}  \arrow{r}{\sim}
&  \varLambda_{\cO}/\cI_{\phi_1,\phi_2}.
\end{tikzcd}
\end{equation}

We next establish a precise relation between the congruence ideal and the constant term of $\cE_{\phi_1,\phi_2}$. 
Previously it was only known by the work of Wiles \cite[Thm~4.1]{wiles90}  that, after inverting $p$ and the height one primes corresponding to trivial zeros, $(\zeta_{\phi_1,\phi_2})$ divides $ \cI_{\phi_1,\phi_2}$. 

\begin{thm} \label{congruencemodule-general} If $\phi\not\equiv\omega_p^{-1}\pmod{\varpi}$, then 
one has $ \cI_{\phi_1,\phi_2} \cdot \varLambda_{\cO}[1/p] = \zeta_{\phi_1,\phi_2}  \cdot \varLambda_{\cO}[1/p]$.
\end{thm}
\begin{proof} As $\varLambda_{\cO}[1/p]$ is a Dedekind domain, it suffices to show equality after localization at an  arbitrary maximal ideal  $\cP$. 
Localizing the isomorphism in  \eqref{eq:congruencesEisenstein} at  $\cP$, we obtain 
\begin{align}\label{eq:comm_diag_cong_ideal0}
 \cT^{\cusp}_{\cP} / \cJ_{\phi_1,\phi_2,\cP}  \simeq
 \varLambda_{\sw(\cP)}/\cI_{\phi_1,\phi_2,\cP}.
\end{align} 

Choose a sufficiently small open admissible affinoid $\cU$ of $\cW^{\parallel}$ containing  $\cP$. 
The image of $\cE_{\phi_1,\phi_2}$ under the map $\res$ from \eqref{eq:fund_seq_ord} is contained in a free of rank $1$ direct summand in  $ \cO(\cU)[\sP_1(p^n)]^{(\phi_p)}$ having   $\cO(\cU)$-basis  $c_{\phi_1,\phi_2}$. The module  
$\mathbf{M}_{\phi_1,\phi_2}=\res^{-1}(\cO(\cU)\cdot c_{\phi_1,\phi_2})$ is Hecke stable and  fits in a short exact sequence of free $\cO(\cU)$-modules 
\[
0\to \mathbf{S}^{\ord}_{\cU} \to \mathbf{M}_{\phi_1,\phi_2} \xrightarrow{\res} \cO(\cU)\cdot c_{\phi_1,\phi_2} \to 0.
\]
 Recall that $\Spec(\varLambda_{\cO})$ is the irreducible component of $\Spec(\mathfrak{h})$ corresponding to $\cE_{\phi_1,\phi_2}$. Taking the complete localization of the above short exact sequence at $\cP$ viewed as a maximal ideal of $\cT^{\leqslant 0}_{\cU}$, we obtain a short exact sequence of free $\varLambda_{\sw(\cP)}$-modules
\begin{align}\label{eq:fund_seq_eisenstein_component}
0\to \mathbf{S}^{\ord}_{\cU,\cP} \to \mathbf{M}_{\phi_1,\phi_2,\cP} \xrightarrow{\res} \varLambda_{\sw(\cP)}\cdot c_{\phi_1,\phi_2} \to 0.
\end{align}

If   $\zeta_{\phi_1,\phi_2}\notin \cP$ then $ (\zeta_{\phi_1,\phi_2}) =\varLambda_{\sw(\cP)}\supset
\cI_{\phi_1,\phi_2,\cP}$. Otherwise, as $\res(\cE_{\phi_1,\phi_2})\in \zeta_{\phi_1,\phi_2} \cdot \varLambda_{\sw(\cP)}^\times \cdot c_{\phi_1,\phi_2}$ by  Proposition~\ref{const_Eisen_family}, there exists a $\varLambda_{\sw(\cP)}$-adic cuspform $\cH\in \mathbf{S}^{\ord}_{\cU,\cP}$ congruent to $\cE_{\phi_1,\phi_2}$  modulo $\zeta_{\phi_1,\phi_2}$. Using \eqref{eq:cuspidal-duality}, one obtains  a surjective $\varLambda_{\sw(\cP)}$-algebra homomorphism 
\begin{align}\label{eq:cong_module}
\cT^{\cusp}_{\cP}/\cJ_{\phi_1,\phi_2,\cP}\twoheadrightarrow \varLambda_{\sw(\cP)}/ (\zeta_{\phi_1,\phi_2}), \quad h\mapsto C(\go,h\cdot \cH) \mod \zeta_{\phi_1,\phi_2}.
\end{align}
  In view of \eqref{eq:comm_diag_cong_ideal0},  this   precisely shows  the  inclusion of principal ideals $(\zeta_{\phi_1,\phi_2}) \supset \cI_{\phi_1,\phi_2,\cP}  $ in the discrete valuation ring $ \varLambda_{\sw(\cP)}$.  It remains to show the opposite inclusion.

Observe  that  the freeness of $\cT^{\cusp}_{\cP} $ as $\varLambda_{\sw(\cP)}$-module allows us to  lift the isomorphism \eqref{eq:comm_diag_cong_ideal0} to a morphism $\cT^{\cusp}_{\cP} \to \varLambda_{\sw(\cP)}$ of $\varLambda_{\sw(\cP)}$-modules   (though not necessarily to an algebra homomorphism). Using \eqref{eq:cuspidal-duality}, this lift gives rise to a $\varLambda_{\sw(\cP)}$-adic cuspform $\cG$ (not necessarily an eigenform) with Fourier coefficients in $\varLambda_{\sw(\cP)}$,  such that  
$C(\gm,\cG) \equiv C(\gm,\cE_{\phi_1,\phi_2}) \mod \cI_{\phi_1,\phi_2,\cP}$ for all non-zero ideals $\gm \subset \go$. As this does not provide in general a congruence between the constant terms, we are led to split the remainder of the argument in two cases. 
As $\cI_{\phi_1,\phi_2,\cP}  \subset (\zeta_{\phi_1,\phi_2})$, one can define a $\varLambda_{\sw(\cP)}$-adic  family
\begin{align} \label{eq:FP}
 \cF_0=\frac{\cE_{\phi_1,\phi_2}- \cG}{2^{-d} \zeta_{\phi_1,\phi_2}} \in \mathbf{M}_{\phi_1,\phi_2,\cP},
 \end{align}
which provides a section of the map $\res$ in \eqref{eq:fund_seq_eisenstein_component},  i.e.,
$\mathbf{M}_{\phi_1,\phi_2,\cP}=\mathbf{S}^{\ord}_{\cU,\cP} \oplus \varLambda_{\sw(\cP)}\cdot \cF_0$ as $ \varLambda_{\sw(\cP)} $-module. 

Assume first  that $\gn_1\ne\go$. In this case   \eqref{Eisensteinconstanterms} yields 
$C((0),\cE_{\phi_1,\phi_2}) =0$, hence $\cE_{\phi_1,\phi_2} \equiv \cG \mod \cI_{\phi_1,\phi_2,\cP}$ by the $q$-expansion Principle. Consequently, $\res(\cE_{\phi_1,\phi_2})\in \cI_{\phi_1,\phi_2,\cP}\cdot c_{\phi_1,\phi_2}$ which combined with $\res(\cE_{\phi_1,\phi_2})\in \zeta_{\phi_1,\phi_2} \cdot \varLambda_{\sw(\cP)}^\times \cdot c_{\phi_1,\phi_2}$ from Proposition~\ref{const_Eisen_family}, implies that  $(\zeta_{\phi_1,\phi_2})\subset  \cI_{\phi_1,\phi_2,\cP}$.

Suppose now that $\gn_1=\go$. Given $[\gc]\in\Cl_F^+$, it  follows from \eqref{Eisensteinconstanterms} that  the value of $\cE_{\phi_1,\phi_2}$ at the 
‘infinity’ cusp $\infty(\gc)$  of $\overline{\mathscr{X}_\gc}(\mu_{p^n},v)$  equals $2^{-d} \zeta_{\phi_1,\phi_2}(X) \phi_1^{-1}(\gc\gd)\ne 0$ (see \S\ref{sec:122}). Then $ \cF_0$ has   constant term $1$ at $\infty=\infty(\gd^{-1})$ and its 
specialization $F_{\cP}$    in weight $\sw(\cP)$ lies the generalized eigenspace $\mathbf{M}^{\dagger}_{\sw(\cP)}\lsem E_{\cP} \rsem$
attached to the specialization $E_{\cP}$  of  $\cE_{\phi_1,\phi_2}$   in weight $\sw(\cP)$. 
In particular, for any prime  $\mathfrak{q}\nmid \gn p$ the operator $S_\mathfrak{q}$ acts on  
$F_{\cP} \mod \varpi$ as $(\phi_1\phi_2 \omega_p)(\mathfrak{q})$. The operator $S_\mathfrak{q}$  defined in \cite[p.~24]{AIP} is renormalized by the factor 
$\mathrm{N}_{F/\mathbb{Q}}(\mathfrak{q})^{-2}$ to match the determinant formula of the $p$-adic Galois representations described in \cite[Thm.~5.1]{AIP}. 
 Throughout this work (as in \cite[p.~44]{demaria-thesis}) we do not  renormalize
$S_\mathfrak{q}$. 

If the inclusion of principal ideals $\cI_{\phi_1,\phi_2,\cP}  \subset (\zeta_{\phi_1,\phi_2}) $ of $\varLambda_{\sw(\cP)}$ was strict, then:
\begin{enumerate}
\item $C(\gm,F_{\cP}) =0$ for all non-zero ideals $\gm \subset \go$, and   
\item for $[\gc]\in\Cl_F^+$, the value of $F_{\cP}$  at the  cusp $\infty(\gc)$  of 
$\overline{\mathscr{X}_\gc}(\mu_{p^n},v)$  equals $\phi_1^{-1}(\gc\gd)$.  
\end{enumerate}
Inspired by the work of  Emerton \cite[Prop.~1]{Em99} when  $F=\Q$, further extended by  Lafferty \cite[p.~766]{Laff19},   
we now provide a contradiction using a geometrically flavored argument based on the irreducibility of the Igusa tower 
  $\overline{\mathscr{X}_\gc}(\mu_{p^\infty},0)$.  As the special fiber  $\overline{\mathscr{X}_\gc}(\mu_{p^\infty},0)_1$ is geometrically irreducible 
by a theorem of Ribet    \cite{ribet1975p} (see also \cite{Hida09}), the $q$-expansion principle shows that 
  $F_{\cP}$  coincides with its constant term $C((0),F_{\cP})$ seen as a modular form of weight $0$
   (see \cite[Prop.~4.2]{AIS} for more details).   Hence, by definition of the operator $S_\mathfrak{q}$, the value of 
 $S_\mathfrak{q}(C((0),F_{\cP}))$ on $\overline{\mathscr{X}_\gc}(\mu_{p^\infty},0)$ is exactly the value of $C((0),F_{\cP})$ on 
        $\overline{\mathscr{X}_{\gc\gq^{-2}}}(\mu_{p^\infty},0)$.  Combining this  with the fact that $S_\mathfrak{q}$  acts on  
$F_{\cP} \mod \varpi$ as $(\phi_1\phi_2 \omega_p)(\mathfrak{q})$, we find 
        \[(\phi_1\phi_2 \omega_p)(\mathfrak{q})\cdot \phi_1^{-1}(\gc\gd)\equiv \phi_1^{-1}(\gc\gq^{-2}\gd) \pmod{\varpi},\]
 i.e., $\phi=\phi_2 \phi_1^{-1}\equiv \omega_p^{-1}\pmod{\varpi}$. This contradicts our running assumption on $\phi$. 
Hence we conclude that $\cI_{\phi_1,\phi_2,\cP} = (\zeta_{\phi_1,\phi_2}) $.
 \end{proof}

The techniques employed in the above proof, namely the construction of the $\varLambda_{\sw(\cP)}$-adic  form  $\cF_0$ in \eqref{eq:FP}, have the following important consequence.

\begin{prop}\label{prop:hecke_duality} Let $\cP$ be a maximal ideal of the  ordinary Hecke algebra $\mathfrak{h}[1/p]$ corresponding to an  Eisenstein intersection point of $\cE_{\phi_1,\phi_2}$ and $\cC_{\cusp}$, and let  $\cT_{\phi_1,\phi_2,\cP}$ be the quotient of $\cT_{\cP}$ acting faithfully on the $\varLambda_{\sw(\cP)}$-module $\mathbf{M}_{\phi_1,\phi_2,\cP} $ in \eqref{eq:fund_seq_eisenstein_component}. Then, one has 
an isomorphism of $\varLambda_{\sw(\cP)}$-modules
\begin{align}\label{eq:structure_hecke_alg_not_inert_case}
\cT_{\phi_1,\phi_2,\cP}  \simeq \cT_{\cP}^{\cusp} \times_{\varLambda_{\sw(\cP)}/(\zeta_{\phi_1,\phi_2})} \varLambda_{\sw(\cP)}, 
\end{align}
and a perfect pairing of $\Lambda_{\sw(\cP)}$-modules
\begin{align}\label{eq:non-cusp-pairing}
\cT_{\phi_1,\phi_2,\cP}\times \mathbf{M}_{\phi_1,\phi_2,\cP}   \to \varLambda_{\sw(\cP)}, \quad (h,\cG)\mapsto C(\go,h\cdot \cG). 
\end{align}
In particular,  there exists an element $h_0 \in \cT_{\phi_1,\phi_2,\cP}$ such that  for any $\cG \in \mathbf{M}_{\phi_1,\phi_2,\cP}$, we have 
\[ C(\go,h_0 \cdot \cG) \cdot \sum_{[\gc]\in\Cl_F^+}  \phi_1^{-1}(\gc\gd)  [\gc] = C((0), \cG).\]

\end{prop}
\begin{proof} The isomorphism \eqref{eq:structure_hecke_alg_not_inert_case} follow from the localization of  
\eqref{eq:congruencesEisenstein} at $\cP$ and Theorem~\ref{congruencemodule-general}.  

We assume first that $\gn_1 \ne \go$. By Nakayama's lemma,  to show the   perfectness of the pairing 
\eqref{eq:non-cusp-pairing} of free $\varLambda_{\sw(\cP)}$-modules, it suffices to establish that it is   right  injective residually and left injective. As $C((0),\cE_{\phi_1,\phi_2}) = 0$ by  \eqref{Eisensteinconstanterms} we deduce that  $C((0),\cG) = 0$ for all $\cG \in \mathbf{M}_{\phi_1,\phi_2,\cP}$. The left injectivity in \eqref{eq:non-cusp-pairing} then follows from 
$q$-expansion principle and from the formula $C(\go , h\cdot (T_\gq   \cdot \cG)) = C(\gq, h\cdot \cG)$, 
where  $\mathfrak{q}$ is any non-zero integral ideal of $\go$ (see \cite[(4.67)]{hida-PAF}), since  by definition
$\cT_{\phi_1,\phi_2,\cP}$ acts faithfully on  $\mathbf{M}_{\phi_1,\phi_2,\cP}$. For  right injectivity of the reduction of \eqref{eq:non-cusp-pairing} modulo  the maximal ideal of 
$\varLambda_{\sw(\cP)}$, we  use again the vanishing of the constant term, the $q$-expansion principle and the  formula $C(\go , T_\gq   \cdot \overline{\cG}) = C(\gq,  \overline{\cG})$ for $\overline{\cG}$ in the generalized 
eigenspace $\mathbf{M}_{\phi_1,\phi_2,\cP}\otimes_{\varLambda_{\mathrm{\sw}(\cP)}}\overline{\Q}_p$. 

If $\gn_1=\go$, we consider a $\varLambda_{\sw(\cP)}$-basis of $\mathbf{M}_{\phi_1,\phi_2,\cP}$ adapted to the direct sum decomposition 
$\mathbf{S}^{\ord}_{\cU,\cP} \oplus \varLambda_{\sw(\cP)}\cdot \cF_0$ and a $\varLambda_{\sw(\cP)}$-basis of 
$\cT_{\phi_1,\phi_2,\cP}$ obtained by completing a lift of a basis of $\cT_{\cP}^{\cusp}$ by the element 
$h_0=(0,2^{-d} \zeta_{\phi_1,\phi_2})$ which belongs to $\cT_{\phi_1,\phi_2,\cP}$ by \eqref{eq:structure_hecke_alg_not_inert_case}.
The pairing \eqref{eq:non-cusp-pairing} is represented by a block upper triangular matrix with coefficients in  $\varLambda_{\sw(\cP)}$, 
because $h_0$ annihilates $\mathbf{S}^{\ord}_{\cU,\cP}$. Moreover, by \eqref{eq:cuspidal-duality} the matrix has 
an  invertible block of co-dimension $1$. The perfectness of the pairing 
\eqref{eq:non-cusp-pairing} is then equivalent to the last diagonal coefficient being in $\varLambda_{\sw(\cP)}^\times$. 
By construction, this coefficient is given by
\[C(\go,h_0\cdot \cF_0)=C\left(\go,h_0\cdot \frac{\cE_{\phi_1,\phi_2}}{2^{-d} \zeta_{\phi_1,\phi_2}}\right)= 1. \qedhere \]

\end{proof}  
          
\section{$\sL$-invariants and Galois deformations}\label{Deformations-and-mainthm}
Let $\phi$ be a finite order totally odd  character of $\rG_{F}$. Henceforth, we assume that $\phi$ is {\it irregular} at the prime $\gp$ of $F$ determined by $\iota_p$, i.e.,  $\phi$ has trivial restriction to $\rG_{\gp}$.
In this section, we first review Darmon--Dasgupta--Pollack's definition \cite{DDP} of an $\sL$-invariant  attached to $\phi$ at $\gp$
and  we show its non-vanishing.  We then define a nearly-ordinary deformation functor for  a residually reducible representation
and compute its tangent space using Galois cohomology under the assumption  that the Leopoldt defect $\delta_{F,p}$ vanishes.

\subsection{Non-vanishing of an  $\sL$-invariant}\label{sec:21}
Let $w_0\mid \gp$ be  the prime of $H=\overline{\Q}^{\ker(\phi)}$ singled out by $\iota_p$.  As $\gp$ splits completely in  $H$, i.e., $F_\gp=H_{w_0}$, the  group $\Gal(H/F)$ acts freely and transitively on the set of primes of $H$ above $\gp$ and we have a canonical identification $ \Sigma_{\gp} \simeq \Hom_{\Q_p-\text{alg}}(H_{w_0},\overline{\Q}_p)$ via $\iota_p$.   Consider the homomorphisms
\begin{align}
\begin{split}
\ord_{\gp}& :\cO_{H}[\tfrac{1}{\gp}]^\times\otimes \overline{\Q}_p  \longrightarrow \overline{\Q}_p, 
\quad u\otimes 1\mapsto \ord_{w_0}(u), \text{ and } \\
\log_{\gp} &:\cO_{H}[\tfrac{1}{\gp}]^\times\otimes \overline{\Q}_p  \longrightarrow \overline{\Q}_p,
\quad u\otimes 1\mapsto \log_p( \rN_{F_{\gp}/\Q_p}(\iota_p(u)))=
\sum_{\sigma\in \Sigma_\gp} \log_p(\sigma(\iota_p(u))), 
\end{split}
\end{align}
where  $\cO_{H}[\frac{1}{\gp}]^\times$ denotes  the group of $\gp$-units and we recall that $\iota_p$ embeds $H$ into $H_{w_0}$. 
As in \cite[(7)]{DDP}, the  $p$-adic number 
\begin{align}\label{eq:L-inv}
\sL^{\gp}(\phi)=-\frac{\log_{\gp}(u_{\phi})}{\ord_{\gp}(u_{\phi})}\in \overline{\Q}_p,
\end{align}
depends only on $\phi$ and $\gp$, and not on the particular choice of a basis 
 $u_{\phi}$ of the line $(\cO_{H}[\frac{1}{\gp}]^\times\otimes \overline{\Q}_p)[\phi]$.  
 If $\gp\in\Sigma_p$ is the  unique irregular prime  for $\phi$, then 
$\sL^{\gp}(\phi)$ coincides with Gross'   $\sL$-invariant  $\sL(\phi)$  from  \cite{gross81}.   
We next describe an explicit basis of $\rH^1(F,\phi)$. One has an exact sequence \begin{align}\label{eq:24}
0\to \Hom(\rG_H,\overline{\Q}_p)\to \Hom\left((\cO_H\otimes \Z_p)^\times, \overline{\Q}_p\right)\to \Hom(\cO_H^\times,\overline{\Q}_p),
\end{align}
of left $\overline{\Q}_p[\Gal(H/F)]$-modules. Consider the coset decomposition as in \cite[(3)]{BDF21}
\begin{align}\label{sigma-tilde} 
 \Hom_{\Q-\text{alg}}(H,\overline{\Q})
= \coprod_{\sigma\in \Sigma} \tilde{\sigma}\cdot  \Gal(H/F), \text { where } 
\end{align}
  \begin{itemize}
 \item for  $\sigma\in\Sigma_{\gp}$, 
 $\tilde{\sigma}$ denotes the unique extension to $H$ determining via $\iota_p$ the place $w_0$,
  
 \item for  $\sigma\in \Sigma\setminus\Sigma_{\gp}$,  $\tilde{\sigma}$ denotes an  an  arbitrary extension  of $\sigma$ to $H$.  
\end{itemize}
As in  \cite[\S4]{bet18}, using  the natural identification  $\displaystyle \Hom_{\Q-\text{alg}}(H,\overline{\Q}_p)=\coprod_{w|p} 
\Hom_{\Q_p-\text{alg}}(H_w,\overline{\Q}_p)$, 
 we consider  the basis  $\left(\log_p(\iota_p \circ \tilde\sigma\circ g^{-1}(\cdot))\right)_{\sigma\in \Sigma, g\in  \Gal(H/F) }$ of 
 $\Hom\left((\cO_H\otimes \Z_p)^\times,\overline{\Q}_p\right)$. 
 As $(\cO_H^\times \otimes \overline{\Q}_p)[\phi]=\{0\}$, it follows from \eqref{eq:24} that  $\Hom(\rG_H,\overline{\Q}_p)[\phi^{-1}]\simeq\Hom\left((\cO_H\otimes \Z_p)^\times,\overline{\Q}_p\right)[\phi^{-1}]$, yielding the following: 
\begin{lemma}\label{lemma-21}
The space $\rH^1(F,\phi)\simeq \Hom(\rG_H,\overline{\Q}_p)[\phi^{-1}]$ is $d$-dimensional with  basis 
$([\eta_{\phi}^{(\tilde{\sigma})}])_{\sigma\in \Sigma}$, where $\eta_{\phi}^{(\tilde{\sigma})}$ is the cocycle  whose  image in $\Hom\left((\cO_H\otimes \Z_p)^\times,\overline{\Q}_p\right)$ equals \[ (u \otimes v) \mapsto \displaystyle\sum_{g\in \Gal(H/F)} \phi(g) \log_p \left( \iota_p \circ  \tilde\sigma\circ g^{-1} (u) \cdot v\right).\] 
Moreover $[\eta_{\phi}^{(\tilde{\sigma})}]$ is only ramified at the prime 
of $F$ determined by $\sigma$ via the decomposition $\Sigma=\coprod_{ \gq\in\Sigma_p} \Sigma_{\gq}$. 
\end{lemma}
As $(\phi-\mathbf{1})$ is a basis of the coboundaries and $\phi(\tau)=-1\ne 1$, where $\tau \in \rG_F$ denotes a  complex conjugation,  
each  $[\eta]\in \rH^1(F,\phi)$ is represented by a unique cocycle $\eta$ such that $\eta(\tau)=0$. We will fix $\tau$ and $\eta$ throughout this article and make no further comment. 
The  image of $\eta_{\mathbf{1}}=\log_p\circ \epsilon_{\cyc} \in \Hom(\rG_\Q,\overline{\Q}_p)$, where $\epsilon_{\cyc}$ denotes the $p$-adic cyclotomic character,  under the natural restrictions
\[
\Hom(\rG_\Q,\overline{\Q}_p)\hookrightarrow \Hom(\rG_H,\overline{\Q}_p)
\xrightarrow{\eqref{eq:24}} \Hom\left((\cO_H\otimes \Z_p)^\times,\overline{\Q}_p\right)
\] 
corresponds to $\log_p\circ \rN_{H/\Q}$. The following cohomological interpretation of  $\sL^{\gp}(\phi)$  will be 
used in \S\ref{sec:34}. 

\begin{prop}\label{L-invariant}
Let $\loc_{\gp}$ denote the restriction map of $\rG_F$ to $\rG_{\gp}$. 
\begin{enumerate}
\item There exists a unique  $[\eta_{\phi}^{(\gp)}]\in \rH^1(F,\phi)$ unramified outside $\gp$ and such that \[ \loc_{\gp}([\eta_{\phi}^{(\gp)}-\eta_{\mathbf{1}}]) \in \ker \left( \Hom(\rG_{\gp},\overline{\Q}_p) \to  \Hom(\rI_{\gp},\overline{\Q}_p) \right).\] 
\item Under the local Artin map, one has the following equality in $\Hom(F_\gp^\times,\overline{\Q}_p)$:
\[
 \loc_{\gp}([\eta_{\phi}^{(\gp)}-\eta_{\mathbf{1}}] ) = \sL^{\gp}(\phi) \ord_{\gp}.
\]
\item For all $\sigma\in \Sigma\setminus \Sigma_{\gp}$,  $\loc_{\gp}([\eta_{\phi}^{(\tilde{\sigma})}]) \in \overline{\Q}_p^\times \cdot \ord_{\gp}$.
\item One  has $\sL^{\gp}(\phi)  \ne 0$.
\end{enumerate}
\end{prop}
\begin{proof}
(i) By Lemma~\ref{lemma-21},  any $[\eta]\in \rH^1(F,\phi)$ unramified outside $\gp$ can be written as $\displaystyle \sum_{\sigma \in \Sigma_{\gp}} h_\sigma [\eta_{\phi}^{(\tilde{\sigma})}]$ for some $h_\sigma\in \overline{\Q}_p$. Moreover, the condition $\eta_{|\rI_{\gp}}=\eta_{\mathbf{1}|\rI_{\gp}}$ yields  $h_{\sigma}=1$ for all $\sigma\in \Sigma_{\gp}$ as required.

 (ii) We follow the proof  of \cite[Prop.~2.4]{BDP}.  Choose $u_0\in\cO_{H}[\tfrac{1}{w_0}]^\times$ having non-zero valuation  at $w_0$, and  $0$ at all other finite places.  
Then $u_{\phi}=\sum_{g\in \Gal(H/F)} \phi(g)  g^{-1}(u_0)$  is a basis of the line $(\cO_{H}[\frac{1}{\gp}]^\times\otimes \overline{\Q}_p)[\phi]$.  Consider the exact sequence 
\begin{align}\label{eq:26}
0\to \Hom(\rG_H,\overline{\Q}_p)
\xrightarrow{\left(\mathrm{loc}_{w_0},(\mathrm{loc}_{w}^\circ)_{w\ne w_0}\right)}
  \Hom\left(
 H_{w_0}^\times \times
\prod_{w_0 \ne w|p} \cO_{H_w}^\times , \overline{\Q}_p\right)\to \Hom\left(\cO_H\left[\tfrac{1}{w_0}\right]^\times,\overline{\Q}_p\right),  
\end{align}
coming from  localization at $p$ of the dual of the Artin reciprocity map and the diagonal inclusion, from which we deduce 
\begin{align}\label{Linveq}
0=\left(\eta_{\phi}^{(\gp)}-\eta_{\mathbf{1}}\right)(u_0)= \left(\mathrm{loc}_{w_0}(\eta_{\phi}^{(\gp)}-\eta_{\mathbf{1}}) \right) (u_0) + \sum_{w_0\ne w \mid p} \left(\mathrm{loc}^\circ_{w}(\eta_{\phi}^{(\gp)}-\eta_{\mathbf{1}}) \right) (u_0).\end{align}

It follows from (i) and the local-global compatibility in class field theory  that
\begin{align*}
&\ord_{w_0}(u_0)\cdot (\eta_{\phi}^{(\gp)}-\eta_{\mathbf{1}})(\Frob_{\gp})
= \mathrm{loc}_{w_0}(\eta_{\phi}^{(\gp)}-\eta_{\mathbf{1}})  (u_0)\overset{\eqref{Linveq}}{=} \sum_{w_0\ne w \mid p} \mathrm{loc}^\circ_{w}(\eta_{\mathbf{1}}-\eta_{\phi}^{(\gp)})  (u_0)\\
&=	\sum_{\sigma\in\Sigma_{\gp}} \sum_{\mathbf{1} \ne g \in \Gal(H/F)}(1-\phi(g))  \log_p( \iota_p \circ \tilde\sigma\circ g^{-1}(u_0)) +\sum_{\sigma\in \Sigma\setminus \Sigma_{\gp}} \sum_{g\in \Gal(H/F)}  \log_p(\iota_p \circ \tilde\sigma\circ g^{-1}(u_0)) 
 \\ & =\log_p\left( \rN_{H/\Q}(u_0) \right)	- \sum_{\sigma\in\Sigma_{\gp}} \sum_{g  \in \Gal(H/F)}(\phi(g))  \log_p( \iota_p \circ \tilde\sigma\circ g^{-1}(u_0)) =-\log_{\gp}(u_{\phi})= 
\ord_{w_0}(u_{\phi}) \cdot \sL^{\gp}(\phi), 
\end{align*}
where we have used that $\log_p\left( \rN_{H/\Q}(u_0)\right)=0$ as $\rN_{H/\Q}(u_0) \in \pm p^{\Z}$.

(iii) For each $\sigma\in \Sigma\setminus \Sigma_\gp$, we have seen in Lemma~\ref{lemma-21} that $\eta_{\phi}^{(\tilde{\sigma})}$ is unramified at $\gp$. A similar computation to the one above shows that
\[
\eta_{\phi}^{(\tilde{\sigma})}(\Frob_{\gp})
=-\sum_{g\in \Gal(H/F)} \phi(g)\log_p( \iota_p \circ \tilde\sigma\circ g^{-1}(u_0)),
\]
which  is non-zero by the Baker--Brumer Theorem, as the elements $( g^{-1}(u_0))_{g \in \Gal(H/F)}$ are $\Z$-linearly independent (this follows from examining the valuations at the $[H:F]$ places of $H$ above  $\gp$). 

(iv) The proof mixes ingredients all present in Gross \cite{gross81}. 
For  $w$ a  place of $H$ above $\gp$ and $u\in H_w^\times$,   we let 
$|u|^\circ_{w}=q_{w}^{-\mathrm{ord}_{w}(u)} \mathrm{N}_{H_w/\Q_p}(u)$ denote its  local absolute value, where $q_w$ is the order of the residue field 
 (see \cite[(1.8)]{gross81} where it is denoted by $||u||_{w,p}$).   Exactly as in 
\cite[Prop.~1.16]{gross81}, one checks using the Baker--Brumer Theorem,  the injectivity of the $\overline{\Q}$-linear  map: 
 \begin{align*}
 \lambda_{\gp}:  \cO_H\left[\tfrac{1}{\gp} \right]^{\times,\tau=-1} \otimes \overline{\Q} & \hookrightarrow \overline{\Q}_p[\Gal(H/F)] \\ 
 u \otimes 1 & \mapsto  \sum_{g \in \Gal(H/F)} \log_p(|\iota_p\circ g^{-1}(u)|^\circ_{w_0})[g]= \sum_{g \in \Gal(H/F)} \log_{\gp}(g^{-1}(u))[g],  
 \end{align*}
where  $\cO_H\left[\tfrac{1}{\gp} \right]^{\times,\tau=-1} $ is the submodule on which the complex conjugation $\tau$ acts by $-1$.
As $\lambda_{\gp}$ is $\Gal(H/F)$-equivariant, taking the $\phi$-isotypic component (note that  $\phi$ is  totally odd), yields
\[ 0\ne \lambda_{\gp}(u_{\phi})=\ord_{w_0}(u_0) \sL^{\gp}(\phi) 
\sum_{g\in \Gal(H/F)} \phi(g^{-1})[g]. \qedhere\]
\end{proof}

\begin{prop}\label{non-vanishing}
 Assume that there is a unique prime $\gp$ in $F$ above $p$, and that  
  the fields $F$ and $H$ are both Galois. Then $\sL(\phi)   + \sL(\phi^{-1})   \ne 0$.

\end{prop}

\begin{proof} As  $F$ and $H$ are Galois, the set of representatives 
$\{  \tilde\sigma| \sigma \in \Sigma_{\gp}\}$ defined in \eqref{sigma-tilde}
can be identified with the decomposition subgroup of $\Gal(H/\Q)$ at $w_0$. 
In particular, $\tilde{\sigma} (u_0)/u_0\in\cO_H^{\times}$ for all $\sigma \in \Sigma_{\gp}$.

For any $\sigma \in \Sigma_{\gp}=\Sigma\simeq\Gal(F/\Q)$ and $g \in \Gal(H/F)$, we let  $g_{\sigma}= \tilde\sigma g  \tilde\sigma^{-1}$. Using the notation from the proof of Proposition~\ref{L-invariant}(ii), we have: 
\begin{align*}
-\ord_{w_0}(u_0) \cdot \sL(\phi) & =  \sum_{\sigma \in \Sigma_{\gp}} \sum_{g\in \Gal(H/F)} \phi(g) \log_p( \iota_p \circ \tilde\sigma g^{-1} (u_0))   
=  \sum_{\sigma \in \Sigma} \sum_{g\in \Gal(H/F)} \phi(g) \log_p(\iota_p \circ g_{\sigma}^{-1} \tilde\sigma (u_0))  \\
&= \sum_{\sigma \in \Sigma} \sum_{g\in \Gal(H/F)} \phi(g) 
\log_p(\iota_p \circ g_{\sigma}^{-1} \left(\tilde\sigma (u_0)/u_0\right)) +  \sum_{\sigma \in \Sigma_{\gp}} \sum_{g\in \Gal(H/F)} \phi(g) \log_p( \iota_p \circ g_{\sigma}^{-1} (u_0 ))  \\
&=    \sum_{\sigma \in \Sigma_{\gp}}  \sum_{g\in \Gal(H/F)} \phi(g) \log_p(\iota_p \circ g_{\sigma}^{-1} (u_0)) = 
 \sum_{\sigma \in \Sigma}  \sum_{g\in \Gal(H/F)}  
 \phi(\tilde\sigma^{-1}g\tilde\sigma)   \log_p(\iota_p \circ g^{-1}(u_0)) 
\end{align*}
where the penultimum equality follows from  $ \tilde{\sigma} (u_0)/u_0 \in \cO_H^{\times}$ and the vanishing of $(\cO_H^{\times} \otimes_{\Z} \overline{\Q})[\phi(\tilde\sigma^{-1} \text{$\cdot$ }\tilde\sigma) ]$.

 We proceed by contradiction  assuming that $\sL(\phi) + \sL(\phi^{-1})=0$,  i.e., 
 \begin{align}\label{thingtocontradict}
\sum_{g\in \Gal(H/F)}
\log_p(\iota_p \circ g^{-1}(u_0))
 \sum_{\sigma \in \Sigma} 
\left( \phi(\tilde\sigma^{-1}g\tilde\sigma) +\phi(\tilde\sigma^{-1}g^{-1}\tilde\sigma)\right) =0.
\end{align}

Since the kernel of $\log_p$ on $\overline{\Q}^\times$ consists
of elements of the form $p^a \cdot u$, where $a \in \Q$ and $u$ is a root of unity, one can check using the valuations at 
places $w$ of $H$ dividing $p$ that the only possible non-trivial $\Q$-linear relation between these logarithms is 
\[\sum_{g\in \Gal(H/F)}\log_p(\iota_p  \circ g^{-1}(u_0))=0.\] 
The Baker--Brumer Theorem then implies  that this is the only possible non-trivial relation over $\overline{\Q}$.
As the $\overline{\Q}$-valued function $\displaystyle g\mapsto \sum_{\sigma \in \Sigma} \left( \phi(\tilde\sigma^{-1}g\tilde\sigma) +\phi(\tilde\sigma^{-1}g^{-1}\tilde\sigma)\right)$ is non-constant ({\it e.g.} the parity of $\phi$ implies that the values at $\mathbf{1}$ and $\tau$ differ), this  contradicts \eqref{thingtocontradict}. 
\end{proof}

\begin{lemma}\label{dimsplitcase}\ 	
\begin{enumerate}
		\item  $\dim_{\overline{\Q}_p} \ker\left(\rH^1(F,\phi) \xrightarrow{\loc_{\gp}} \rH^1(F_\gp,\overline{\Q}_p)\right)=\max(d-d_\gp-1,0)$. 
		\item  If $\Sigma_p^{\mathrm{irr}}= \{\gp\}$, then	
		$\displaystyle \dim_{\overline{\Q}_p} \ker\left( \rH^1(F,\phi) \to \bigoplus_{\gq \ne \gp} \rH^1(F_{\gq},\phi)\right)=d_{\gp}$.
	\end{enumerate}

\end{lemma}
\begin{proof} (i) By Lemma~\ref{lemma-21}, $\{ \eta_{\phi}^{(\tilde{\sigma})} \}_{  \sigma\in \Sigma\setminus \Sigma_{\gp}}$
 is a basis of  $\ker\left(\rH^1(F,\phi) \to \rH^1(\rI_\gp,\overline{\Q}_p)\right)$, proving the claim if $d=d_\gp$. 
Otherwise the dimension is at most $|\Sigma\setminus \Sigma_{\gp}|=d-d_\gp$ and in fact it is one less 
because $\rH^1(\rG_{\gp}/\rI_{\gp},\overline{\Q}_p)=\ker\left(\rH^1(F_{\gp},\overline{\Q}_p)\to \rH^1(\rI_{\gp},\overline{\Q}_p)\right)$ is a line 
in which any element of the above basis has a non-zero image by Proposition~\ref{L-invariant}(iii).  

(ii)  Again by Lemma~\ref{lemma-21}, we have 	$\dim_{\overline{\Q}_p} \ker\left(\rH^1(F,\phi) \to \bigoplus_{\gq \ne \gp} \rH^1(\rI_{\gq},\phi)\right)=d_{\gp}$. As $\rH^1(\rG_{\gq}/\rI_{\gq},\phi^{\rI_{\gq}})$ vanishes for all  primes  $\gq\in \Sigma_p\setminus \{\gp\}$ (see \cite[\S1.2]{DDP}),
the restriction map $\rH^1(F_{\gq},\phi)\to \rH^1(\rI_{\gq},\phi)$ is injective by the  inflation-restriction exact sequence, proving the claim.  
 \end{proof}
\subsection{Nearly-ordinary cuspidal deformation functor}\label{sec:42}
Let $\phi_1$ and $\phi_2$ be finite order  Hecke characters such that $\phi=\phi_2\phi_1^{-1}$ is totally odd and irregular at a {\it unique} prime $\gp$ above $p$, i.e., $ \Sigma_p^{\mathrm{irr}}=\{\gp \}$. The goal of this subsection is to introduce a nearly-ordinary deformation functor and compute its  tangent space. This  will be used in \S\ref{sec:24} and \S\ref{sec:34} to prove Theorem~\ref{main-thm}(i) and (iv), respectively. 
Let $\cA$ be the category of Artinian local $\overline{\Q}_p$-algebras $A$ with maximal ideal $\gm_A$ and residue field $\overline{\Q}_p$, where the morphisms are local homomorphisms of $\overline{\Q}_p$-algebras inducing identity on the residue field. 
Let $[\eta]$ be a  non-trivial cocycle of the $d_\gp$-dimensional space  $\ker \left(\rH^1(F,\phi^{-1}) \to \prod_{
\gq \in \Sigma_p\setminus\{\gp\}} \rH^1(F_\gq,\phi^{-1}) \right)$ (see Lemma~\ref{dimsplitcase}) and 
fix  a representative $\eta$  such that  $\eta(\tau)=0$. As the centralizer of the image of $\rho=\left(\begin{smallmatrix} \phi_1 & \phi_2 \eta \\ 0 & \phi_2 \end{smallmatrix}\right)$  consists only of scalar matrices, the functor $\cD_{\eta}$ assigning to $A\in  \cA$ the set of lifts $\rho_A:\rG_F\to \GL_2(A)$ of $\rho$ modulo strict equivalence is pro-representable by a 
complete Noetherian local $\overline{\Q}_p$-algebra $\cR^{\univ}_\eta$ (see \cite{Mazur}). 
\begin{defn}\label{defnno}
Let $\cD^{\nord}_{\eta}$ be the functor assigning to  $A\in  \cA$ the set of couples $(\rho_A, \Fil_A)$ such that  
\begin{enumerate}
\item $\rho_A: \rG_F\to \GL_2(A)$ is a continuous representation lifting $\rho$, 
\item for $\gq\in \Sigma_p\setminus \{\gp\}$, $\rho_{A|\rG_{\gq}}$ has a $\rG_{\gq}$-stable quotient, 
free of rank $1$ as $A$-module,  lifting $\phi_1$,
 \item  $\Fil_A\subset A^2$ is a free direct factor over $A$ of rank $1$  which is $\rG_{\gp}$-stable, 
 \end{enumerate}
modulo strict equivalence relation  $(\rho_A,\Fil_A)\sim (M\rho_AM^{-1}, M\cdot\Fil_A)$ for $M\in 1+\rM_2(\gm_A)$. 
Let $\cD^{\ord}_{\eta} \subset \cD^{\nord}_{\eta}$ be the subfunctor of ordinary deformations, i.e., those for which the inertia subgroups act trivially on the free rank $1$ quotients in (ii) and (iii)  above.
\end{defn}
\begin{prop}\label{nord_functor}
We fix  a representative 
$ \rho_{\univ}=\left(\begin{smallmatrix} a&b\\ c& d\end{smallmatrix}\right) : \rG_F\to  \GL_2(\cR^{\univ}_\eta)$ 
of the universal deformation of $\rho$ and denote by $(e_1,e_2)$ the canonical basis of $(\cR^{\univ}_\eta)^2$. 
 \begin{enumerate}
\item  The functor $\cD^{\nord}_{\eta}$ is pro-representable by a complete Noetherian local $\overline{\Q}_p$-algebra $\cR^{\nord}_{\rho_{\univ}}$, which is  the  quotient of $\cR^{\univ}_\eta\lsem Y\rsem$ by the ideal $\mathscr{J}^{\nord}_{( \rho_{\univ},Y)}$ defined  in \eqref{nearlyordinaryidealatp} below. The universal deformation is given by the pushforward  $\rho^{\nord}$  of $ \rho_{\univ}$ along $\cR^{\univ}_\eta \hookrightarrow \cR^{\univ}_\eta\lsem Y\rsem  \twoheadrightarrow  \cR^{\nord}_{\rho_{\univ}}$ together with  the universal nearly-ordinary line $\cR^{\nord}_{\rho_{\univ}}\cdot (e_1+ Y e_2)$ at $\gp$.
Changing the representative $\rho_{\univ}$ with $\rho_{\univ}'=M \rho_{\univ} M^{-1}$ for $ M=\left( \begin{smallmatrix}
\alpha & \beta \\ \gamma & \delta \end{smallmatrix} \right) \in \ker \left(\GL_2(\cR^{\univ}_{\eta}) \to \GL_2(\overline{\Q}_p) \right)$ induces an isomorphism $\cR^{\nord}_{\rho_{\univ}} \simeq \cR^{\nord}_{\rho'_{\univ}} $ of $\cR^{\univ}_\eta$-algebras
sending $Y$ to $\frac{\gamma + \delta Y}{\alpha + \beta Y}$.

\item $\cD^{\ord}_{\eta}$ is pro-representable by a complete Noetherian local $\overline{\Q}_p$-algebra $\cR^{\ord}_\eta$ which is the  quotient of $\cR^{\univ}_\eta$ by the ideal $\mathscr{J}^{\ord}$ defined  in \eqref{ordinaryidealatp} below, and which is  independent from the particular choice of  $\rho_{\univ}$.

\item We have an injective map for the tangent space 
  $t^{\nord}_{\eta}$ of $\cR^{\nord}_{\rho_{\univ}}$  given by  
\[\Hom_{\overline{\Q}_p-\mathrm{alg}}(\cR^{\nord}_{\rho_{\univ}},\overline{\Q}_p[\varepsilon]) \hookrightarrow \rH^1(F,\ad\rho) \oplus \overline{\Q}_p, 
\quad \varphi \mapsto \left(\left[\frac{\varphi(\rho^{\nord}) \rho^{-1} -\mathbf{1}_2}{\varepsilon} \right], \frac{\varphi(Y)}{\varepsilon}\right),\] 
where by an abuse of notation  $\frac{\lambda \varepsilon}{\varepsilon}$ simply means $\lambda \in \overline{\Q}_p$.
\end{enumerate}
\end{prop} 
\begin{proof} 
We show (i) and (ii) simultaneously by  analyzing the local deformation conditions at the primes above $p$, starting with those at the unique irregular prime $\gp$. We have 
 \begin{align}\label{Y-line}
\begin{pmatrix}
1&0\\ -Y& 1\end{pmatrix} \begin{pmatrix} a(g)&b(g)\\ c(g)& d(g)\end{pmatrix}\begin{pmatrix} 1&0\\ Y& 1\end{pmatrix}=
\begin{pmatrix} a(g)+b(g)Y & b(g) \\ c(g)+(d(g)-a(g))Y-b(g)Y^2 & d(g)-b(g)Y\end{pmatrix}
\end{align}
leading us to consider the following  ideals  of $\cR^{\univ}_{\eta}\lsem Y \rsem$:
\begin{align}
\begin{split}
\cI^{\nord}_{(\rho_{\univ},Y,\gp)}&=\left( c(g)+(d(g)-a(g))Y-b(g)Y^2;  g \in  \mathrm{G}_{\gp} \right) \\  \cI^{\ord}_{(\rho_{\univ},Y,\gp)}&= \left(d(g)-1-b(g)Y;  g \in \mathrm{I}_{\gp} \right) + \cI^{\nord}_{(\rho_{\univ},Y,\gp)}. 
\end{split}
\end{align}
As $[\eta]$ is non-trivial and unramified outside $\gp$, it is necessarily ramified at $\gp$, guaranteeing the existence of 
   $g_\gp\in  \mathrm{I}_{F_\gp}$ such that $\eta(g_\gp)\neq 0$. This implies that  $b(g_\gp)\in (\cR^{\univ}_{\eta})^\times$, hence 
  $ Y  \equiv \tfrac{d(g_\gp)-1}{b(g_\gp)} \mod \cI^{\ord}_{(\rho_{\univ},Y,\gp)}$ and 
 $\cR^{\univ}_{\eta}\lsem Y \rsem/ \cI^{\ord}_{(\rho_{\univ},Y,\gp)}    \simeq \cR^{\univ}_{\eta}/  \mathscr{J}^{\ord}_{( \rho_{\univ},\gp)}$, 
  where   $\mathscr{J}^{\ord}_{( \rho_{\univ},\gp)}$ is the image of $\cI^{\ord}_{(\rho_{\univ},Y,\gp)} $ via the homomorphism 
  $\cR^{\univ}_{\eta}\lsem Y \rsem\to \cR^{\univ}_{\eta}$ sending $Y$ to  $\tfrac{d(g_\gp)-1}{b(g_\gp)}$. 
  
  Letting  $M=\left( \begin{smallmatrix}
\alpha & \beta \\ \gamma & \delta \end{smallmatrix} \right) \in \ker \left(\GL_2(\cR^{\univ}_{\eta}) \to \GL_2(\overline{\Q}_p) \right)$ and $ \rho'_{\univ}(g)=\left(\begin{smallmatrix}
a'(g)&b'(g)\\ c'(g)& d'(g)\end{smallmatrix}\right)=M \rho_{\univ}(g) M^{-1}$ another representative of the universal deformation,  a direct calculation shows that 
\begin{align}\label{representativedependenceidealcondp}
\cI^{\nord}_{(\rho_{\univ},Y,\gp)}=\cI^{\nord}_{( \rho'_{\univ}, \frac{\gamma + \delta Y}{\alpha + \beta Y},\gp)} \text{ and } \cI^{\ord}_{(\rho_{\univ},Y,\gp)}=\cI^{\ord}_{( \rho'_{\univ}, \frac{\gamma + \delta Y}{\alpha + \beta Y},\gp)} .
\end{align}
It follows that the ideals   $\mathscr{J}^{\ord}_{( \rho_{\univ},\gp)}$ and $\mathscr{J}^{\ord}_{( \rho'_{\univ},\gp)}$ are equal, hence their common value, 
 denoted by $\mathscr{J}^{\ord}_\gp$, defines a deformation condition in the sense of Ramakrishna. 

Let $\gq \in \Sigma_p \setminus \{ \gp \}$ and choose  $ g_\gq \in \rG_{\gq}$ such that $\phi_1(g_\gq) \ne \phi_2(g_\gq)$. 
By Hensel's Lemma and \cite[Lem.~1.4.3]{bel-che09}, there exists a matrix  $M=\left( \begin{smallmatrix}
\alpha & \beta \\ \gamma & \delta \end{smallmatrix} \right) \in  \GL_2(\cR^{\univ}_{\eta})$ with $\gamma \in \gm_{\cR^{\univ}_{\eta}}$ such that $M \rho_{\univ}(g_\gq) M^{-1}=\left(\begin{smallmatrix}
\alpha_\gq& 0 \\ 0 & \beta_\gq \end{smallmatrix}\right)$. As $\res_{\gq}([\eta])=0$, it follows that
 $\alpha_\gq - \phi_1(g_\gq)\in \gm_{\cR^{\univ}_{\eta}}$ and $\beta_\gq - \phi_2(g_\gq)\in \gm_{\cR^{\univ}_{\eta}}$. This also implies  that $b'(g) \in \gm_{\cR^{\univ}_{\eta}}$ for all $g \in \rG_{\gq}$, where $\left(\begin{smallmatrix}
a'(g)&b'(g)\\ c'(g)& d'(g)\end{smallmatrix}\right)=M \rho_{\univ}(g) M^{-1}$.   The  ideals 
\[\mathscr{J}^{\nord}_\gq=\left( b'(g);  g \in  \mathrm{G}_{\gq} \right) \text{ and }  \mathscr{J}^{\ord}_\gq= \left(a'(g)-1;  g \in \mathrm{I}_{\gq} \right) + \mathscr{J}^{\nord}_\gq \]
of $\cR^{\univ}_{\eta}$ do not  depend  on the particular choice of $M$ (nor on the choice of $\rho_{\univ}$), as any other choice would differ by a diagonal matrix of $\GL_2(\cR^{\univ}_{\eta})$ and  conjugating  by it would not affect these ideals. 
 Therefore, $\cD^{\ord}_{\eta}$ is representable by  $\cR^{\ord}_\eta=\cR^{\univ}_\eta/\mathscr{J}^{\ord} $, where 
\begin{align}\label{ordinaryidealatp}
\mathscr{J}^{\ord}= \sum_{\gq \mid p}\mathscr{J}^{\ord}_\gq, 
\end{align}
and similarly, $\cD^{\nord}_{\eta}$ is representable by  $\cR^{\nord}_{\rho_{\univ}}=\cR^{\univ}_\eta \lsem Y \rsem/\mathscr{J}^{\nord}_{( \rho_{\univ},Y)}$, where 
\begin{align}\label{nearlyordinaryidealatp} 
\mathscr{J}^{\nord}_{( \rho_{\univ},Y)} = \cI^{\nord}_{(\rho_{\univ},Y,\gp)} + \sum_{\gq \in \Sigma_p \setminus \{\gp \}} \mathscr{J}^{\nord}_\gq. 
\end{align} 
The universal nearly-ordinary deformation $\rho^{\nord}:\rG_F \to \GL_2(\cR^{\nord}_{\rho_{\univ}})$ is the pushforward of $\rho_{\univ}$ along  $\cR^{\univ}_\eta \hookrightarrow \cR^{\univ}_\eta \lsem Y \rsem \twoheadrightarrow \cR^{\nord}$. To see this,  let $x$ be a point of $\cD_\eta^{\nord}(A)$ which is then represented by a push-forward  $\rho_A$ of $ \rho_{\univ}$ along the homomorphism 
$\varphi_A: \cR^{\univ}_\eta \to A$,  and a $\gp$-nearly-ordinary line $F_A\subset A^2$. The $\gq$-nearly-ordinary lines for $\gq \in \Sigma_p \setminus \{\gp\}$ are uniquely determined by $\rho_A$ thanks to Hensel's lemma, i.e.,  when a $\gq$-nearly-ordinary line exists it is necessarily unique. 
The filtration $F_A$ has  a basis  $e_1+y e_2$ with $y\in \gm_A$, because $F_A \otimes_{A} \overline{\Q}_p = e_1 \cdot \overline{\Q}_p$ (since $[\eta]$ is ramified at $\gp$). 
Let $\widetilde\varphi_A: \cR^{\univ}_{\eta}\lsem Y\rsem \to A$ be the homomorphism extending $\varphi_A$ and sending $Y$ to $y$. By construction, $\widetilde\varphi_A$ factors through $\cR^{\nord}_{\rho_{\univ}}$, and even 
 $\varphi_A$ factors  through $\cR^{\ord}_{\eta}$ if $x \in \cD_\eta^{\ord}(A)$. This shows the desired representability and finishes the proof of (i) and (ii).  
 
(iii) By Mazur's deformation theory $\Hom_{\overline{\Q}_p-\mathrm{alg}}(\cR^{\univ}_\eta,\overline{\Q}_p[\varepsilon])=\rH^1(F,\ad\rho)$. 
As $\cR^{\univ}_\eta\lsem Y\rsem$ is formally smooth over $\cR^{\univ}_\eta$ of relative dimension $1$, its tangent space is given by
\[ 
\Hom_{\overline{\Q}_p-\mathrm{alg}}(\cR^{\univ}_\eta\lsem Y\rsem ,\overline{\Q}_p[\varepsilon]) = \rH^1(F,\ad\rho) \oplus \Hom_{\overline{\Q}_p-\mathrm{alg}}(\overline{\Q}_p\lsem Y\rsem ,\overline{\Q}_p[\varepsilon]). 
\]
Finally $\cR^{\nord}_{\rho_{\univ}}$ is a quotient of $\cR^{\univ}_\eta\lsem Y\rsem$,  hence  its tangent space is a subspace of the above. 
\end{proof} 
For any $v \in \Sigma_p$, we let $\chi_{v}$ be the character acting on the $\rG_{v}$-stable free rank $1$ quotient of $\rho^{\nord}$.
The universal deformation ring of the character  $\phi_{1 \mid \mathrm{I}_{F_v} }:\mathrm{I}_{F_v} \to \overline{\Q}_p$ is given by $\overline{\Q}_p\lsem X_\sigma :\sigma \in \Sigma_v \rsem$. As  $\chi_{v}$ lifts $\phi_{1 \mid \mathrm{I}_{F_v} }$,  this endows $\cR^{\nord}_{\rho_{\univ}}$ a structure of $\varLambda^{\nord}=\varLambda\lsem X_\sigma : \sigma\in  \Sigma \rsem$-algebra, where the $\varLambda$-structure is the one associated to the determinant of $\rho^{\nord}$, because $\varLambda$ is the universal deformation ring of the character  $\det(\rho):\rG_F \to \overline{\Q}_p^\times$. This shows that 
\begin{align}\label{ordinaryaugmentationideal}
\cR^{\nord}_{\rho_{\univ}}/(X_\sigma: \sigma\in  \Sigma) = \cR^{\ord}_{\eta}.
\end{align}

Henceforth we fix a $\rho_{\univ}$ and we compute  $ t^{\nord}_{\eta}$. 
Fixing a  basis $(e_1,e_2)$ of $V_\rho =\overline{\Q}_p^2 $ such that $\rho(\tau)=\left(\begin{smallmatrix} 1 & 0\\ 0 & -1 \end{smallmatrix}\right)$ allows us to 
identify $\ad V_\rho$ with $\rM_2(\overline{\Q}_p)$. As $\rho$ is upper triangular,  the  subspace
$W_\rho\subset \ad V_\rho$ of  upper triangular matrices fits in the following non-split short  exact sequence of $\overline{\Q}_p[\rG_F]$-modules
\begin{align}\label{tgsanscond}
0\rightarrow W_\rho \rightarrow \ad V_\rho \xrightarrow{\mathbf{c}} \phi \rightarrow 0,
\end{align}
where the map $\mathbf{c}$ sends $\left(\begin{smallmatrix} a & b \\ c & d \end{smallmatrix}\right)$ to $c$. 
Restricting to the decomposition group $\rG_{\gp}$ yields the  map 
\[
\mathbf{c}_{\gp}:\rH^1(F,\ad V_\rho) \xrightarrow{\mathbf{c}_*} \rH^1(F,\phi)\xrightarrow{\loc_{\gp}} \rH^1(F_{\gp},\overline{\Q}_p).
\]
In addition, for each regular prime  $\gq$, the property $[\eta]_{|\rG_{\gq}}=0$ implies that there exists $\lambda_{\gq}\in \overline{\Q}_p$
 such that $\eta_{|\rG_{\gq}}=\lambda_{\gq}(\phi^{-1}-1)$ and $\ad V_{\rho \mid \rG_{\gq}} \simeq \phi \oplus \phi^{-1} \oplus \overline{\Q}_p^2$. 
One deduces a homomorphism of $\overline{\Q}_p[\rG_{\gq}]$-modules 
\begin{align}\label{eq:map_B_q} 
\ad V_\rho \xrightarrow{\mathbf{b}_{\gq}} \phi^{-1},\qquad\left(\begin{smallmatrix} a & b \\ c & d \end{smallmatrix}\right)
\mapsto b+\lambda_{\gq}(d-a)-\lambda_{\gq}^2 c,
\end{align}
 allowing us to define the map 
\[ 
\mathbf{b}_{\gq}:\rH^1(F,\ad V_\rho) \xrightarrow{\loc_{\gq}} \rH^1(F_{\gq},\ad V_\rho)\xrightarrow{\mathbf{b}_{\gq}} \rH^1(F_{\gq},\phi^{-1}). 
\]
It follows from Proposition~\ref{nord_functor}(iii) that one has an inclusion 
$ t^{\nord}_{\eta} \hookrightarrow \mathbf{K}_1  \oplus  \overline{\Q}_p$,  where 
\begin{align}\label{eq:nord_tangent_space}
\mathbf{K}_1 =\ker\left(\rH^1(F,\ad V_\rho)\xrightarrow{(\mathbf{c}_{\gp},(\mathbf{b}_{\gq})_{\gq \in \Sigma_p \setminus \{\gp\}})} \rH^1(F_{\gp},\overline{\Q}_p) \oplus \bigoplus_{\gq \in \Sigma_p \setminus \{\gp\}} \rH^1(F_{\gq},\phi^{-1}) \right).
\end{align}
As $\ker(\mathbf{c}_{\gp})\supset \rH^1(F,W_\rho)$, we have 
\[
\displaystyle \mathbf{K}_2=\mathbf{K}_1  \cap \rH^1(F,W_\rho)=
\ker\left(\rH^1(F,W_\rho)\xrightarrow{(\mathbf{b}_{\gq})_{\gq \in \Sigma_p \setminus \{\gp\}}} \bigoplus_{\gq \in \Sigma_p \setminus \{\gp\}} \rH^1(F_{\gq},\phi^{-1}) \right).
\] 
In order to find an upper bound for the dimension of $\mathbf{K}_1 $ we establish  some intermediate calculations. 
The short exact sequence of $\overline{\Q}_p[\rG_F]$-modules
\[
0\rightarrow \phi^{-1} \rightarrow W_\rho  \rightarrow \overline{\Q}_p^2 \rightarrow 0,
\]
where the map $W_\rho  \twoheadrightarrow \overline{\Q}_p^2$ is given by $\left(\begin{smallmatrix} a & b \\ 0 & d\end{smallmatrix}\right) \mapsto (a,d)$, yields a long exact sequence
\begin{align}\label{long}
0\rightarrow \rH^{0}(F, W_\rho ) \rightarrow \rH^{0}(F, \overline{\Q}_p^2)\xrightarrow{\delta}\rH^1(F,\phi^{-1}) \rightarrow \rH^1(F, W_\rho ) \rightarrow \rH^1(F, \overline{\Q}_p^2) \rightarrow \rH^{2}(F,\phi^{-1}).
\end{align}
\begin{lemma}\label{49}
Suppose that $\delta_{F,p}=0$.  
\begin{enumerate}
\item One has $\dim_{\overline{\Q}_p} \rH^1(F, W_\rho)=d+1$ and $\rH^2(F,W_\rho )=\rH^2(F,\phi)=\{0\}$.
\item The image of the connecting homomorphism $\delta:\rH^{0}(F, \overline{\Q}_p^2)\to\rH^1(F,\phi^{-1}) $ is  generated by \[ [\eta] \in \ker \left(\rH^1(F,\phi^{-1}) \to \prod_{
\gq \in \Sigma_p \setminus \{\gp\}}
\rH^1(F_\gq,\phi^{-1}) \right)\]
defining the residual representation $\rho$.
\end{enumerate}
\end{lemma}
\begin{proof} (i) Lemma~\ref{lemma-21} and the global Euler characteristic formula shows  $\rH^2(F,\phi)=\{0\}$. 
Counting dimensions in \eqref{long} and again  applying  the global Euler characteristic formula, proves the rest. 

 (ii) Let $\rho'=\left(\begin{smallmatrix}
0 & \eta' \\ 0 & 0\end{smallmatrix}\right) \in  Z^1(F, W_\rho )$, where  $\eta' \in  Z^1(F,\phi^{-1})$ represents an element in the image of $\delta$. It follows from  \eqref{long} that the cohomology class $[\rho']$ vanishes, i.e.,  $\rho'=\rho M \rho^{-1}-M$ for some upper triangular matrix $M \in W_\rho$. A direct computation then shows that  $[\eta']\in \overline{\Q}_p\cdot [\eta]$.
\end{proof}
\begin{prop}\label{tgnord}
Suppose that $\delta_{F,p}=0$. 
\begin{enumerate} 
\item Then $\dim_{\overline{\Q}_p} \mathbf{K}_2  \leqslant d_{\gp}+1$.
\item Suppose moreover that $d_{\gp}<d$.  Then $\dim_{\overline{\Q}_p} \mathbf{K}_1  \leqslant d$, in particular $\dim_{\overline{\Q}_p} (t^{\nord}_{\eta}) \leqslant d+1$. 
\end{enumerate}
\end{prop}
\begin{proof}
(i)
By Lemma~\ref{49}, the long exact sequence \eqref{long} gives rise  to a short exact sequence
\begin{align}\label{eq:seq_w_rho}
0 \to \rH^1(F,\phi^{-1})/(\overline{\Q}_p\cdot[\eta]) \xrightarrow{\varphi} \rH^1(F, W_\rho ) \rightarrow \rH^1(F, \overline{\Q}_p^2) \rightarrow 0.
\end{align}
By Lemma~\ref{dimsplitcase}(ii)  the space $\displaystyle 
\ker\left( \rH^1(F,\phi^{-1}) \to \bigoplus_{\gq \in \Sigma_p \setminus \{\gp\}} \rH^1(F_{\gq},\phi^{-1})\right)$
is $d_{\gp}$-dimensional and contains the line $\overline{\Q}_p\cdot  [\eta]$. As the map $[\eta']\mapsto \left[\left(\begin{smallmatrix}
0& \eta' \\ 0& 0\end{smallmatrix}\right)\right]$ identifies the latter with $\mathbf{K}_2\cap \mathrm{im} (\varphi)$, we deduce that
 $\dim_{\overline{\Q}_p} \left(\mathbf{K}_2\cap \mathrm{im} (\varphi)\right)=d_{\gp}-1$.
It  follows from 
 \eqref{eq:seq_w_rho} and $\delta_{F,p}=0$ that $\dim_{\overline{\Q}_p}(\mathbf{K}_2) \leqslant (d_{\gp}-1)+2\cdot 1=d_{\gp}+1$. 
 
(ii)  Clearly $\mathbf{K}_1 /\mathbf{K}_2$ injects in $\ker\left(\rH^1(F, \phi) \to \rH^1(F_{\gp}, \overline{\Q}_p)\right)$. Using (i) and Lemma~\ref{dimsplitcase}(i), we find 
 $\dim_{\overline{\Q}_p} (\mathbf{K}_1 ) = \dim_{\overline{\Q}_p}(\mathbf{K}_2)+ \dim_{\overline{\Q}_p} \left(\mathbf{K}_1 /\mathbf{K}_2\right) \leqslant (d_{\gp}+1)+(d-d_{\gp}-1)=d$.  The  second assertion then follows from \eqref{eq:nord_tangent_space}. 
\end{proof}

\section{Local geometry of the eigenvariety}\label{sec:03}
In this section we prove Theorem~\ref{main-thm}. We assume $\Sigma_p^{\mathrm{irr}}=\{\gp\}$ and $\delta_{F,p}=0$, 
and we recall the basis $\eta_{\mathbf{1}}\in \Hom(\rG_F,\overline{\Q}_p)$ from  \S\ref{sec:21}.

\subsection{Nearly-ordinary pseudo-characters and generalized matrix algebras}\label{sec:gma}
Recall that  $\cT^{\cusp}$ (resp.~$\cT^{\nord}$, $\varLambda$, $\varLambda^{\nord}$) denotes the  complete local ring of $\cC_{\cusp}$ (resp.~$\cE$, $\cW^{\parallel}$,  $\cW\times\cW^{\parallel}$)  at $f=f_{\phi_1}$ (resp.~$f$, $\sw(f)$,  $\kappa(f)$),  whose maximal ideal is denoted by $\gm$ (resp.~${\gm}_{\nord}$, $\gm_{\varLambda}$,  $\gm_{\varLambda^{\nord}}$). As $\cT^{\cusp}$  (resp.~$\cT^{\nord}$) is  reduced, finite and torsion-free over  $\varLambda$ (resp.~ $\varLambda^{\nord}$), generated  by  $\{T_{\gl},U_v : \gl\nmid \gn p, v\in \Sigma_p\}$,  it is  equidimensional of dimension $1$ (resp.~ $d+1$).
We remind  that $U_v$ is the normalized Hecke operator specializing to Hida's $U_v^\circ$ in classical weights.  
There exists a continuous two-dimensional pseudo-character $\mathrm{Ps}_{\cE}: \rG_F \rightarrow \cO(\cE)$ such that for all primes $\gl \nmid p\gn$ we have $\mathrm{Ps}_{\cE}(\Frob_{\gl})= T_{\gl}$.  Let $\mathrm{Ps}_{\cT^{\nord}}: \rG_F \rightarrow \cT^{\nord}$  be the push-forward of $\mathrm{Ps}_{\cE}$ by  $\cO(\cE) \rightarrow \cT^{\nord}$. 
As $\cT^{\nord}$ is reduced its  total quotient ring  $\mathbf{L}$  is a product of fields indexed by the finite set of its minimal primes.  
Recall the that  $\phi=\phi_2\phi_1^{-1}$ is totally odd, i.e., $\phi(\tau)=-1$ for  any complex conjugation $\tau$.  
By \cite[\S2]{wiles} there exists a generically irreducible  continuous  representation 
\[
\rho_{\mathbf{L}}= \left(\begin{smallmatrix} a & b \\ c & d \end{smallmatrix}\right):\rG_F\to \GL_2(\mathbf{L}), 
\]
 such that $\mathrm{tr}(\rho_{\mathbf{L}})=\mathrm{Ps}_{\cT^{\nord}}$, $\det(\rho_{\mathbf{L}})=\phi_1\phi_2 \chi_{\cyc}$, where $\chi_{\cyc}:\rG_F\to \varLambda^\times$ is the universal cyclotomic character.
Hida's  nearly-ordinary cuspidal  Hecke algebra $\mathfrak{h}^{\nord}$ of tame level $\gn$ introduced in \cite{hida90} 
is a canonical integral formal model of the admissible open of $\cE$ defined as the locus where $ |U_v|_p=1$ for all $v\in \Sigma_p$. 
Hence $\cT^{\nord}$ is the strict completed local ring of $\mathfrak{h}^{\nord}$ at $f$ and  
the irreducible components of $\Spec(\cT^{\nord})$, i.e.,  the minimal prime ideals of $\cT^{\nord}$  
are in bijection with  the nearly-ordinary cuspidal Hida eigenfamilies containing $f$.  
It follows from \cite[Thm.~I]{hida90} that $\rho_{\mathbf{L}}$ is nearly-ordinary, i.e., for every $v\in \Sigma_p$ there exists 
$\mathbf{Fil}_{\mathbf{L},v} \in \mathbf{P}^1(\mathbf{L})$ such that $\mathbf{Fil}_{\mathbf{L},v}$ is stable by $\rG_{v}$ and the action on the quotient is by a character 
$\phi_1\chi_{v}:\rG_v \to  (\cT^{\nord})^\times$  whose precomposition with the Artin reciprocity map $\rec_{v}: F_{v}^\times \rightarrow \rG^{\mathrm{ab}}_{v}$  sends any uniformizer $\varpi_v \in F_{v}^\times$ to the corresponding Hecke operator  $U_{\varpi_v}=\left[\left(\begin{smallmatrix} \varpi_v & 0 \\ 0 & 1 \end{smallmatrix}\right)\right]$. As $U_{v} f= \phi_1(v) f$, precomposing the restriction of $\chi_v$ to the inertia subgroup of  $\rG^{\mathrm{ab}}_{v}$ with the local Artin reciprocity map yields a homomorphism $\cO_{F_v}^\times \to  (\cT^{\nord})^\times$  which is congruent to $\mathbf{1}$ modulo 
${\gm}_{\nord}$ and is therefore  trivial on the torsion part of $\cO_{F_v}^\times $. 
Choosing a basis of the free part of $\cO_{F_v}^\times$, we deduce that  $\chi_v:\rI_v\to (\cT^{\nord})^\times$ factors through the
tautological character 
\begin{align}\label{eq:no-char}
\chi_v:\rI_v \to \cO_{F_v}^\times  \to  \overline{\Q}_p\lsem X_\sigma : \sigma\in \Sigma_v\rsem^\times,
\end{align}
yielding a natural continuous $\overline{\Q}_p$-algebra homomorphism $\overline{\Q}_p\lsem X_\sigma : \sigma\in \Sigma_v\rsem \to \cT^{\nord}$.
Recall also that the determinant map yields a continuous $\overline{\Q}_p$-algebra homomorphism $\varLambda=\overline{\Q}_p\lsem X\rsem \to \cT^{\nord}$,  where the identification $\cO\lsem X\rsem=\cO\lsem \Gal(F_\infty/F)\rsem$ is induced by the fixed topological generator $\gamma$ of $\Gal(F_{\infty}/F)$. 
The density of classical points in $\cE$ and the compatibility between the local and global Langlands correspondences for $\GL_2$ at places above $p$,  implies that the resulting injective homomorphism
\begin{align}\label{loc-glob-wt} 
\varLambda^{\nord}=\varLambda\lsem X_\sigma : \sigma\in \Sigma \rsem \to \cT^{\nord}. \end{align}
 is precisely the one given by  weight map $\sw: \cE \to \cW\times\cW^{\parallel}$. 
By \cite[Lem.~4.3, Lem.~4.6]{DKV}, there exists $\gamma_0 \in \rG_F$ and a basis $(e_1,e_2)$ of $\mathbf{L}^2$ such that the following properties hold:
 \begin{enumerate}[label=\textbf{(P\arabic*)}]
 \item \label{(P1)} $\phi(\gamma_0) \ne 1$ and $\rho_{\mathbf{L}}(\gamma_0)=\left(\begin{smallmatrix} \phi_1(\gamma_0)  & 0 \\ 0 & \phi_2(\gamma_0) \end{smallmatrix}\right)$.
  \item \label{(P2)}  $\rho_{\mathbf{L}}= \left(\begin{smallmatrix} a & b \\ c & d \end{smallmatrix}\right)$ such that $a(\rG_F)\subset \cT^{\nord}, \quad d(\rG_F)\subset \cT^{\nord}, \quad  b(\rG_F)c(\rG_F)\subset {\gm}_{\nord}$. 
 \item \label{(P3)} For all $v \in \Sigma_p$, $\mathbf{Fil}_{\mathbf{L},v}$ is in general position in the basis $(e_1,e_2)$, i.e., $\mathbf{Fil}_{\mathbf{L},v}=  \mathbf{L} \cdot ( e_1+ y_v e_2)\subset \mathbf{L}^2$ with $y_v \in \mathbf{L}^\times$. Thus, for all $g\in \rG_v$ one has 
\[ y_v  \cdot b(g) = d(g)-\phi_1\chi_{v}(g) \text{ and }  y_v^{-1} \cdot c(g)= a(g)- \phi_1\chi_{v}(g).  \]
 \end{enumerate}
Let $\mathbf{B}$ (resp.~$\mathbf{C}$) be the $\cT^{\nord}$-module generated by $b(\rG_F)$ (resp.~$c(\rG_F)$). As $\rho_{\mathbf{L}}$ is absolutely irreducible, it follows from \cite[Prop.~2]{bellaiche-chenevier-eis} that  $\mathbf{B}$ and $\mathbf{C}$ are torsion-free and of finite type,  hence they  can be endowed  with the locally convex topology.
\begin{prop}\label{the_module_B}
Assume $\Sigma_p^{\mathrm{irr}}=\{\gp\}$.  There is a natural injection 
\begin{align}\label{entryB}
	\Hom(\mathbf{B}/{\gm}_{\nord} \mathbf{B}, \overline{\Q}_p) \hookrightarrow \ker\left(\rH^1(F,\phi^{-1}) \to \bigoplus_{\gq \ne \gp} \rH^1(F_{\gq},\phi^{-1})\right), \quad h\mapsto \left[g\mapsto \frac{h(b(g))}{\phi_2(g)}\right]. 
\end{align}
Moreover, $y_\gp \cdot \mathbf{B} \subset \gm_{\nord}$. 
\end{prop}
\begin{proof} 
The first assertion follows by the same arguments as in \cite[Thm.~4.2 \and (113)]{DDP} using the fact, established in  \cite[Prop.~3]{bellaiche-chenevier-eis},  that the cocycles in the image of  $\Hom_{\cT^{\nord}}(\mathbf{B}/{\gm}_{\nord} \mathbf{B}, \overline{\Q}_p) \hookrightarrow \rH^1(F,\phi^{-1})$ are continuous. In virtue of \cite[Lem.~4.8]{DKV}, $\mathbf{B}$ is generated by $\{b(g) \mid g \in \rI_{\gp}\}$ and hence  $y_\gp \cdot  \mathbf{B} \subset {\gm}_{\nord}$ by \ref{(P3)}  noting the fact that  $\phi_{\mid \rG_\gp}=\mathbf{1}$. 
\end{proof}

By rescaling $(e_1,e_2)$, we henceforth  assume that $y_\gp=1$. The above proposition then yields the following additional  non-formal property
of our GMA:
 \begin{enumerate}[label=\textbf{(P\arabic*)}]   \setcounter{enumi}{3}
 \item \label{(P4)}  $\mathbf{B} \subset \gm_{\nord}$. 
 \end{enumerate}
 
 This property and the  fact that the $\cT^{\nord}$-module $\mathbf{C}$ may  not  in general be contained in ${\gm}_{\nord}$ will play an important role in the proof of Theorem~\ref{main-thm}.

\subsection{Regularity and an $R=T$  Theorem}\label{sec:24}
In this subsection we prove Theorem~\ref{main-thm}(i) by establishing an $R=T$ theorem. 
It will become clear in \S\ref{sec:34} that one should not expect such modularity results in general. 
If $d_{\gp}=1$, then  Lemma~\ref{dimsplitcase}(ii) implies that $\ker\left(\rH^1(F,\phi^{-1}) \to \bigoplus_{\gq \ne \gp} \rH^1(F_{\gq},\phi^{-1})\right)$ is a line. Choosing a basis $\eta=\eta_{\phi^{-1}}^{(\gp)}$  of this line, we consider  
 $\cR^{\nord}_{\rho_{\univ}}$  and $\cR^{\ord}_{\eta}$ from Proposition~\ref{nord_functor}.

\begin{thm}\label{r_equal_t}
Assume $d_{\gp}=1$ and  $\delta_{F,p}=0$.
\begin{enumerate}
\item There is a natural isomorphism of regular local $\varLambda^{\nord}$-algebras $\cR^{\nord}_{\rho_{\univ}}\xrightarrow{\sim}\cT^{\nord}$.  
\item There is a natural isomorphism of $\varLambda$-algebras $\cR^{\ord}_{\eta}\xrightarrow{\sim}\cT^{\cusp}$ which are
 local complete intersection rings.
Moreover  the Zariski tangent space of $\cC_{\cusp}$ at $f$ has 
dimension  at most~$d-1$.
\end{enumerate} 
\end{thm}
\begin{proof}
(i) As the image of the morphism \eqref{entryB} is one dimensional  generated by $[\eta]$, \cite[Prop.~1.7.4]{bel-che09} ensures the existence of a deformation $\rho_{\cT}^{\nord}:\rG_F \to \GL_2(\cT^{\nord})$ of  $\rho=\left(\begin{smallmatrix} \phi_1 & \phi_2\eta \\ 0 & \phi_2 \end{smallmatrix}\right)$ such that $\det\rho_{\cT}^{\nord}=\phi_1\phi_2\chi_{\cyc}$ and $\mathrm{tr}(\rho_{\cT}^{\nord})(\Frob_\gl)=T_\gl$ for all $\gl \nmid \gn p$.
The  argument used in  \cite[Prop.~3.4(i)]{BDP}  based on the nearly-ordinariness of $\rho_{\cT}^{\nord} \otimes \mathbf{L}$ and the non-vanishing of $[\eta]_{\mid \rI_\gp}$ implies that  $\rho_{\cT}^{\nord}$ satisfies the condition of Definition~\ref{defnno}(iii). Moreover, as $\phi_{\mid \rG_\gq} \ne  \mathbf{1}$  for 
$\gq\in \Sigma_p\setminus \{\gp\}$,  Hensel's lemma yields that $\rho_{\cT}^{\nord}$ satisfies  condition (ii)  in Definition~\ref{defnno}. Thus, we obtain a
  $\varLambda^{\nord}$-algebra homomorphism
\begin{align}\label{surjRtoT}
\cR^{\nord}_{\rho_{\univ}} \to \cT^{\nord}, 
\end{align}
which is surjective as $\cT^{\nord}$ is generated over $\varLambda^{\nord}$ by the Hecke operators $\{T_\gl,U_v :\; \gl \nmid \gn p, v\in \Sigma_p \}$. As $\cT^{\nord}$ is equidimensional of dimension $d+1$, Proposition~\ref{tgnord} implies that  
\eqref{surjRtoT} is an isomorphism of regular local $\varLambda^{\nord}$-algebras. 

(ii) Reducing  the isomorphism \eqref{surjRtoT}  modulo the ideal $(X_\sigma : \sigma\in \Sigma)\subset \varLambda^{\nord}$ makes the 
 the nearly-ordinary characters $\chi_v$ unramified yielding a  $\varLambda$-algebra isomorphism $\cR^{\ord}_{\eta}\xrightarrow{\sim}\cT^{\cusp}$ (see \eqref{loc-glob-wt}). Finally, as $\{X_\sigma : \sigma\in \Sigma\}$ is a  regular sequence of $d$ elements in $\cT^{\nord}$, it follows that the ring $\cT^{\cusp}$ is local complete intersection.
 
 The local ring $\cR^{\ord}_{\eta}$ pro-represents the sub-functor of $\cD^{\nord}_{\eta}$ requiring additionally that the quotients occurring in Definition~\ref{defnno} to be unramified. As explained in Proposition~\ref{nord_functor}, $\Fil_A$ is then  uniquely determined, 
 showing  in particular that the tangent space $t^{\ord}_{\eta}$ of $\cR^{\ord}_{\eta}$ is contained in $\mathbf{K}_1 $ (see \eqref{eq:nord_tangent_space}). Arguing as  in the proof of Proposition~\ref{tgnord}, we get the following exact sequence
\[0 \to \mathbf{K}_3 \to t^{\ord} \to \ker\left(\rH^1(F, \phi^{-1}) \to \rH^1(F_{\gp}, \overline{\Q}_p)\right) ,\]
where  $\mathbf{K}_3=\ker\left(\mathbf{K}_2 \to \Hom(\rG_F,\overline{\Q}_p)\right)$ is associated to $\left(\begin{smallmatrix} a & b \\ 0 & d\end{smallmatrix}\right) \mapsto a$. As by \eqref{eq:seq_w_rho} one has 
$ \mathbf{K}_2\xrightarrow{\sim}\rH^1(F, \overline{\Q}_p^2)$  in this case, one deduces that  
  $\dim_{\overline{\Q}_p} \mathbf{K}_3 = \dim_{\overline{\Q}_p} \mathbf{K}_2 -1=1$, hence $\dim t^{\ord}  \leqslant d-1$.
\end{proof}
  
\subsection{Etaleness and the reducibility ideal}\label{sec:etale}
The aim of this subsection is to prove Theorem~\ref{main-thm}(ii-iii). 
\begin{lemma}\label{key_lemma}
If $d_{\gp} \geqslant d-1$, then  $\mathbf{C} \subset \gm_{\nord}$. In particular, \ref{(P4)} implies that 
$\mathbf{B}\mathbf{C}\subset \gm_{\nord}^2$.
\end{lemma}
\begin{proof}
The properties of generalized matrix algebras stated in \cite[Prop.~3]{bellaiche-chenevier-eis} yield a natural  injection of  $\Hom(\mathbf{C}/ \gm_{\nord} \mathbf{C},\overline{\Q}_p)$ in $\rH^1(F,\phi)$. 
As $y_\gp=1$, the relation \ref{(P3)} becomes 
\begin{align}\label{eq:p-nearly_ord}
b(g)=d(g)-\phi_1(g)\chi_{\gp}(g) \text { and } c(g)=a(g)-\phi_1(g)\chi_{\gp}(g) \text{ for  all } g\in \rG_{\gp}, 
\end{align}
implying in particular that $c(\rG_{\gp})\subset {\gm}_{\nord}$. Letting $\pi:\mathbf{C} \twoheadrightarrow \mathbf{C}^{*}= \frac{\mathbf{C}}{{\gm}_{\nord} \cap \mathbf{C}}$ denote the
projection  of $\cT^{\nord}$-modules, we deduce a  natural  injection
  \[\Hom( \mathbf{C}^{*}/{\gm}_{\nord}  \mathbf{C}^{*},\overline{\Q}_p) \hookrightarrow  \ker\left(\rH^1(F,\phi) \to  \rH^1(F_\gp,\overline{\Q}_p)\right),\]
  the latter space being zero by  Lemma~\ref{dimsplitcase}(i). Nakayama's Lemma  implies  $\mathbf{C}^{*}=0$, i.e.,  $\mathbf{C}\subset {\gm}_{\nord}$.  
\end{proof}
\begin{thm}\label{proofetalecase} 
Assume $d_{\gp}\geqslant d-1$. Then $\cC_{\cusp}$ and $\cE$ are smooth at $f_{\phi_1}$. 
Moreover the weight maps $\kappa:\cE \to \cW$ and  $\sw:\cC_{\cusp}\to \cW^{\parallel}$ are ramified at $f_{\phi_1}$ if, and only if, 
 $d_{\gp} = d$ and $\sL(\phi) +\sL(\phi^{-1}) = 0$. 
\end{thm}
\begin{proof} By  \ref{(P4)} and Lemma~\ref{key_lemma}, the $\rG_F$-representation $\rho_{\mathbf{L}}$ takes values in $\GL_2(\cT^{\nord})$ and reduces to $\left(\begin{smallmatrix} \phi_1 & 0 \\ 0 & \phi_2 \end{smallmatrix}\right)$ modulo $\gm_{\nord}$. 
Hence the pushforward  $\rho_{\cT^{\cusp}}=\left(\begin{smallmatrix} a & b \\ c& d \end{smallmatrix}\right)$ along $\cT^{\nord} \twoheadrightarrow \cT^{\cusp}$ takes values in $\GL_2(\cT^{\cusp})$ and has same reduction 
modulo $\gm$.  With a slight abuse of notation, we will still denote by  $\chi_v$ its composition  with the surjection $\cT^{\nord}  \twoheadrightarrow \cT^{\cusp}$.
Recall the Cartesian diagram 
\begin{equation}\label{nearlyordtoord}
\begin{tikzcd}
\varLambda^{\nord} \arrow{r}{\kappa^{\#}} \arrow[d, two heads] & \cT^{\nord} \arrow[d, two heads] \\
\varLambda \arrow{r}[above]{\sw^{\#}}                & \cT^{\cusp}
\end{tikzcd}\end{equation}
where the vertical maps are reduction modulo  the ideal $(X_\sigma : \sigma\in \Sigma)\subset \varLambda^{\nord}$. 
This implies that the Zariski relative tangent spaces of $\sw^{\#}$  and $\kappa^{\#}$ are    isomorphic:   
\begin{align}\label{eq:relative-tangent}
\Hom_{\overline{\Q}_p}(\gm/(\gm^2,\gm_{\varLambda}),\overline{\Q}_p )\simeq  \Hom_{\overline{\Q}_p}({\gm}_{\nord}/({\gm}_{\nord}^2,\gm_{\varLambda^{\nord}}), \overline{\Q}_p).
\end{align}

 We take an element of the tangent space of $\cT^{\cusp}$, which we see as a  $\overline{\Q}_p$-algebra homomorphism 
\[\varphi: \cT^{\cusp}/\gm^2 \rightarrow \overline{\Q}_p[\varepsilon].\]  

 By a standard argument in deformation theory, one has 
 \begin{equation}\label{formoftgvector}
\varphi\circ \rho_{\cT^{\cusp}}= \left( \mathbf{1}_2 +\varepsilon
\left(\begin{smallmatrix} a' & b' \\ c' & d' \end{smallmatrix}\right) \right)  \left(\begin{smallmatrix} \phi_1 & 0 \\ 0 & \phi_2 \end{smallmatrix}\right)=
\left(\begin{smallmatrix} (1+\varepsilon a')\phi_1 & \varepsilon b' \phi_2\\ \varepsilon c'  \phi_1& (1+\varepsilon d')\phi_2 \end{smallmatrix}\right)
:\rG_F\to \GL_2(\overline{\Q}_p[\varepsilon]),
\end{equation}
where $a', d' \in \rH^1(F,\overline{\Q}_p)$,  $[b'] \in \rH^1(F,\phi^{-1})$ and $[c'] \in \rH^1(F,\phi)$, and
$\varphi$ belongs to the relative tangent space \eqref{eq:relative-tangent} if, and only if, $d'=-a'$.   
As  $\delta_{F,p}=0$,  we have  that  $a'=\lambda\eta_{\mathbf{1}}$ and $d'=\mu \eta_{\mathbf{1}}$ for some 
$(\lambda,\mu)\in \overline{\Q}_p^2$. Moreover, for $v\in \Sigma_p$  the character  $\varphi(\chi_v)=\mathbf{1}+\chi'_v\varepsilon$ is unramified, i.e., 
$\chi'_v$ is unramified. 

\medskip
Suppose first that $d_{\gp}=d-1$ and let us show in this case the vanishing of the relative tangent space \eqref{eq:relative-tangent}. 
 For this we additionally assume that $\varphi$ vanishes on $\gm_{\varLambda}$, {\it i.e.}  $\mu=-\lambda$.  Let $\gq\ne  \gp$ be the regular prime in $\Sigma_p$ and choose  $ g_\gq \in \rG_{\gq}$ such that $\phi_1(g_\gq) \ne \phi_2(g_\gq)$.  As $ b(g_\gq)c(\rG_\gq)\subset {\gm}^2$ by   Lemma~\ref{key_lemma}, and 
\[\left(\varphi(d)-\phi_1 \varphi(\chi_\gq)\right)(g_\gq)  =((\mathbf{1}-\varepsilon a')\phi_2-\phi_1(\mathbf{1}+\chi'_\gq\varepsilon))(g_{\gq})\in \overline{\Q}_p[\varepsilon]^\times,\] 
 we deduce from  \ref{(P3)}    that $\left(\varphi(a)-\phi_1 \varphi(\chi_\gq)\right)(\rG_\gq)=((\mathbf{1}+\varepsilon a')\phi_1-\phi_1(\mathbf{1}+\chi'_\gq\varepsilon))(\rG_{\gq})= \{0\}$. 
 Hence $a'_{|\rG_{\gq}}= \lambda\eta_{\mathbf{1}|\rG_{\gq}}=\chi'_{\gq}$ is unramified  implying  $\lambda=0$ and  $\chi'_{\gq}=0$. 
Hence $a'=d'=0$ and it follows then from  \eqref{eq:p-nearly_ord} that
  $b'_{|\rI_{\gp}}=c'_{|\rI_{\gp}}=0$. By Proposition~\ref{the_module_B},   $[b']$  is also unramified at $\gq$,  hence unramified everywhere and therefore $[b']=0$.  As $\phi^{-1}_{\mid \rG_{\gp}}$ is trivial, it follows that  all coboundaries vanish on $\rG_{\gp}$, hence we have $b'_{\mid \rG_{\gp}} = 0$.
 Using again  \eqref{eq:p-nearly_ord}, the vanishing of  $d'$ and $b'_{\mid \rG_{\gp}}$, first implies $\chi'_{\gp}=0$ which combined with the vanishing of $a'$, yields $c'_{|\rG_{\gp}}=0$.  It follows then from  Lemma~\ref{dimsplitcase}(i) that  $[c']=0$.  As $b'(\gamma_0)=c'(\gamma_0)=0$, we conclude that $a'=b'=c'=d'=0$. So far we have shown that 
\begin{equation} 
\label{eq:T-ell}
\begin{split}
\varphi(T_\gl)&= \varphi\circ \mathrm{tr}(\rho_{\cT^{\cusp}})(\Frob_\gl)=\phi_1(\gl) + \phi_2(\gl) \in \overline{\Q}_p \text{ for all $\gl \nmid \gn p$, and} \\ 
\varphi(U_v)&= \phi_1 \varphi \circ \chi_v(\Frob_v)= \phi_1(\Frob_v) \in \overline{\Q}_p, \text { for all } v\in \Sigma_p. 
\end{split}
\end{equation}
Given that $\cT^{\cusp}$ is generated over $\varLambda$ by the Hecke operators 
 $\{T_\gl,U_v : \gl \nmid \gn p, v\in \Sigma_p \}$, we can conclude that $\varphi$ factors  through the maximal ideal,  hence $\dim_{\overline{\Q}_p} \gm/(\gm^2,\gm_{\varLambda})=0$. In particular, the local homomorphism $\kappa^{\#}: \varLambda  \to \cT^{\cusp}$ is unramified. But $\kappa^{\#}$ is also finite flat by construction of $\cC_{\cusp}$,  hence $\kappa^{\#}$ is \'etale. 
As $\cT^{\cusp}$ is the pushout of \eqref{nearlyordtoord}, the local homomorphism $ \sw^{\#}: \varLambda^{\nord} \to \cT^{\nord}$ is also unramified. Moreover,  since $\cT^{\nord}$ has Krull dimension $d+1$ and $\varLambda^{\nord}$ is regular of dimension $d+1$, the unramifiedness of $\sw^{\#}$ implies  that $ \cT^{\nord}$ is also regular of dimension $d+1$. Now, we can invoke the miracle flatness theorem to claim that $ \sw^{\#}$ is flat, and even étale, as it is also unramified. 

\medskip
Now suppose that $d_{\gp}=d$, i.e.,   $\Sigma_p = \{ \gp \}$, and recall that the infinitesimal character  $\varphi(\chi_\gp)=\mathbf{1}+\chi'_\gp\varepsilon$ is unramified. By \eqref{eq:p-nearly_ord} we have 
$c'=\lambda \cdot \eta_{\mathbf{1}}$ on $\rI_{\gp}$, therefore $ c' = \lambda \cdot \eta_{\phi}^{(\gp)}$ by Proposition~\ref{L-invariant}(i), and similarly one has $b' =\mu  \eta_{\phi^{-1}}^{(\gp)}$. Using   \eqref{eq:p-nearly_ord}, Proposition~\ref{L-invariant}(ii)  yields
\begin{equation}\label{coordinatescuspdefominertcase}
\begin{split}
\mu  \sL(\phi^{-1})=\mu  (\eta_{\phi^{-1}}^{(\gp)}-\eta_{\mathbf{1}})(\Frob_{\gp}) & =
  (b' - d')(\Frob_{\gp})=\\
 =- \chi'_\gp(\Frob_\gp) & = (c'-a')(\Frob_{\gp})=
\lambda  (\eta_{\phi}^{(\gp)}-\eta_{\mathbf{1}})(\Frob_{\gp})=\lambda  \sL(\phi). 
\end{split}
\end{equation} 
As $\sL(\phi^{-1}) \cdot \sL(\phi) \ne 0$ by Proposition~\ref{L-invariant}(iv),  we obtain $\mu= \frac{\sL(\phi)}{ \sL(\phi^{-1})} \lambda $.  If we assume that  $\sL(\phi) +\sL(\phi^{-1}) \ne 0$, then the relative tangent space
imposing the additional condition $\lambda +\mu=0$ has to vanish and we conclude exactly as in the case 
$d_{\gp}=d-1$. It remains to treat the case where $\sL(\phi) +\sL(\phi^{-1}) =0$
In this case, we only know by  \eqref{coordinatescuspdefominertcase}  that  the tangent vector $\varphi$ is determined uniquely by $\lambda \in \overline{\Q}_p$. Hence, the tangent space is at most one-dimensional. As the Krull dimension of  $\cT^{\cusp}$ is one, the tangent space is exactly one-dimensional  coinciding with the relative tangent space, hence  $\cT^{\cusp}$  is a discrete valuation ring and the weight map $\Lambda\to \cT^{\cusp}$ is ramified. 
 Finally, $\cT^{\nord}$  is also regular, since it has Krull dimension $d+1$, and its quotient by the ideal $(X_{\sigma}, \sigma \in \Sigma) $, generated by $d$ elements, is a discrete valuation ring. 
\end{proof}

\subsection{Failure of  factoriality in the middle range}\label{sec:34}
In this subsection, we treat the remaining case where $2 \leqslant d_{\gp}  \leqslant d-2$, thus proving Theorem~\ref{main-thm}(iv). 
We will  prove that given any minimal prime ideal $\cP$ of $\cT^{\nord}$, the local ring 
$\cT=\cT^{\nord}/\cP$ is never factorial. Let $\gm$ denote the maximal ideal of $\cT$. We let $\mathrm{Ps}: \rG_F \rightarrow \cT$ be the pseudo-character
obtained by reducing $\mathrm{Ps}_{\cT^{\nord}}$ from \S\ref{sec:gma} modulo $\cP$.  By a slight abuse of notations, we denote by $\mathbf{B}$ and  $\mathbf{C}$ their respective   images along the projection $\mathbf{L} \twoheadrightarrow \cT^{\nord}_{\cP}$, 
where   $\cT^{\nord}_{\cP}$ is nothing but the field at the generic point corresponding to  $\cP$. 
 
\begin{prop}\label{nearly_ordinary_deform_big_local_degree}
Suppose  $d_{\gp}  \geqslant 2$ and $\delta_{F,p}=0$.  Given $0\ne [\eta]\in \rH^1(F,\phi^{-1})$  unramified outside $\gp$, there does not exist  a nearly-ordinary deformation $\rho_{\cT}^{\nord}:\rG_F\to \GL_2(\cT)$ of $\rho=\left(\begin{smallmatrix} \phi_1 & \phi_2\eta \\ 0 & \phi_2\end{smallmatrix}\right)$ whose  trace equals~$\mathrm{Ps}$. 
\end{prop}
\begin{proof} 
If a nearly-ordinary deformation $\rho_{\cT}^{\nord}$ as in the statement existed, it would give a surjective $\varLambda^{\nord}$-algebra homomorphism $\varphi:\cR^{\nord}_{\rho_{\univ}}\to \cT$, where $\cR^{\nord}_{\rho_{\univ}}$ is the Noetherian $\varLambda^{\nord}$-algebra representing the 
nearly-ordinary deformation functor $\cD^{\nord}_{\eta}$ (one can check that $\rho_{\cT}^{\nord}$ satisfies Definition~\ref{defnno}(iii) using 
$\eta_{|\rI_{\gp}}\ne 0$ as in \cite[Prop.~3.4]{BDP}). As $\cT$ is equidimensional of dimension $d+1$, Proposition~\ref{tgnord}  would then imply that $\varphi$ is an isomorphism. 
 
As $d_{\gp}  \geqslant 2$,  by Lemma~\ref{lemma-21} there exists  a cocycle $[b]\in \rH^1(F,\phi^{-1}) \setminus (\overline{\Q}_p \cdot [\eta])$ which is unramified outside $\gp$. The representation  $\rho_{\varepsilon}=\left(1+\left(\begin{smallmatrix} 0 & b \\ 0 & 0\end{smallmatrix}\right)\varepsilon\right)\rho$  endowed with the filtration $\overline{\Q}_p[\varepsilon]\cdot e_1$ yields a non-trivial element  
$\varphi_\cR\in\Hom_{\overline{\Q}_p}(\cR^{\nord}_{\rho_{\univ}}/\gm_{\varLambda^{\nord}},\overline{\Q}_p[\varepsilon])$ hence a non-trivial element  
$\varphi_\cT=\varphi_\cR\circ\varphi^{-1} \in\Hom_{\overline{\Q}_p}(\cT/\gm_{\varLambda^{\nord}},\overline{\Q}_p[\varepsilon])$. 
Meanwhile $\tr \rho_{\varepsilon}=\tr\rho$ implies that $\varphi_\cT(T_{\gl})=0$ for all  $\gl\nmid \gn p$. 
Moreover  $\varphi_\cT(U_v)=0$, as $\Frob_v$ acts trivially on the quotient by the line generated by $e_1$ (resp. $e_2$) 
for $v=\gp$ (resp. $v\in \Sigma_p\setminus \{\gp\}$). Hence $\varphi_\cR\circ\varphi^{-1}$ is  trivial after all,  leading to a  contradiction. 
\end{proof}
\begin{cor}\label{dim_of_B_mod_mB}
Let the assumption be as in Proposition~\ref{nearly_ordinary_deform_big_local_degree}. Then  one has
\begin{align}\label{eq:dim_of_B_modulo_mB}
\dim_{\overline{\Q}_p} \Hom_{\cT}(\mathbf{B}/\gm \mathbf{B},\overline{\Q}_p)  \geqslant 2.
\end{align}
\end{cor}
\begin{proof}
If $\dim_{\overline{\Q}_p} \Hom_{\cT}(\mathbf{B}/\gm \mathbf{B},\overline{\Q}_p) = 1$, 
then the argument of Theorem~\ref{r_equal_t} would give a $\GL_2(\cT)$-valued nearly-ordinary deformation of $\left(\begin{smallmatrix} \phi_1 & \phi_2\eta \\ 0 & \phi_2\end{smallmatrix}\right)$ for some $0\neq \eta\in \rH^1(F,\phi^{-1})$ unramified outside $\gp$, contradicting Proposition~\ref{nearly_ordinary_deform_big_local_degree}.
\end{proof}
\begin{thm}\label{nonsm}
If $2 \leqslant d_{\gp} \leqslant d-2$ and $\delta_{F,p}=0$, then $\cT$ is not factorial. In particular, Theorem~\ref{main-thm}(iv) holds.
\end{thm}
\begin{proof}
Suppose that $\cT$ is factorial domain. By Gauss's lemma the representation $\rho_{\mathbf{L}}$ from \S\ref{sec:gma}
can be conjugated by a diagonal matrix so that it takes values in $\GL_2(\cT)$ and that the resulting ideal $\mathbf{B}\subset \cT$ is as large as possible, 
i.e.,  the greatest common divisor of the generators of $\mathbf{B}$ is $1$. 
By Corollary~\ref{dim_of_B_mod_mB}, the ideal $\mathbf{B}$ is not monogenic and hence, $\mathbf{B}\subset \gm$. It follows from \eqref{entryB} that there is a  natural injection
\begin{align}\label{eq:b-nord}
\Hom(\mathbf{B}/{\gm} \mathbf{B}, \overline{\Q}_p) \hookrightarrow \mathrm{im}\left(\rH^1(F_{\gp},\phi^{-1})\to \rH^1(\rI_{\gp},\phi^{-1})\right), 
\end{align}
implying that $b^{\nord}=\dim_{\overline{\Q}_p} \Hom(\mathbf{B}/{\gm} \mathbf{B},\overline{\Q}_p) \leqslant d_{\gp}$. Hence, we have
\begin{align}\label{eq:tnord-first}
	\dim_{\overline{\Q}_p}\left( \frac{\mathbf{B}+\gm^2}{\gm^2}\right)=\dim_{\overline{\Q}_p}\left(  \frac{\mathbf{B}}{\mathbf{B}\cap\gm^2} \right)\leqslant \dim_{\overline{\Q}_p} \left(\frac{\mathbf{B}}{\gm \mathbf{B}}\right)\leqslant d_{\gp}. 
\end{align} 
As $\dim_{\overline{\Q}_p} (\gm/\gm^2)\geqslant \dim(\cT)=d+1$, we deduce that 
\begin{align}\label{eq:tnord-second}
	\dim_{\overline{\Q}_p} \left(\frac{\gm}{\mathbf{B}+\gm^2}\right)=\dim_{\overline{\Q}_p} \left(\frac{\gm}{\gm^2}\right)-\dim_{\overline{\Q}_p}\left( \frac{\mathbf{B}+\gm^2}{\gm^2}\right)\geqslant d+1-d_{\gp}\geqslant 3. 
\end{align} 
We will contradict this by a Galois deformation argument showing  $\dim_{\overline{\Q}_p} \Hom\left(\gm/(\mathbf{B}+\gm^2),\overline{\Q}_p\right)  \leqslant 2$.
Given any $\varphi\in \Hom\left(\gm/(\mathbf{B}+\gm^2),\overline{\Q}_p\right)$ the corresponding representation $\varphi\circ\rho_{\mathbf{L}}: \rG_F \to \GL_2(\overline{\Q}_p[\varepsilon])$ is lower triangular of the form $\left(\begin{smallmatrix} \phi_1(1+\lambda\eta_{\mathbf{1}}\varepsilon) & 0 \\ \ast & 
\phi_2(1+\mu \eta_{\mathbf{1}}\varepsilon)\end{smallmatrix}\right)$ for some $\lambda, \mu\in \overline{\Q}_p$.  
For any $v\in \Sigma_p$  $\varphi(\chi_v)=1+\chi'_v \varepsilon$ with  $\chi'_v\in \Hom(\rG_v, \overline{\Q}_p)$.  
Since $y_{\gp} \mathbf{B} \subset \gm$ and $\mathbf{B}$ has at least two generators, Gauss's lemma yields that $y_{\gp} \in \cT$ and hence, $y_{\gp} \mathbf{B} \subset \mathbf{B}$. It then follows from \ref{(P3)} that $\chi'_{\gp}=\mu \cdot \eta_{\mathbf{1}|\rG_{\gp}}$. Since $\phi(\gq)\ne 1$ for $\gq\in \Sigma_p\setminus \{\gp\}$, $\tr(\varphi\circ\rho_{\mathbf{L} \mid \rG_\gq})=\phi_1\varphi(\chi_v) + (\phi_1\varphi(\chi_v))^{-1} \det(\varphi\circ\rho_{\mathbf{L} \mid \rG_\gq})$ can be written uniquely as the sum of two characters lifting $\phi_1 + \phi_2$. In particular, this implies that $\chi'_{\gq}=\lambda\cdot \eta_{\mathbf{1}|\rG_{\gq}}$.
As the $\varLambda^{\nord}$-algebra $\cT$ is topologically generated  by  $\{T_{\gl},U_v : \gl\nmid \gn p, v\in \Sigma_p\}$ and the image by 
$\varphi$ of all these quantities is uniquely determined by  $\lambda$ and $\mu$, it follows that tangent space of $\cT/(\mathbf{B}+\gm^2)$ is at most two-dimensional. 
\end{proof}

\section{Arithmetic applications}\label{sec:04}
This section provides some  arithmetic applications of our results. 
For simplicity we assume that $\phi_1=\mathbf{1}$ hence $\phi=\phi_2$ and we
denote by $\mathrm{Ps}_{\cT^{\cusp}}$ the push-forward of the pseudo-character  $\mathrm{Ps}_{\cT^{\nord}}$ from \S\ref{sec:gma} 
along the natural projection $\cT^{\nord}\twoheadrightarrow \cT^{\cusp}$. The corresponding family of $p$-adic Galois representations
\begin{equation}\label{ordinarycuspGMA} \rho_{\mathbf{K}}=\left(\begin{smallmatrix} a & b \\ c & d \end{smallmatrix}\right):\rG_F\to \GL_2(\mathbf{K}),\end{equation}
where $\mathbf{K}$ is the  total quotient ring of $\cT^{\cusp}$, 
enjoys all the properties of $\rho_{\mathbf{L}}$ in addition to being ordinary at all $v\in \Sigma_p$. Exactly as in \S\ref{sec:gma}, 
by the results  of  \cite{DKV} there exists $\gamma_0 \in \rG_F$ and a basis $(v_1,v_2)$ of $\mathbf{K}^2$ such that the properties \ref{(P1)}, \ref{(P2)}, \ref{(P3)} and \ref{(P4)} from \S\ref{sec:gma} hold for $ \rho_{\mathbf{K}}$. In particular, $a$ and $d$ take values in $\cT^{\cusp}$, while $b$ takes value in  the maximal ideal $\gm$ of $\cT^{\cusp}$.

\subsection{Geometric construction of Eisenstein congruences at trivial zeros}\label{generic cuspidal deformation}
 In this subsection, we do not assume the vanishing of the Leopoldt defect  $\delta_{F,p}$.
\begin{prop}\label{simple_trivial_zero}
Assume $\# \Sigma_p^{\mathrm{irr}}=1$. Then  $\zeta_{\phi}\cdot\varLambda=X\cdot \varLambda$, i.e., $\zeta_{\phi}$ has a simple trivial zero at $X=0$.
\end{prop}
\begin{proof}
Suppose $\zeta_{\phi}\cdot \varLambda \subset X^2\cdot \varLambda$.  It then follows from Proposition~\ref{const_Eisen_family}
that the constant terms of  $\cE_{\mathbf{1},\phi}$ vanish modulo $X^2$, hence by Theorem~\ref{fund_exact_seq} there exists a cuspidal family $\cF$ with Fourier coefficients in $\varLambda$ such that  $\cF \equiv \cE_{\mathbf{1},\phi} \pmod{X^2}$.  
As $\cF\mod{X^2}$ is an eigenform, it corresponds (by Hida duality and after renormalizing $X$) to the  $\overline{\Q}_p$-algebra homomorphism 
\[\varphi_{\cF,\varepsilon}:\cT^{\cusp}\to \varLambda/(X^2) = \overline{\Q}_p[\varepsilon]
\]
such that  $\varphi_{\cF,\varepsilon}\circ \mathrm{Ps}_{\cT^{\cusp}}=1+\phi\cdot\left( 1+\eta_{\mathbf{1}}\cdot \varepsilon\right)$ 
and sending $U_\gp$ to $1$ (see \eqref{Eisensteinnonconstanterms}).

It follows that  $ \varphi_{\cF,\varepsilon}(a)=1$  and  $ \varphi_{\cF,\varepsilon}(d)=\phi(1+\eta_{\mathbf{1}}\varepsilon)$. 
By Proposition~\ref{the_module_B},  $\varphi_{\cF,\varepsilon}\circ b$ yields an element of $\rH^1(F,\phi^{-1})$ which is unramified outside $\gp$. Furthermore,
it follows from \eqref{eq:p-nearly_ord} that  $\varphi_{\cF,\varepsilon}(b)=\eta^{(\gp)}_{\phi}\varepsilon$, hence 
\[
\varphi_{\cF,\varepsilon}(U_\gp)=
\varphi_{\cF,\varepsilon}(\chi_{\gp}(\Frob_{\gp}))=\varphi_{\cF,\varepsilon}\left((d-b)(\Frob_{\gp})\right)=1-\sL(\phi)\varepsilon.
\]
As $\sL(\phi)\neq 0$ by  Proposition~\ref{L-invariant}(iv), this  leads to a contradiction. 
\end{proof}
Recall that the Eisenstein family $\cE_{\mathbf{1},\phi}$ is the unique non-cuspidal eigenfamily containing $f=f_{\mathbf{1}}$, unless $d_{\gp}=d$ in which case there is also the  Eisenstein family $\cE_{\phi,\mathbf{1}}$. Let $\mathbf{M}^{\ord}_{\gm_f}=\mathbf{M}^{\ord}_{\cU} \widehat{\otimes}_{\cT^{\leqslant 0}_{\cU}} \cT$
be the completed localization of  $\mathbf{M}^{\ord}_{\cU}$ with respect to the maximal ideal $\gm_f$ corresponding to $f$. 
Combining Proposition~\ref{simple_trivial_zero} with \eqref{eq:structure_hecke_alg_not_inert_case} has the following consequence. 
\begin{cor}\label{TC} Assume $\# \Sigma_p^{\mathrm{irr}}=1$  and  also assume that $\sL(\phi) +\sL(\phi^{-1})  \ne 0$ when $d_{\gp} = d$. Then 
\[	\cT \simeq 
	\begin{cases}
		\varLambda \times_{\overline{\Q}_p} \cT^{\cusp} & \text{  if } d_{\gp}<d, \\
		\varLambda \times_{\overline{\Q}_p} \varLambda \times_{\overline{\Q}_p} \cT^{\cusp}  & \text{  if } d_{\gp}=d.
	\end{cases}\]
\end{cor}
We now construct a cuspidal family that enjoys intriguing arithmetic properties and is directly related to the Gross--Stark conjecture through its Fourier coefficients at 
$p$.
\begin{thm}\label{epsilon_deformaion_of_DDP}
Assume that $\# \Sigma_p^{\mathrm{irr}}=1$, and further assume that  $\sL(\phi) +\sL(\phi^{-1})  \ne 0$ when $d_{\gp} = d$.  There exists a cuspidal family $\cF$ with  coefficients in $\varLambda$  containing $f$,  such that $\cF$ is an eigenform modulo $X^2$ and the  corresponding  $\overline{\Q}_p$-algebra homomorphism $\varphi_{\cF,\varepsilon}:\cT^{\cusp}\to \varLambda/(X^2) = \overline{\Q}_p[\varepsilon]$ satisfies \begin{align}\label{eq:cong_cuspform_and_Eisen_family}
		\varphi_{\cF,\varepsilon}\circ \mathrm{Ps}_{\cT^{\cusp}}=(1+\lambda \eta_{\mathbf{1}}\cdot\varepsilon)+\phi\left( 1+\mu \eta_{\mathbf{1}}\cdot \varepsilon\right), \text{ where} 
	\end{align}
	\begin{align}\label{eq:lambda_and_mu}
		(\lambda,\mu)=
		\begin{cases}
			(0, -\frac{1}{\log_p(u)}) & \text{ if } d_{\gp}< d, \\
			\left(- \frac{\sL(\phi)}{\left(\sL(\phi)+\sL(\phi^{-1})\right)\log_p(u)},-\frac{\sL(\phi^{-1})}{\left(\sL(\phi)+\sL(\phi^{-1})\right)\log_p(u)}\right) & \text{ if } d_{\gp}=d. 			
		\end{cases}
	\end{align}
	Moreover, the Fourier coefficients $C(\gl,\cF)$, for $\gl \nmid \gn p$ and $C(v,\cF)$, for $v\in\Sigma_p$ satisfy
	\[
	\left.\frac{d}{dX}\right|_{X=0} C(\gl,\cF)=
	\begin{cases}
	\frac{\log_p(\rN_{F/\Q}(\gl))}{\log_p(u)}
		 &
		\text{ if } d_{\gp}< d,\\
		\frac{\left(\sL(\phi)+\phi(\gl)\sL(\phi^{-1})\right) \log_p(\rN_{F/\Q}(\gl))}{\left(\sL(\phi)+\sL(\phi^{-1})\right)\log_p(u)} & \text{ if } d_{\gp}=d, 
	\end{cases}
	\]
	\[
	\left.\frac{d}{dX}\right|_{X=0} C(v,\cF)=
	\begin{cases}
		\delta_{\gp}(v)\cdot \frac{\sL(\phi)}{\log_p(u)} &
		\text{ if } d_{\gp}< d,\\
		\frac{\sL(\phi)\sL(\phi^{-1})}{\left(\sL(\phi)+\sL(\phi^{-1})\right)\log_p(u)}  & \text{ if } d_{\gp}=d. 
	\end{cases}
	\]
	 Here $\delta_{\gp}(v)$ is $1$, if $v=\gp$, and zero  otherwise. 
\end{thm}
\begin{proof} 
Proposition~\ref{prop:hecke_duality}  induces a perfect pairing 
\begin{align}\label{eq:hecke_duality_modular_family} 
\mathbf{M}^{\ord}_{\gm_f} \times \cT \to \varLambda, \quad (\cF,h) \mapsto C(\go,h\cdot \cF).
\end{align}
Set $\cT'=\cT/ X \cT$ whose maximal ideal is denoted by $\gm'$, and let $\mathbf{M}^\dagger_{1}\lsem f \rsem$ be the 
generalized eigenspace at $f$ in the space of weight $1$ overconvergent forms. 
The perfect pairing \eqref{eq:hecke_duality_modular_family} yields a natural isomorphism: 
\begin{align}\label{dualitygeneralized} 
\mathbf{M}^\dagger_{1}\lsem f \rsem \xrightarrow{\sim} \Hom_{\overline{\Q}_p}( \cT', \overline{\Q}_p);\; g\mapsto [h\mapsto C(\go,h\cdot g)]. 
\end{align}

We set $\displaystyle p'=\prod_{ \gq\in\Sigma_p\setminus \{\gp\}} \gq$ and we  let  $E_1^{(p')}$ denote the unique $p'$-stabilization of $E_1(\mathbf{1},\phi)$ such that for all $\gq\in \Sigma_p\setminus \{\gp\}$ one has $U_{\gq}\cdot E_1^{(p')}= E_1^{(p')}$.
Note that when $d_{\gp}=d$ we have $E_1^{(p')}=E_1(\mathbf{1},\phi)$.  It is easy to see that $(U_\gp -1)^2\cdot E_1^{(p')}=0$, hence $E_1^{(p')}\in \mathbf{M}^\dagger_{1}\lsem f \rsem$. By  definition, one can verify using \eqref{eq:fourier_coeff_wt_1_eisen_series}, that
\begin{align}\label{eq:const_term_p'_stab}
 C((0),E_1^{(p')})=2^{-d}  \prod_{\gq\in \Sigma_p\setminus\{\gp\}} (1-\phi(\gq)) \sum_{[\gc]\in\Cl_F^+}\left( L(0,\phi)+\delta(\phi)\phi^{-1}(\gc\gd)L(0,\phi^{-1})\right)[\gc]\neq 0.
\end{align}

The element  $E^\dagger=f-E_1^{(p')}\in  \mathbf{M}^\dagger_{1}\lsem f \rsem\setminus\{0\}$ gives   via \eqref{dualitygeneralized} 
 a $\overline{\Q}_p$-module homomorphism $ \cT'/\gm'^2 \to \overline{\Q}_p$, whose  restriction yields a tangent vector $[\varphi_{\varepsilon}]:\gm'/{\gm'}^2\to \overline{\Q}_p$
of $\cT'$. The latter determines uniquely  a $\overline{\Q}_p$-algebra homomorphism $\varphi_{\varepsilon}:\cT' \twoheadrightarrow \overline{\Q}_p[\varepsilon]$, 
 corresponding to the $\varepsilon$-eigenform $f_\varepsilon= f + \varepsilon E^\dagger$. We have
\begin{align}\label{eq:coef_f_epsilon} \begin{split}
\varphi_{\varepsilon}(T_\ga)& =C(\go,T_\ga \cdot f_\varepsilon) =C(\ga,f_{\varepsilon})=C(\ga,f)+\varepsilon C(\ga,E^{\dagger}) \\
\varphi_{\varepsilon}(U_\gq)& =C(\go,U_\gq \cdot f_\varepsilon) =C(\gq,f_{\varepsilon})=C(\gq,f)+\varepsilon C(\gq,E^{\dagger}) \\
\varphi_{\varepsilon}(h_0) \sum_{[\gc]\in\Cl_F^+}  [\gc]& =C(\go,h_0 \cdot f_\varepsilon) \sum_{[\gc]\in\Cl_F^+}  [\gc]=C((0),f_{\varepsilon})=C((0),E^{\dagger}) \cdot \varepsilon =- C((0),E_1^{(p')}) \cdot \varepsilon \ne 0.
\end{split}
\end{align}
for all integral ideals $\ga$ (resp. $\gq$) of $F$ prime to $\gn p$ (dividing $\gn p$). We have a natural pseudo-character $\mathrm{Ps}_{\varepsilon}:\rG_F\to \cT'\xrightarrow{\varphi_{\varepsilon}} \overline{\Q}_p[\varepsilon]$ sending $\Frob_{\gl}$ to $C(\gl,f_{\varepsilon})$ for all prime ideals $\gl\nmid \gn p$. Since $C(\gl,E^{\dagger})=0$ for all prime ideals $\gl\nmid \gn p$, we have $\mathrm{Ps}_{\varepsilon}=\phi+1$ by \eqref{eq:coef_f_epsilon} and the Chebotarev density theorem. 
It follows from Corollary~\ref{TC} that the maximal ideal of $\cT$ is equal to 
\[
\begin{cases} (\gm_{\varLambda},0)\oplus (0,\gm) & \text{  if } d_{\gp}<d, \\
(\gm_{\varLambda},0,0)\oplus (0,\gm_{\varLambda},0) \oplus (0,0,\gm)& \text{  if }d_{\gp}=d.
\end{cases} 
\]
 Therefore, $[\varphi_{\varepsilon}]$ can be written as 
\[
[\varphi_{\varepsilon}]=
\begin{cases}
[\varphi_{\varepsilon,(\mathbf{1},\phi)}^{\eis}]+[\varphi_{\varepsilon}^{\cusp}] & \text{ if } d_{\gp}<d,\\
[\varphi_{\varepsilon,(\mathbf{1},\phi)}^{\eis}]+[\varphi_{\varepsilon,(\phi,\mathbf{1})}^{\mathrm{eis}}]+[\varphi_{\varepsilon}^{\cusp}] & \text{ if } d_{\gp}=d,
\end{cases}
\]
for some $\overline{\Q}_p$-algebra homomorphisms $\varphi_{\varepsilon}^{\cusp}: \cT^{\cusp} \to \overline{\Q}_p[\varepsilon]$ and $\varphi_{\varepsilon,(\mathbf{1},\phi)}^{\eis},\varphi_{\varepsilon,(\phi,\mathbf{1})}:\varLambda \to \overline{\Q}_p[\varepsilon]$. Moreover, the latter two homomorphisms induce reducible pseudo-characters $\mathrm{Ps}^{\eis}_{\varepsilon,(\mathbf{1},\phi)},\mathrm{Ps}^{\eis}_{\varepsilon,(\phi,\mathbf{1})}:\rG_F \to \varLambda \to \overline{\Q}_p [\varepsilon]$ sending $g\in \rG_F$ respectively to 
\[ 
1+\phi(g)\left(1-\frac{\alpha \cdot \eta_{\mathbf{1}}(g)}{\log_p(u)}\cdot \varepsilon\right) 
\text{ and }
\left(1-\frac{\beta \cdot \eta_{\mathbf{1}}(g)}{\log_p(u)}\cdot \varepsilon\right)+\phi(g)
\] 
for some $\alpha,\beta\in \overline{\Q}_p$. Here  $\cE(\mathbf{1},\phi) \mod X^2$ (resp.~$\cE(\phi,\mathbf{1}) \mod X^2$) gives rise to the infinitesimal pseudo-character $g \mapsto 
1+\phi(g)\left(1-\frac{\eta_{\mathbf{1}}(g)}{\log_p(u)}\cdot \varepsilon\right) $ (resp. $ g \mapsto
\left(1-\frac{\eta_{\mathbf{1}}(g)}{\log_p(u)}\cdot \varepsilon\right)+\phi(g))$.

By \eqref{Eisensteinconstanterms}, we know that the constant term of $\cE(\mathbf{1},\phi)$  is non-zero, and  has a simple zero at $X=0$ by Proposition~\ref{simple_trivial_zero}. Thus, using the operator $h_0$ from Proposition~\ref{prop:hecke_duality}, which determines the constant term via Hida's duality, we must have $\alpha\neq 0$. If $d_{\gp}<d$, using $\mathrm{Ps}_{\varepsilon}=\phi+1$ we find  
\begin{align}\label{eq:varphi_cusp_epsilon}
\varphi_{\varepsilon}^{\cusp} \circ \mathrm{Ps}_{\cT^{\cusp}}(g)=1+\phi(g)\left(1+\frac{\alpha \cdot \eta_{\mathbf{1}}(g)}{\log_p(u)}\cdot \varepsilon\right). 
\end{align}
After rescaling  $\varphi_{\varepsilon}^{\cusp}$ if necessary, we obtain the desired cuspidal family  $\cF$ via Hida's duality.

We now assume that $d_{\gp}=d$. First, we show that $\beta \ne 0$. If $\beta=0$, then \eqref{eq:varphi_cusp_epsilon} holds again as $\mathrm{Ps}_{\varepsilon}=\phi+1$. One can verify by using \eqref{eq:p-nearly_ord} that $\varphi_{\varepsilon}^{\cusp}(c)=0$ on $\rG_F$ as it is in $\rH^1(F,\phi)$ and unramified everywhere (here $F$  has a unique prime above $p$). In addition, \eqref{eq:p-nearly_ord} and Proposition~\ref{L-invariant} imply that
\[
1=\varphi_{\varepsilon}^{\cusp}(c-a)(\Frob_{\gp})=\varphi_{\varepsilon}^{\cusp}(b-d)(\Frob_{\gp})=1+\alpha\sL(\phi)\varepsilon,
\]
which leads to a contradiction as $\sL(\phi)$ is non-zero by Proposition~\ref{L-invariant}. Therefore, we must have $\beta\neq 0$. Moreover, by using  \eqref{eq:p-nearly_ord} and Proposition~\ref{L-invariant}, one again obtains the following equation  (as in \eqref{coordinatescuspdefominertcase})
\[
\alpha\sL(\phi)=\beta\sL(\phi^{-1}).
\]
It follows from $\det(\rho_{\mathbf{K}})=\phi\chi_{\cyc}$ that 
\[
\begin{split}
\phi(1- \eta_{\mathbf{1}}(\log_p u)^{-1}\varepsilon)
=\varphi_{\varepsilon}^{\cusp}(\phi\chi_{\cyc})
= \left(1+\frac{\alpha \cdot \eta_{\mathbf{1}}}{\log_p(u)}\cdot \varepsilon\right) 
\phi\left(1+\frac{\beta \cdot \eta_{\mathbf{1}}}{\log_p(u)}\cdot \varepsilon\right).
\end{split}
\]
Hence, we obtain another equation $\alpha+\beta=-1$, and thus
\[
\alpha=-\frac{\sL(\phi^{-1})}{\sL(\phi)+\sL(\phi^{-1})} \text{ and } 
\beta=-\frac{\sL(\phi)}{\sL(\phi)+\sL(\phi^{-1})}.
\]

This again gives the desired cuspidal family $\cF$ by Hida's duality and completes the proof of the first part of the claim.

Finally, the argument regarding the properties of the Fourier coefficients of $\cF$ is essentially the same as that in \cite[Prop.~5.7]{BDP}, to which we refer the reader for details.
\end{proof}

We notice that if $d_{\gp} \geqslant d-1$ and $\delta_{F,p}= 0$, then the cuspidal family $\cF$ constructed in the above theorem is an eigenfamily by Theorem~\ref{main-thm}(ii-iii).

We highlight that $\cF$ coincides $\pmod{X^2}$ with the $\varLambda_{\cO}$-adic cuspidal family $\mathscr{H}$ that Darmon--Dasgupta--Pollack built in \cite[Prop.3.4]{DDP} using purely analytic arguments assuming $\delta_{F,p}= 0$ under the non--vanishing hypothesis \cite[(11)]{DDP}. This algebraic construction will be expanded  in a forthcoming work of Betina--Hsieh on trivial zeros  of Katz $p$-adic $L$-functions using $p$-adic Eisenstein congruences for the unitary group $\mathrm{U}(2,1)$.

Arguing as in  \cite[Cor.~5.9]{BDP}, Theorem~\ref{epsilon_deformaion_of_DDP} yields a geometrically flavored   proof of the rank 1 Gross--Stark conjecture over totally real fields. 
\begin{cor}\label{Gross--Stark-rk-1} Assume that $\# \Sigma_p^{\mathrm{irr}}=1$, and further assume that  $\sL(\phi) +\sL(\phi^{-1})  \ne 0$ when $d_{\gp} = d$. Then we have
\[   \zeta_{\phi}'(0)=- \frac{\mathscr{L}(\phi)}{\log_p(u)} L(0,\phi).  \]

In particular, $\mathrm{ord}_{X=0}\zeta_{\phi}(X)=\#  \Sigma_p^{\mathrm{irr}}$.
\end{cor}
\begin{proof}
We sketch the argument when $d_{\gp}=d$. A similar argument applies when $d_{\gp}<d$.
The idea  is to write $E_1(\mathbf{1},\phi)-f$ as the following linear combination 
\[
\left( \mathscr{L}(\phi^{-1}) + \mathscr{L}(\phi) \right) \left(E_1(\mathbf{1},\phi)-f\right)=E_{\mathbf{1},\phi}+ E_{\phi,\mathbf{1}}.
\]
of the generalized $p$-adic overconvergent weight one modular forms \[ E_{\mathbf{1},\phi}=   \frac{ (\mathscr{L}(\phi)+ \mathscr{L}(\phi^{-1}))}{\log_p(u)^{-1}\mathscr{L}(\phi)} \left.\tfrac{d}{dX}\right|_{X=0}(\cF-\cE_{\mathbf{1},\phi}) \text{ and } E_{\phi,\mathbf{1}}=  \frac{(\mathscr{L}(\phi)+ \mathscr{L}(\phi^{-1}))}{\log_p(u)^{-1}\mathscr{L}(\phi^{-1})} \left.\tfrac{d}{dX}\right|_{X=0}(\cF-\cE_{\phi,\mathbf{1}}).\] 
This is obtained by  comparing the Fourier coefficients of $\cF \mod X^2$ at non-zero ideals of  $\go$, which are fully determined in Theorem \ref{epsilon_deformaion_of_DDP}, with those of  $\cE_{\mathbf{1},\phi} \mod X^2$ and $\cE_{\mathbf{1},\phi} \mod X^2$, defined in \eqref{Eisensteinnonconstanterms}. It is crucial to observe that by Proposition~\ref{prop:hecke_duality} the same linear combination is automatically satisfied by the  constant terms at the infinity cusps as well. 
By computing them directly, we obtain 
\begin{align*}
 \sum_{[\gc]\in\Cl_F^+}\left( L(0,\phi)+\delta(\phi)\phi^{-1}(\gc\gd)L(0,\phi^{-1})\right)[\gc]=2^{d}
 \log_p(u) \cdot   \left.\tfrac{d}{dX}\right|_{X=0} C\left((0), - \frac{1}{\mathscr{L}(\phi)}\cE_{\mathbf{1},\phi}-\frac{1}{\mathscr{L}(\phi^{-1})} \cE_{\phi,\mathbf{1}}\right)
\\ = -\log_p(u)    \sum_{[\gc]\in\Cl_F^+} \left( \mathscr{L}(\phi)^{-1} \zeta_{\phi}'(0) + \delta(\phi)  \mathscr{L}(\phi^{-1})^{-1}   \zeta_{\phi^{-1}}'(0)  \phi^{-1}(\gc\gd) \right)  [\gc].
 \end{align*}
If $\phi$ is ramified, i.e.,  $\delta(\phi)=0$, the asserted formula follows immediately, whereas  
if $\phi$ is  a totally odd  character of $\Cl_F^+$, it  follows from the linear independence of characters. Finally, as $\mathscr{L}(\phi)\neq 0$ by Proposition~\ref{L-invariant}, it follows that  $\zeta_{\phi}(X)$ has a simple zero at $X=0$.  \end{proof}

\begin{rem}\label{rem:higher_rank}  Let  $E_1^{\mathrm{reg}}$ be the unique stabilization of $E_1(\mathbf{1},\phi)$ 
at all $\gq \in \Sigma_p\setminus \Sigma_p^{\mathrm{irr}} $ such that $U_{\gq}\cdot E_1^{\mathrm{reg}}= E_1^{\mathrm{reg}}$. Assume that $\Sigma_p^{\mathrm{irr}}$ contains at least two irregular primes, $\gp_1$ and $\gp_2$.
Let $E_1^{(\gp_1)}$ (resp. $E_1^{(\gp_2)}$) be the unique $\gp_1$-stabilization (resp. $\gp_2$-stabilization) of $E_1^{\mathrm{reg}}$.  Then, it follows from Proposition~\ref{constant_term_finite_wt}  that both of  $E_1^{(\gp_1)}$ and $E_1^{(\gp_2)}$ are $p$-adically cuspidal, distinct, and hence belong to the cuspidal generalized eigenspace $\mathbf{S}^\dagger_{1}\lsem f \rsem \subset \mathbf{M}^\dagger_{1}\lsem f \rsem $. In particular, $\dim_{\overline{\Q}_p} \mathbf{S}^\dagger_{1}\lsem f \rsem \geqslant 2$,  hence $\dim_{\overline{\Q}_p}  \cT^{\cusp}/ X  \cT^{\cusp} \geqslant 2$  by \eqref{eq:cuspidal-duality}. It follows that $\cT^{\cusp}$ is ramified over $\varLambda$.
\end{rem}

\subsection{An unramified Iwasawa module structure}\label{sec:unram_iawasawa_mod}
Recall that we set $H=\overline{F}^{\ker \phi}$. Let $H_{\infty}$ be the compositum of all $\Z_p$-extensions of $H$. Then, we have $\Gal(H_{\infty}/H)\simeq \Z_p^r$ for some $r\in \Z_{\geqslant 2}$.
We identify $\Z_p\lsem \Gal(H_{\infty}/H)\rsem$ with $\Z_p\lsem X_1\ldots,X_r\rsem$ via $\gamma_i\mapsto X_i+1$ for $i=1,\ldots,r$, where $\gamma_1,\ldots,\gamma_r$ are topological generators of $\Gal(H_{\infty}/H)$. Denote by $M_{\infty}$ the maximal abelian unramified $p$-extension of $H_{\infty}$. The Galois group $Y_{\infty}=\Gal(M_{\infty}/H_{\infty})$ has a natural action of $\Gal(H_{\infty}/H)$ given by $\gamma\cdot g=\tilde{\gamma} g \tilde{\gamma}^{-1}$ for all $\gamma\in \Gal(H_{\infty}/H)$ and $g\in Y_{\infty}$, the  action being independent of the choice of  lifting 
$\tilde{\gamma}\in Y_{\infty}$ of $\gamma$, as $Y_{\infty}$ is an abelian. Therefore, we can view $Y_{\infty}$ as a $\Z_p\lsem X_1,\ldots,X_r\rsem$-module. It is known \cite{greenberg} that $Y_{\infty}$ is a finitely generated torsion module over $\Z_p\lsem X_1,\ldots,X_r\rsem$.   The structure of this module has been extensively studied, and several deep conjectures (most notably Greenberg's generalized conjecture) have been proposed regarding its behavior.

The goal of this subsection is to construct an infinite extension  $L_{\infty}$ of $H_{\infty}$ which is contained in $M_{\infty}$ and such that 
the natural map $Y_{\infty}\to \Gal(L_{\infty}/H_{\infty})$ factors through the quotient by $(X_1,\ldots,X_r)$.

Recall   $\rho_{\mathbf{K} }=\left(\begin{smallmatrix} a & b \\ c & d \end{smallmatrix}\right)$ from  \eqref{ordinarycuspGMA} satisfying the  properties \ref{(P1)}, \ref{(P2)}, \ref{(P3)} and \ref{(P4)}  and  set $x(g,h)=b(g)c(h)$ for all $g,h\in \rG_F$.
The functions $a$, $d$ and $x$ satisfy the properties (I)--(IV) in \cite[pp.~563-564]{wiles}. For a $\overline{\Q}_p$-algebra homomorphism $\varphi_{\varepsilon}\in \Hom(\cT^{\cusp}, \overline{\Q}_p[\varepsilon])$, we write 
\[
\varphi_{\varepsilon}(a)=1+a'\varepsilon,\; \varphi_{\varepsilon}(d)=\phi(1+d'\varepsilon),\text{ and } \varphi_{\varepsilon}(x)=x'\varepsilon.
\]
In addition, for all $v\in \Sigma_p$, we write $\varphi_{\varepsilon}(\chi_v)=1+\chi'_v \varepsilon$, where $\chi_v:\rG_F\to \cT^{\nord}\twoheadrightarrow \cT^{\cusp}$ is the unramified character mapping $\Frob_v$ on $U_v$. One can verify by the above quoted properties (II)-(IV) that for all $g,h,l\in \rG_F$, the functions $a',d'$ and $x'$ satisfy 
\begin{enumerate}[label=\textbf{(T\arabic*)}]
 \item \label{(T1)} $a'(gh)=a'(g)+a'(h)+x'(g,h)$,
 \item \label{(T2)} $d'(gh)=d'(g)+d'(h)+x'(h,g)$,
\item \label{(T3)} $x'(gh,l)=x'(h,l)+\phi(h)x'(g,l)$,
\item \label{(T4)} $x'(g,hl)=x'(g,h)+\phi(h)x'(g,l)$, and
\item \label{(T5)} $a'(\gamma_0)=d'(\gamma_0)=x'(g,\gamma_0)=x'(\gamma_0,g)=a'(1)=d'(1)=x'(g,1)=x'(1,g)=0$.
\end{enumerate}
\begin{rem}\label{ordinary_basis_reg_prime}
For an arbitrary regular prime $\gq$, one can use Hensel's lemma to write $\rho_{\mathbf{K}}$ in  a basis where  $\rho_{\mathbf{K} }(\Frob_{\gq})=\left(\begin{smallmatrix} \alpha & 0 \\ 0 & \beta \end{smallmatrix}\right)$ with $\alpha,\beta \in \cT^{\cusp,\times}$ and $\alpha\neq \beta \mod \gm$. In this basis,  $\rho_{\mathbf{K} }$ satisfies properties similar to $\ref{(T1)}$--$\ref{(T5)}$ with $\gamma_0$ replaced by 
$\Frob_{\gq}$. Property \ref{(P3)} holds only for the irregular prime $\gp$, while property \ref{(P4)} still holds. \end{rem}

\begin{prop}\label{nonvanish_x'}
Assume $\Sigma_p^{\mathrm{irr}}=\{\gp\}$,  $\Sigma_p \setminus \Sigma_p^{\mathrm{irr}}=\{\gq \}$, $\delta_{F,p}=0$ and that 
\begin{align}\label{R2}
\text{there exists a surjective  $\overline{\Q}_p$-algebra homomorphism } \varphi_{\varepsilon}: \cT^{\cusp}/(X)\twoheadrightarrow \overline{\Q}_p[\varepsilon].
\end{align}
Then $x'\neq 0$ on $\rG_F\times \rG_F$. 
\end{prop}

Note that by Theorem~\ref{main-thm}(iv) the assumption \eqref{R2} is always satisfied when $2 \leqslant d_{\gp} \leqslant d-2$.

\begin{proof}
Suppose $x'=0$ on $\rG_F\times \rG_F$. Then $a',d'\in \Hom(\rG_F,\overline{\Q}_p)$. Since $\delta_{F,p}=0$, we can write $a'=\lambda\eta_{\mathbf{1}}$ and $d'=\mu \eta_{\mathbf{1}}$ for some $\mu,\lambda\in \overline{\Q}_p$. We shall show that this results in  $\mu=0=\lambda$ and $\chi'_{v}=0$ on $\rG_v$ for all $v\in \Sigma_p$, thereby contradicting  \eqref{R2}.
Since $\varphi_{\varepsilon}(\det (\rho_{\mathbf{K} }))=\phi$ by \eqref{R2}, we have $\lambda+\mu=0$. Moreover, for the regular prime $\gq \in  \Sigma_p \setminus \Sigma_p^{\mathrm{irr}}$, we have
\[
\varphi_{\varepsilon}(a+d)=\varphi_{\varepsilon}(\chi_{\gq}+\phi\chi_{\cyc}\chi_{\gq}^{-1})
\]
on $\rG_{\gq}$. It is straightforward to verify from the fact that $\phi(\gq)\ne 1$, and that $a$ (resp. 
$d$) lifts $\mathbf{1}$ (resp. $\phi$), that $a'=\chi'_{\gq}$ and $d'=-\chi_{\gq}$ on $\rG_{\gq}$, which yields $\mu=0=\lambda$ as $\chi'_{\gq}=0$ on $I_{\gq}$. Thus, $\chi'_{\gq}=0$ on $\rG_{\gq}$.

To show $\chi'_{\gp}=0$ on $\rG_{\gp}$, one can see from \eqref{the_module_B}  that $\phi^{-1}\varphi_{\varepsilon}(b)$ lies in  $\rH^1(F,\phi^{-1})$ and is unramified outside $\gp$. Moreover, \eqref{eq:p-nearly_ord} yields that $\varphi_{\varepsilon}(b)=0$ on $\rI_{\gp}$, since $d'=0$ and $\chi'_{\gp}=0$ on $\rI_{\gp}$. Thus, $\phi^{-1}\varphi_{\varepsilon}(b)$ is unramified everywhere and hence trivial. In particular, $\chi'_{\gp}=0$ on $\rG_{\gp}$ (again by \eqref{eq:p-nearly_ord}).
\end{proof}

 Let $g_0,l_0\in \rG_F$ be such that $x'(g_0,l_0)\neq 0$. Then, we can recover $\rho_{\mathbf{K} }$ from $x(g,h)$ by setting $b(g)=\tfrac{x(g,l_0)}{x(g_0,l_0)}$ and $c(g)=x(g_0,g)$ for all $g\in \rG_F$. Set \begin{align} \eta_{b}(g)=x'(g,l_0)  \text{ and }\eta_c(g)=x'(g_0,g) \text{ for all }\rG_F. \end{align} 
\begin{cor}\label{nonvani_x'_on_G_H}
Assume the hypotheses of Proposition~\ref{nonvanish_x'}. Then one has $x'\neq 0$ on $\rG_H\times \rG_H$. 
\end{cor}
\begin{proof}
One can verify from \ref{(T3)} and \ref{(T4)} that  $\phi^{-1}\eta_b\in \mathrm{Z}^1(F,\phi^{-1})$ and $\eta_c\in \mathrm{Z}^1(F,\phi)$. Furthermore, it follows from \ref{(T5)} that neither $\phi^{-1}\eta_b$ nor $\eta_c$ is a non-zero coboundary. Hence, by Proposition~\ref{nonvanish_x'}, the cohomology class $[\phi^{-1}\eta_b] \in \rH^1(F,\phi^{-1})$ is non-zero. 

In particular, the homomorphism $g \in G_H \mapsto x'(g,l_0)\neq 0$ is non-zero  by the inflation-restriction exact sequence  (here $l_0 \in \rG_F$). Therefore, one can choose $g_1\in \rG_H$ such that the cocycle $g \mapsto x'(g_1,g)$ is non-zero (because $x'(g_1,l_0) \ne 0$ with $(g_1,l_0) \in G_H \times \rG_F$). By the same argument, the homomorphism $h \in G_H \mapsto x'(g_1,h)$ is non-zero.
\end{proof}

Since $H_{\infty}$ is defined to be the compositum of all $\Z_p$-extensions of $H$, one has 
\begin{itemize}
\item $\eta_b=0=\eta_c$ on $\rG_{H_{\infty}}$, and
\item $d'|_{\rG_{H_{\infty}}}$ is a group homomorphism.
\end{itemize}
\begin{defn}Let $L_{\infty}$ denote the extension of $H_{\infty}$ cut out by $d|_{\rG_{H_{\infty}}}$.
\end{defn}
\begin{lemma}
Assume the hypotheses of Proposition~\ref{nonvanish_x'} and suppose that $\phi$ is non-quadratic. Then $d'\neq 0$ on $\rG_{H_{\infty}}$. In particular, the field $H_{\infty}$ is a proper subfield of $L_{\infty}$.
\end{lemma}
\begin{proof}
Suppose that $d'=0$ on $\rG_{H_{\infty}}$. Since the functions $\eta_b,\eta_c$, and $d'$  all vanish on  $\rG_{H_{\infty}}$, it follows from property $\ref{(T2)}$ that the map $d':\rG_H\to \overline{\Q}_p$ factors through the quotient $\Gal(H_{\infty}/H)$. As $\Gal(H_{\infty}/H)$ is an abelian group, we have $d'(gh)=d'(hg)$ for all $g,h\in \rG_H$. Therefore, by property $\ref{(T2)}$, we have 
\begin{align}\label{eq:comm_x'}
x'(g,h)=x'(h,g) \text{ for all } g,h\in \rG_H.
\end{align}
By Corollary~\ref{nonvani_x'_on_G_H}, we  choose an element $h_1\in \rG_H$ such that $x'\neq 0$ on $\rG_H\times \{h_1\}$. It then follows from \eqref{eq:comm_x'} and the inflation-restriction exact sequence that \[ x'(\cdot,h_1)=x'(h_1,\cdot)   \in \Hom(\rG_H,\overline{\Q}_p)[\phi]\cap \Hom(\rG_H,\overline{\Q}_p)[\phi^{-1}]\] which is trivial, since the character $\phi$ is not quadratic. This contradiction implies that the function $d'$ must be nontrivial on $\rG_{H_{\infty}}$.
\end{proof}
\begin{lemma}\label{prop_d'} For any $\gamma\in \Gal(H_{\infty}/H)$, we denote by $\tilde{\gamma}$ an arbitrary lift to $\Gal(L_{\infty}/H)$. 
Let the assumption be as in Proposition~\ref{nonvanish_x'}.  Then 
$d'(\tilde{\gamma}g\tilde{\gamma}^{-1})=d'(g)$
for all $\gamma\in \Gal(H_{\infty}/H)$ and $g\in \Gal(L_{\infty}/H_{\infty})$. In particular, $L_{\infty}$ is Galois over $H$.
\end{lemma}
\begin{proof}
The proof follows the same argument as that of \cite[Lem.~4.1]{bet19}.
\end{proof}
\begin{cor}\label{prop_L_infty}
Assume the hypotheses of Proposition~\ref{nonvanish_x'} and suppose that $\phi$ is non-quadratic. Then the  $L_{\infty}$ is unramified over $H_{\infty}$. In particular, we obtain a surjective homomorphism 
\[
Y_{\infty}/(X_1,\ldots,X_r)Y_{\infty}\twoheadrightarrow \Gal(L_{\infty}/H_{\infty}).
\]
\end{cor}
\begin{proof}
By Lemma~\ref{prop_d'}, to show that $L_{\infty}$ is unramified over $H_{\infty}$, it suffices to prove that $d'=0$ on $\rG_{H_{\infty}}\cap I_{F_v}$ for all $v\in \Sigma_p$. For an irregular prime $\gp$, we have  $\mathbf{B}\subset \gm$ by \ref{(P4)}. Since $\varphi_{\varepsilon}(b)=0$ on $\rG_{H_{\infty}}$ and $\chi'_{\gp}=0$ on $I_{F_{\gp}}$, we have $d'=0$ on $\rG_{H_{\infty}}\cap I_{F_{\gp}}$ by \eqref{eq:p-nearly_ord}. For the regular prime $\gq$, we write $\rho_{\mathbf{K} }$ in a GMA as in Remark~\ref{ordinary_basis_reg_prime}. We then have that $d'$ is $0$ on $\rI_{\gq}$, as we have seen in the proof of  Proposition~\ref{nonvanish_x'} that $-d'=a'=\chi'_{\gq}$ on $\rG_{\gq}$. Therefore, $L_{\infty}$ is unramified over $H_{\infty}$ at $p$. Moreover, the extension $L_{\infty}/H_{\infty}$ is unramified at $\gn$  by \cite[Prop.~4.4]{BDP}.  From this, we obtain a natural surjection $Y_{\infty}\twoheadrightarrow \Gal(L_{\infty}/H_{\infty})$ which factors through the quotient  $Y_{\infty}/(X_1,\ldots,X_r)Y_{\infty}$ by Lemma~\ref{prop_d'}.
\end{proof}

\subsection{GMAs and residual extensions}\label{sec:43}
Throughout this subsection, we assume that $d_\gp \geqslant d-1$, $\delta_{F,p}=0$, and that $\mathscr{L}(\phi)+\mathscr{L}(\phi^{-1}) \ne 0$ if $d_\gp=d$. Recall that we showed in \S\ref{sec:etale} that $\cT^{\nord}=\varLambda^{\nord}$ and there exists a unique, up to Galois conjugacy, nearly-ordinary Hida family $\cF^{\nord}$ containing $f$ whose Fourier coefficients are all in $\varLambda^{\nord}$. By Lemma~\ref{key_lemma} and \ref{(P4)}, we have a nearly-ordinary  deformation $\rho_{\cT^{\nord}}=\rho_{\mathbf{L}}:\rG_F\to \GL_2(\cT^{\nord})$  of $\left(\begin{smallmatrix} 1 & 0 \\ 0 & \phi \end{smallmatrix}\right)$, such that \eqref{eq:p-nearly_ord} holds.  We aim to show that 
\begin{align}\label{eq:embed_j}
\dim_{\overline{\Q}_p} \Hom(\mathbf{B}/{\gm}_{\nord} \mathbf{B},\overline{\Q}_p) =  d_\gp.
\end{align}

Recall that we have a decomposition $\Sigma=\coprod_{ \gq\in\Sigma_p} \Sigma_{\gq}$. To simplify the notation, we set $\eta_{\sigma}=\log_p\circ \sigma$ for all $\sigma\in \Sigma$. Notice that under this notation $\eta_{\mathbf{1}}=\sum_{\sigma\in \Sigma} \eta_{\sigma}$. Recall that $
\varLambda^{\nord}=\overline{\Q}_p\lsem X, \{X_{\sigma} \}_{ \sigma\in \Sigma} \rsem$ such that for all $ \gq\in\Sigma_p$, we have
\begin{align}\label{eq:univ_char}
\chi_{\gq}|_{\rI_{\gq}} \equiv 1+\sum_{\sigma\in \Sigma_{\gq}} \eta_{\sigma} \cdot X_{\sigma},\;  
\chi_{\cyc}\equiv 1- \frac{\eta_{\mathbf{1}}}{\log_p(u)} \cdot X \pmod{{\gm}_{\nord}^2}.
\end{align}
Here $\chi_{\gq}$ is the  character defined in \S\ref{sec:gma}.  Define a natural basis of $\Hom_{\overline{\Q}_p-\mathrm{alg}}(\cT^{\nord},\overline{\Q}_p[\varepsilon])$ by taking 
\begin{enumerate}
\item $\varphi_{\varepsilon,0}$, the morphism sending $X$ to $\varepsilon$ and each $X_{\sigma}$ to $0$  (i.e., $\varphi_{\cF,\varepsilon}$  from Theorem~\ref{epsilon_deformaion_of_DDP}), and
\item for $\sigma \in \Sigma$,  $\varphi_{\varepsilon,\sigma}$ is the element sending $\rho_{\cT^{\nord}}$ to $\left(\begin{smallmatrix} \ast & \ast \\ \ast & \phi \end{smallmatrix}\right)$ and $X_{\sigma'}$ to $\delta_{\sigma,\sigma'} \varepsilon$ for all $\sigma' \in \Sigma$. \end{enumerate}
The tangent vector $\varphi_{\varepsilon,\sigma}$ always exists, since we can consider the vector $\varphi_{\varepsilon,\sigma}'$  sending  $X$ to $\varepsilon$ and $X_{\sigma'} $ to $\delta_{\sigma,\sigma'} \varepsilon$, and then translate by a multiple of the tangent vector $\varphi_{\varepsilon,0}$ to ensure that the bottom-right entry of the image of $\rho_{\cT^{\nord}}$ is $\phi$ (as $\mu \ne 0 $ in  Theorem~\ref{epsilon_deformaion_of_DDP}), without affecting the value of $\varphi'_{\varepsilon,\sigma}(\chi_\gp)_{\mid \rI_{\gp}}$ (as $\varphi_{\varepsilon,0}(\chi_\gp)$ is unramified). Here, the image of  $\rho_{\cT^{\nord}}$ under any tangent vector has the form \eqref{formoftgvector}. 

For any weight character $\kappa=\left((2v_{\sigma})_{\sigma \in \Sigma},w\right) \in \Z_p^{d}\times \Z_p$, we let $\varphi_{\kappa,\varepsilon}= w \varphi_{\varepsilon,0} + \sum_{\sigma \in \Sigma} 2 v_{\sigma} \varphi_{\varepsilon,\sigma} $. The composition $\varphi_{\kappa,\varepsilon}\circ b$ gives rise to a cocycle $\eta_{\kappa}\in \rH^1(F,\phi^{-1})$ unramified outside $\gp$ as defined in \eqref{entryB}. Recall from Lemma~\ref{lemma-21} that every such cocycle is uniquely determined by its restriction to $\rI_{\gp}$.
\begin{prop}\label{cocycle_diff_direction}
Let the notation be as above. Assume that either the hypothesis of Theorem~\ref{main-thm}(ii) holds,  or the hypothesis of Theorem~\ref{main-thm}(iii) holds and $\mathscr{L}(\phi)+\mathscr{L}(\phi^{-1}) \ne 0$. Then, for any weight character $\kappa=\left((2v_{\sigma})_{\sigma},w\right)\in \Z_p^d\times \Z_p$, one has 
\[
\eta_{\kappa}|_{\rI_{\gp}}= w \mu\eta_{\mathbf{1} \mid \rI_\gp}- \sum_{\sigma\in \Sigma_{\gp}} 2v_{\sigma}\eta_{\sigma},
\] 
where $\mu \ne 0$ is defined in \eqref{eq:lambda_and_mu}. In particular,   \eqref{eq:embed_j} holds.
\end{prop}
\begin{proof}
If $v_{\sigma}=0$ for all $\sigma\in \Sigma$, the assertion follows immediately from \eqref{eq:p-nearly_ord} and Theorem~\ref{epsilon_deformaion_of_DDP}. For the remaining cases, we recall that \eqref{eq:p-nearly_ord} reads $b=d- \chi_{\gp}$ on $\rG_{\gp}$. Then the first part of the assertion follows from \eqref{eq:univ_char} and Theorem~\ref{epsilon_deformaion_of_DDP}. In particular, $\dim_{\overline{\Q}_p} \Hom(\mathbf{B}/{\gm}_{\nord} \mathbf{B},\overline{\Q}_p) \geqslant [F_\gp:\Q_p]$;
but \eqref{entryB} and Lemma~\ref{dimsplitcase} imply then \eqref{eq:embed_j}.
\end{proof}

\begin{rem}
Using the canonical basis of $\Hom(\rG_{\gp},\overline{\Q}_p)$ given by $\{\eta_{\sigma}\}_{\sigma \in \Sigma}$ and $\mathrm{ord}_{\gp}$, one can compute the $\ord_{\gp}$-coordinate of  $\eta_{\kappa \mid \rG_{\gp}} \in \Hom(\rG_{\gp},\overline{\Q}_p)$ in Proposition~\ref{cocycle_diff_direction}, and thereby relate it to a partial $\mathscr{L}$-invariant, as in the proof of Proposition~\ref{L-invariant} (see the discussion in \cite[\S 3.4]{DPV}); however, we will not pursue this here.
\end{rem}

\section*{Glossary of notation}
\footnotesize
\begin{multicols}{2}
\noindent 
$\mathbf{B}$ \dotfill $\cT^{\nord}$-submodule of $\mathbf{L}$ generated by $b(\rG_F)$, \S\ref{sec:24}\\
$\mathbf{C}$ \dotfill $\cT^{\nord}$-submodule of $\mathbf{L}$ generated by $c(\rG_F)$, \S\ref{sec:24}\\
$\cC, \cC_{\cusp}$ \dotfill Hilbert (cuspidal) eigencurve, \S\ref{sec:full_eigencurve}\\
$\gc$ \dotfill polarization ideal, \S\ref{sec:formal}\\
$d_{\gp}$ \dotfill the degree of $F_{\gp}$ over $\Q_p$\\ 
$\gd$ \dotfill different of $F$ \\
$\cE$ \dotfill  Hilbert cuspidal eigenvariety, \S\ref{sec:full_eigencurve} \\
$\cE_{\phi_1,\phi_2}$ \dotfill Eisenstein family associated to $\phi_1,\phi_2$, \S\ref{sec:eisenstein_series_def}\\
$E_{\gn}$ \dotfill totally positive units,  \S\ref{sec:formal}\\
$f=f_{\phi_1}$ \dotfill  $p$-stabilization of $E_1(\phi_1,\phi_2)$\\
$F$ \dotfill totally real field \\
$F_{\infty}$ \dotfill cyclotomic $\Z_p$-extension of $F$, \S\ref{sec:eisenstein_series_def}\\
$\cF$ \dotfill cuspidal family containing $f$, \S\ref{generic cuspidal deformation}\\
$\Frob_{\gp}$ \dotfill arithmetic Frobenius in $\rG_{\gp}$, \S\ref{sec:21}\\
$\rG_L$ \dotfill absolute Galois group of $L$\\
$\rG_v$ \dotfill decomposition subgroup of $\rG_F$ at $v$ \\
$h_0$\dotfill constant term Hecke operator, \S\ref{sec:123}\\ 
$H$ \dotfill splitting field of $\phi$, \S\ref{sec:21}\\
$H_{\infty}$ \dotfill  compositum of all $\Z_p$-extensions of $H$, \S\ref{sec:unram_iawasawa_mod} \\
$\rI_{\gp}$ \dotfill inertia subgroup of $\rG_{\gp}$\\
$\cI_{\phi_1,\phi_2}$ \dotfill  Eisenstein ideal for $\cE_{\phi_1,\phi_2}$, \S\ref{sec:123}\\
$K$ \dotfill sufficiently large finite extension of $\Q_p$, \S\ref{sec:formal}\\
$\mathbf{K}$ \dotfill the total quotient ring of $\cT^{\cusp}$, \S\ref{sec:04}\\
$\sL(\phi)$ \dotfill $\sL$-invariant, \S\ref{sec:21}\\
$\mathbf{L}$ \dotfill the total quotient ring of $\cT^{\nord}$, \S\ref{sec:24}\\
$L_{\infty}$ \dotfill   infinite  extension of $H_{\infty}$ contained in $M_{\infty}$, \S\ref{sec:unram_iawasawa_mod}\\
$M_{\infty}$ \dotfill max. abelian unramified pro-$p$-extension of $H_{\infty}$, \S\ref{sec:unram_iawasawa_mod}\\
 $\mathbf{M}^{\dagger}_{\cU}, \mathbf{M}^{\dagger}_{\cU}(v)$ \dotfill  $\cO(\cU)$-module of Coleman families, \S\ref{sec:twisted}\\ 
$\gm_{\nord}$ \dotfill the maximal ideal of $\cT^{\nord}$, \S\ref{sec:24}\\
$\gn$  \dotfill   the tame level \\ 
$\gn_i$ \dotfill the conductor of  $\phi_i$ ($i=1,2$)\\
$\go$ \dotfill ring of integers  of $F$ \\
$\cO$ \dotfill ring of integers  of $K$ \\
$\gp$ \dotfill  prime of $F$ above $p$  \\
$\cR^{\nord}_{\rho_{\univ}}$ \dotfill nearly-ordinary deform. ring of $\rho=\left(\begin{smallmatrix} \phi_1 & \phi_2 \eta \\ 0 & \phi_2 \end{smallmatrix}\right)$, \S\ref{sec:42}\\
 $\mathbf{S}^{\dagger}_{\cU}, \mathbf{S}^{\dagger}_{\cU}(v)$ \dotfill   cuspidal Coleman families, \S\ref{sec:twisted}\\ 
$T$ \dotfill the torus $\mathrm{Res}_{\go/\Z}\mathbf{G}_m$, \S\ref{sec:formal}\\
$\cT^{\cusp}$  \dotfill  completed local ring  of $\cC_{\cusp}$  at $f$, \S\ref{sec:gma}  \\
$\cT^{\nord}$  \dotfill  completed local ring  of $\cE$  at $f$, \S\ref{sec:gma}  \\
$\mathscr{T}_{n},\mathscr{T}_{n}^\circ$ \dotfill formal subgroups of $\mathscr{T}$, \S\ref{AIP-OCMS}\\
$u$ \dotfill fixed element in $1+p\Z_p$, \S\ref{sec:eisenstein_series_def}\\
$u_{\phi}$ \dotfill $\phi$-isotypic $\gp$-unit in $\cO_H[1/\gp]^\times\otimes \overline{\Q}_p$, \S\ref{sec:21}\\
$\cW$ \dotfill  $d$-dimensional weight space, \S\ref{AIP-OCMS}\\
$\cW^{\parallel}$ \dotfill  $1$-dimensional parallel weight space, \S\ref{sec:twisted}\\
$w_0$ \dotfill place of $H$ above $\gp$ induced by $\iota_p$, \S\ref{sec:21}\\
$\mathrm{w}$ \dotfill  weight map to $\cW^{\parallel}$, \S\ref{sec:full_eigencurve}\\ 
$X$ \dotfill uniformizer of $\varLambda$, \S\ref{sec:gma}\\
$y_v$ \dotfill slope of the nearly-ordinary line at $v$, \S\ref{sec:gma}\\
$Y_{\infty}$ \dotfill the Galois group of $M_{\infty}$ over $H_{\infty}$, \S\ref{sec:unram_iawasawa_mod}\\
$\gamma$ \dotfill fixed element in $\Gal(F_{\infty}/F)$, \S\ref{sec:eisenstein_series_def}\\
$\delta_{F,p}$ \dotfill  Leopoldt defect of $F$ at $p$\\
$\zeta_{\phi}, \zeta_{\phi_1,\phi_2}$ \dotfill  Deligne--Ribet's $p$-adic zeta, \S\ref{sec:eisenstein_series_def}\\
$[\eta], \eta$ \dotfill element of  $\rH^1(F, \phi^{-1})$, \S\ref{sec:21}\\
$\eta_{\phi}^{(\tilde{\sigma})}$ \dotfill canonical basis of $\rH^1(F, \phi)$,  \S\ref{sec:21}\\
$[\eta_{\phi}^{(\gp)}]$ \dotfill  canonical element of $\rH^1(F, \phi)$, \S\ref{sec:21}\\
$\iota_p$ \dotfill  embedding of $\overline{\Q}$ in  $\overline{\Q}_p$\\
$\kappa$ \dotfill  weight map to $\cW$, \S\ref{sec:full_eigencurve}\\ 
$\varLambda_{\cO}$ \dotfill  Iwasawa algebra $\cO\lsem X \rsem$, \S\ref{sec:eisenstein_series_def}\\
$\varLambda$ \dotfill  completed local ring of $\cW^{\parallel}$ at $\sw(f)$, \S\ref{sec:gma}\\  
$\varLambda^{\nord}$ \dotfill  completed local ring of $\cW\times\cW^{\parallel}$ at $\mathrm{w}(f)$, \S\ref{sec:gma}\\
$\rho_{\mathbf{L}}$ \dotfill nearly ordinary Galois representation, \S\ref{sec:24}\\
$\Sigma$ \dotfill the set of embeddings of $F$ into $\overline{\Q}$, \S\ref{sec:21}\\
$\Sigma_p$ \dotfill the set of primes of $F$ above $p$\\
$\Sigma_p^{\mathrm{irr}}$ \dotfill the set of irregular primes\\
$\Sigma_{\gp}$ \dotfill  subset of $\Sigma$ corresponding to   $\gp\in \Sigma_p$ \\
$\tau$ \dotfill  complex conjugation in $\rG_F$, \S\ref{sec:full_eigencurve}\\
$\phi_1$, $\phi_2$, $\phi=\phi_2\phi_1^{-1}$  \dotfill finite order Hecke characters of $F$ \\
$\phi_p$ \dotfill restriction of $\phi$ to $(\go\otimes \Z_p)^\times$ \\
$\chi_{\cyc}$ \dotfill universal cyclotomic character, \S\ref{sec:full_eigencurve}\\ 
$\chi_v$ \dotfill  nearly-ordinary character, \S\ref{sec:full_eigencurve}
\end{multicols}

\bibliographystyle{siam}

\end{document}